%% file: parabolicdoublecosets-arxivfinal.tex
\documentclass[12pt]{amsart}

\usepackage{tikz,tkz-graph}
\usetikzlibrary{snakes,calc,shapes.misc}
\usepackage{amssymb,amsmath,amsthm}

\usepackage{url,color,enumerate}

\usepackage[margin=1in,letterpaper,portrait]{geometry}

\usepackage[colorlinks=true,citecolor=black,linkcolor=black,urlcolor=blue]{hyperref}

\definecolor{541}{rgb}{1,0.8,0.2}
\definecolor{530}{rgb}{1,0.6,0.0}
\definecolor{light-blue}{rgb}{0.6,0.65,1}
\definecolor{204}{rgb}{0.4,0.0,0.8}
\definecolor{purple}{rgb}{0.4,0.0,0.8}
\definecolor{035}{rgb}{0,0.6,1}
\definecolor{005}{rgb}{0,0.0,1}
\definecolor{red}{rgb}{1,0,0}

\definecolor{mediumelectricblue}{rgb}{0.01, 0.31, 0.59}
\definecolor{jonquil}{rgb}{0.98, 0.85, 0.37}
\definecolor{airforceblue}{rgb}{0.36, 0.54, 0.66}
\definecolor{amber}{rgb}{1.0, 0.75, 0.0}
\definecolor{navyblue}{rgb}{0.0, 0.0, 0.5}
\definecolor{darkcerulean}{rgb}{0.03, 0.27, 0.49}
\definecolor{cornflowerblue}{rgb}{0.39, 0.58, 0.93}
\definecolor{corn}{rgb}{0.98, 0.93, 0.36}

\definecolor{myblue}{named}{cornflowerblue}
\definecolor{myyellow}{named}{corn}

\theoremstyle{plain}
\newtheorem{thm}{Theorem}[section]
\newtheorem{lem}[thm]{Lemma}

\newtheorem{cor}[thm]{Corollary}
\newtheorem{conj}[thm]{Conjecture}

\newtheorem{prop}[thm]{Proposition}
\theoremstyle{definition}
\newtheorem{example}[thm]{Example}
\newtheorem{defn}[thm]{Definition}
\newtheorem{question}[thm]{Question}

\theoremstyle{remark}
\newtheorem{remark}[thm]{Remark}

\newcommand{\given}{\, | \,}
\newcommand{\p}[1]{\mathcal #1}

\newcommand{\x}{{\bf x}}
\newcommand{\N}{{\mathbb N}}

\DeclareMathOperator{\asc}{Asc}
\DeclareMathOperator{\tethers}{Tethers}
\DeclareMathOperator{\wharfs}{Wharfs}
\DeclareMathOperator{\rafts}{Rafts}
\DeclareMathOperator{\floats}{Floats}
\DeclareMathOperator{\ropes}{Ropes}

\DeclareMathOperator{\branch}{branch}
\DeclareMathOperator{\Des}{Des}
\DeclareMathOperator{\des}{Des}
\DeclareMathOperator{\inv}{inv}
\newcommand{\Cmin}{{w_{C\text{-min}}}}
\newcommand{\Cmax}{{w_{C\text{-max}}}}

\newcommand\EE{{\mathbb{E}}}

\newcommand{\mc}[1]{\mathcal #1}

\newtheorem*{repp@ex}{\repp@title (continued)}
\newcommand{\newreppex}[2]{
\newenvironment{repp#1}[1]{
 \def\repp@title{#2~\ref{##1}}
 \begin{repp@ex}}
 {\end{repp@ex}}}
\makeatother
\newreppex{ex}{Example}

\newtheorem*{repp@thm}{\repp@title}
\newcommand{\newreppthm}[2]{
\newenvironment{repp#1}[1]{
 \def\repp@title{#2~\ref{##1}}
 \begin{repp@thm}}
 {\end{repp@thm}}}
\makeatother
\newreppthm{thm}{Theorem}

\newcommand{\naunion}{\overset{\not\sim}{\sqcup}}

\newcommand{\ri}[2]{{\overset{\mathbf{R#1}}{\scriptstyle #2}}}

\newcommand\Tstrut{\rule{0pt}{3.5ex}}

\tikzset{cross/.style={cross out, draw=black, minimum size=2*(#1-\pgflinewidth), inner sep=0pt, outer sep=0pt},cross/.default={1pt}}

\newcommand{\oo}{\raisebox{-0.45ex}{{\begin{tikzpicture} \draw (-.125,-.125) rectangle (.125,.325); \draw (0,0.2) circle (0.05);\draw (0,0) circle (0.05); \end{tikzpicture}}}}
\newcommand{\ox}{\raisebox{-0.45ex}{\begin{tikzpicture} \draw (-.125,-.125) rectangle (.125,.325); \draw (0,0.2) circle (0.05);\draw[fill] (0,0) circle (0.05); \end{tikzpicture}}}
\newcommand{\xo}{\raisebox{-0.45ex}{\begin{tikzpicture} \draw (-.125,-.125) rectangle (.125,.325); \draw[fill] (0,0.2) circle (0.05);\draw (0,0) circle (0.05); \end{tikzpicture}}}

\newcommand{\oX}{\raisebox{-0.45ex}{\begin{tikzpicture} \draw (-.125,-.125) rectangle (.125,.325); \draw (0,0.2) -- (0,0.05); \fill[white] (0,0.2) circle (0.05); \draw (0,0.2) circle (0.05); \draw[fill] (0,0) circle (0.04); \draw (0,0) circle (0.075); \end{tikzpicture}}}
\newcommand{\Ox}{\raisebox{-0.45ex}{\begin{tikzpicture} \draw (-.125,-.125) rectangle (.125,.325); \draw (0,0.2) -- (0,0.05); \fill[white] (0,0.2) circle (0.04); \draw (0,0.2) circle (0.04); \draw[fill] (0,0) circle (0.05); \draw (0,0.2) circle (0.075); \end{tikzpicture}}}
\newcommand{\xx}{\raisebox{-0.45ex}{\begin{tikzpicture} \draw
(-.125,-.125) rectangle (.125,.325);   \draw[fill] (0,0.2) circle
(0.05);\draw[fill] (0,0) circle (0.05); \end{tikzpicture}}}
\newcommand{\axx}{\raisebox{-0.45ex}{\begin{tikzpicture} \draw[rounded corners=3pt]
(-.125,-.125) rectangle (.125,.325); \draw[fill] (0,0.2) circle
(0.05);\draw[fill] (0,0) circle (0.05); \end{tikzpicture}}}
\newcommand{\aox}{\raisebox{-0.45ex}{\begin{tikzpicture} \draw[rounded corners=3pt]
(-.125,-.125) rectangle (.125,.325); \draw (0,0.2) circle
(0.05);\draw[fill] (0,0) circle (0.05); \end{tikzpicture}}}
\newcommand{\Xx}{\raisebox{-0.45ex}{\begin{tikzpicture} \draw (-.125,-.125) rectangle (.125,.325); \draw[fill] (0,0.2) circle (0.04); \draw (0,0.2) -- (0,0.05); \draw[fill] (0,0) circle (0.05); \draw (0,0.2) circle (0.075); \end{tikzpicture}}}
\newcommand{\xX}{\raisebox{-0.45ex}{\begin{tikzpicture} \draw (-.125,-.125) rectangle (.125,.325); \draw[fill] (0,0.2) circle (0.05); \draw (0,0.2) -- (0,0.05); \draw[fill] (0,0) circle (0.04); \draw (0,0) circle (0.075); \end{tikzpicture}}}
\newcommand{\XX}{\raisebox{-0.45ex}{\begin{tikzpicture} \draw (-.125,-.125) rectangle (.125,.325); \draw[fill] (0,0.2) circle (0.05); \draw (0,0.2) -- (0,0.05); \draw[fill] (0,0) circle (0.05); \end{tikzpicture}}}

\newcommand{\omo}{\raisebox{-0.45ex}{\begin{tikzpicture} \draw (-.125,-.125) rectangle (.125,.325);  \draw (0,0.2) circle (0.05);\draw (0,0) circle (0.05); \draw (0,0.15) -- (0,0.05); \end{tikzpicture}}}
\newcommand{\omx}{\raisebox{-0.45ex}{\begin{tikzpicture} \draw (-.125,-.125) rectangle (.125,.325);  \draw (0,0.2) circle (0.05);\draw[fill] (0,0) circle (0.05); \draw (0,0.15) -- (0,0.05); \end{tikzpicture}}}
\newcommand{\xmo}{\raisebox{-0.45ex}{\begin{tikzpicture} \draw (-.125,-.125) rectangle (.125,.325); \draw[fill] (0,0.2) circle (0.05);\draw (0,0) circle (0.05); \draw (0,0.15) -- (0,0.05); \end{tikzpicture}}}
\newcommand{\xmx}{\raisebox{-0.45ex}{\begin{tikzpicture} \draw (-.125,-.125) rectangle (.125,.325); \draw[fill] (0,0.2) circle (0.05);\draw[fill] (0,0) circle (0.05); \draw (0,0.15) -- (0,0.05); \end{tikzpicture}}}
\newcommand{\nodot}{\raisebox{.2ex}{\begin{tikzpicture}
\draw (0,0) circle (0.05); \end{tikzpicture}}}
\newcommand{\yesdot}{\raisebox{.2ex}{\begin{tikzpicture} \draw[fill] (0,0) circle (0.05); \end{tikzpicture}}}
\newcommand{\bothdot}{\raisebox{.1ex}{\begin{tikzpicture} \draw[fill] (0,0) circle (0.04); \draw(0,0) circle (0.075); \end{tikzpicture}}}

\newcommand{\bo}{\raisebox{-0.45ex}{{\begin{tikzpicture} \draw (-.125,-.125) rectangle (.125,.325); \draw[fill] (0,0.2) circle (0.04); \draw(0,0.2) circle (0.075); \draw (0,0) circle (0.05); \end{tikzpicture}}}}
\newcommand{\ob}{\raisebox{-0.45ex}{{\begin{tikzpicture} \draw (-.125,-.125) rectangle (.125,.325); \draw (0,0.2) circle (0.05);\draw(0,0) circle (0.075); \draw[fill] (0,0) circle (0.04); \end{tikzpicture}}}}
\newcommand{\bx}{\raisebox{-0.45ex}{{\begin{tikzpicture} \draw (-.125,-.125) rectangle (.125,.325); \draw[fill] (0,0.2) circle (0.04); \draw(0,0.2) circle (0.075); \draw[fill] (0,0) circle (0.05); \end{tikzpicture}}}}
\newcommand{\xb}{\raisebox{-0.45ex}{{\begin{tikzpicture} \draw (-.125,-.125) rectangle (.125,.325); \draw[fill] (0,0.2) circle (0.05);\draw(0,0) circle (0.075);\draw[fill] (0,0) circle (0.04); \end{tikzpicture}}}}
\newcommand{\bb}{\raisebox{-0.45ex}{{\begin{tikzpicture} \draw (-.125,-.125) rectangle (.125,.325); \draw[fill] (0,0.2) circle (0.04);\draw(0,0.2) circle (0.075); \draw(0,0) circle (0.075);\draw[fill] (0,0) circle (0.04); \end{tikzpicture}}}}

\author[Billey]{Sara C.  Billey}%
\address{Department of Mathematics, University of Washington,
Seattle, WA, USA}
\email{billey@math.washington.edu}%

\author[Konvalinka]{Matja\v z Konvalinka}%
\address{Department of Mathematics, 
Faculty for Mathematics and Physics,
University of Ljubljana,
Institute for Mathematics, Physics and Mechanics, 
Ljubljana, Slovenia}
\email{matjaz.konvalinka@fmf.uni-lj.si}%

\author[Petersen]{T. Kyle Petersen}%
\address{Department of Mathematical Sciences, DePaul University, Chicago, IL, USA}
\email{tpeter21@depaul.edu}%

\author[Slofstra]{William Slofstra}%
\address{
Institute for Quantum Computing, 
Department of Pure Mathematics, 
University of Waterloo, 
Waterloo, Ontario, Canada}
\email{weslofst@uwaterloo.ca}%

\author[Tenner]{Bridget E. Tenner}%
\address{Department of Mathematical Sciences, DePaul University, Chicago, IL, USA}
\email{bridget@math.depaul.edu}%

\thanks{SB was partially supported by the National Science Foundation grant DMS-1101017, MK was supported by Research Program Z1-5434 and Research Project BI-US/14-15-026 of the Slovenian Research Agency, TKP and BET were partially supported by Simons Foundation Collaboration Grants for Mathematicians, and BET was partially supported by a DePaul University Faculty Summer Research Grant.}

\title[Parabolic double cosets]{Parabolic double cosets in Coxeter groups}


\begin{document}

\begin{abstract}
Parabolic subgroups $W_I$ of Coxeter systems $(W,S)$, as well as their ordinary and
double quotients $W / W_I$ and $W_I \backslash W / W_J$, appear in many contexts in
combinatorics and Lie theory, including the geometry and topology of
generalized flag varieties and the symmetry groups of regular polytopes.  The
set of ordinary cosets $w W_I$, for $I \subseteq S$, forms the Coxeter complex of $W$,
and is well-studied. In this article we look at a less studied object: the set of
all double cosets $W_I w W_J$ for $I, J \subseteq S$.
Double cosets are not uniquely presented by triples
$(I,w,J)$.  We describe what we call the lex-minimal presentation, and
prove that there exists a unique such object for each double coset.
Lex-minimal presentations are then used to enumerate double cosets via a finite automaton depending on the Coxeter graph for $(W,S)$.
As an example, we present a formula for the number of parabolic double
cosets with a fixed minimal element when $W$ is the symmetric group
$S_n$ (in this case, parabolic subgroups are also known as Young
subgroups).  Our formula is almost always linear time computable in
$n$, and we show how it can be generalized to any Coxeter
group with little additional work.  We spell out formulas for all finite and
affine Weyl groups in the case that $w$ is the identity element.
\end{abstract}

\maketitle

\tableofcontents

\section{Introduction}

Let $G$ be a group with subgroups $H$ and $K$. The group $G$ is partitioned by the \emph{double quotient}
$$H\backslash G/K=\{H g K \given g \in G\},$$
i.e., the collection of \emph{double cosets} $HgK$. If $G$ is finite, then the number of double cosets in $H\backslash G/K$ is the inner product of the characters of the two trivial representations on $H$ and $K$ respectively, induced up to $G$ \cite[Exercise 7.77a]{ec2}.  Double cosets are usually more complicated than one-sided cosets.  For instance, unlike one-sided cosets, two double cosets need not have the same size.

In this article, we investigate the parabolic double cosets of a finitely generated Coxeter group. That is, given a Coxeter system $(W,S)$ of finite rank $|S|$, we consider cosets
$$W_{I} w W_{J},$$
where $I$ and $J$ are subsets of the generating set $S$, and
$$W_I=\langle s : s \in I\rangle$$
denotes the standard \emph{parabolic subgroup} of $W$ generated by the subset $I$. These cosets are elements of the double quotient $W_I\backslash W/W_J$, though a given coset can be a member of more than one such quotient. Throughout this paper, we will compare elements of $W$ using strong Bruhat order, and indicate these comparisons using ``$\le$.''

The parabolic double cosets are natural objects of study in many contexts. For example, they play a prominent role in the paper of Solomon that first defines the descent algebra of a Coxeter group \cite{Solomon}. For finite Coxeter groups, Kobayashi showed that these double cosets are intervals in Bruhat order, and these intervals have a rank-symmetric generating function with respect to length \cite{kobayashi}. (While rank-symmetric, there exist parabolic double cosets in $S_5$ that are not self-dual.) Geometrically, these intervals correspond to the cell decomposition of certain rationally smooth Richardson varieties.

If we fix $I$ and $J$, then the structure of the double quotient $$W_I \backslash W / W_J$$ is also well-studied. For example, Stanley \cite{Stanley.1980} shows the Bruhat order on such a double quotient is strongly Sperner (for finite $W$), and Stembridge \cite{Stembridge.2004} has characterized when the natural root coordinates corresponding to elements in the quotient give an order embedding of the Bruhat order (for any finitely generated $W$). As cited above, the number of elements in the double quotient is a product of characters,
\begin{equation}\label{eq:fixedIJ}
 \left|W_{I}\backslash W / W_{J}\right|
= \left\langle \mathrm{ind}_{W_{I}}^{W}1_{W_{I}}, \mathrm{ind}_{W_{J}}^{W}1_{W_{J}} \right\rangle,
\end{equation}
where $1_{W_J}$ denotes the trivial character on $W_J$.  For fixed $I$ and varying $J$, Garsia and Stanton connect parabolic double cosets to basic sets for the Stanley-Reisner rings of Coxeter complexes \cite{garsia-stanton.1984}.

In this paper, we are interested in a basic problem about parabolic double
cosets that appears to have been unexamined until now: how many distinct double
cosets $W_I w W_J$ does $W$ have as $I$ and $J$ range across subsets of $S$?
This question is partly motivated by the analogous problem for ordinary cosets,
where the set $\{w W_I : I \subseteq S, w \in W\}$ is equal to the set of cells of the
Coxeter complex. When $W$ is the symmetric group $S_n$, the number
of such cells is the $n$th ordered Bell number \cite[A000670]{oeis}. One fact
that makes the one-sided case substantially simpler than the two-sided version
is that each ordinary parabolic coset has the form $w W_I$ for a unique subset
$I \subseteq S$. If we take $w$ to be the minimal element in the coset, then
the choice of $w$ is also unique.  While double cosets do have unique minimal
elements, different pairs of sets $I$ and $J$ often give the same double coset.
For example, if $e$ denotes the identity element of $W = (W,S)$, we have $W =
W_S e W_J = W_I e W_S$ for \emph{any} subsets $I$ and $J$ of $S$. Thus we
cannot count distinct double parabolic cosets simply by summing
Equation~\eqref{eq:fixedIJ} over all $I$ and $J$.

As mentioned earlier, every double coset has a unique minimal element, and we use this fact to recast our motivating question: for any $w \in W$, how many distinct double cosets does $W$ have with minimal element $w \in W$?

\begin{defn}
Let $(W,S)$ be a Coxeter system of finite rank $|S|$, and fix an element $w \in W$. Set
$$c_w := \# \{\text{double cosets with minimal element } w\}.$$
\end{defn}

Because $W$ has finite rank, $c_w$ is always finite. Our first result
 is a formula for $c_w$ when $W = S_n$, based on what we call the \emph{marine model}, introduced in Section~\ref{sec:marinemodelforSn}.

\begin{thm}\label{main}
There is a finite family of sequences of positive integers $b^K_m$, $m \geq 0$,
such that the number of parabolic double cosets with minimal element $w \in
S_n$ is
\[
 c_w = 2^{|\floats(w)|} \sum_{T \subseteq \tethers(w) } \prod_{\p R \in \rafts(w)} b_{|\p R|}^{K(\p R,T)}.
\]
\end{thm}

The sets $\floats(w)$, $\tethers(w)$, and $\rafts(w)$ are subsets of the
left and right ascent sets of $w$, and will be defined precisely in
Section~\ref{sec:symmetric group}, as will the sequences $b^K_m$.
Given a particular $w \in S_n$, Theorem~\ref{main} enables
fast computations for two reasons.  First, the sequences
$b_m^K$ satisfy a linear recurrence, and thus can be
easily computed in time linear in $m$.  Second, tethers are rare,
and hence the sum typically has only one term, which corresponds to the
empty set. In fact, the expected number of tethers in $S_n$ is
approximately $1/n$.

\begin{prop}\label{prop:expected}
For all $n \geq 1$, the expected value of $|\tethers(w)|$, over all $w
\in S_n$ chosen uniformly, is given by
\[
 \frac{1}{n!}\sum_{w \in S_n} |\tethers(w)| = \frac{(n-3)(n-4)}{n(n-1)(n-2)}.
\]
\end{prop}

Theorem~\ref{main} allows us to calculate
\begin{equation*}
    p_n = \sum_{w \in S_n} c_w,
\end{equation*}
the total number of distinct double cosets in $S_n$. Although this requires summing
$n!$ terms, the approach seems to be a significant improvement over what
was previously known. The initial terms of the sequence $\{p_n\}$ are
\begin{eqnarray}\label{eqn:total in S_n}
&1, 3, 19, 167, 1791, 22715, 334031, 5597524, 105351108,2200768698,&\\
\nonumber &50533675542, 1265155704413, 34300156146805.&
\end{eqnarray}
The sequence $\{p_n\}$ did not appear in the OEIS before our work, and it can now be found at \cite[A260700]{oeis}. We also make the following conjecture.

\begin{conj}
 There exists a constant $K$ so that
 $$\frac{p_n}{n!} \sim \frac{K}{\log^{2n} 2}.$$
\end{conj}

From the enumeration for $n \leq 13$, we observe that the constant $K$
seems to be close to $0.4$.

As predicted, summing Equation~\eqref{eq:fixedIJ} over $I$ and $J$ overcounts the double cosets. In particular, this would produce
\begin{equation*}
    1, 5, 33, 281, 2961, 37277, \ldots,
\end{equation*}
which counts ``two-way contingency tables'' (\cite[A120733]{oeis}, \cite[Exercise~7.77]{ec2}, \cite{Diaconis-Gangoli-1994}, \cite[Section 5]{DHT}). This sequence also enumerates cells in a two-sided analogue of the Coxeter complex recently studied by the third author \cite{Pe15}.

The key to proving the formula in Theorem~\ref{main} is a condition on
pairs of sets $(I,J)$ guaranteeing that each double coset $W_I w W_J$
arises exactly once. In other words, we identify a canonical
presentation $W_I w W_J$ for each double coset with minimal element
$w$.  The canonical presentation we found most useful for enumeration
is what we call the \emph{lex-minimal presentation}.  This
presentation can be easily defined for any Coxeter group.

\begin{defn}[Lex-minimal presentation]\label{defn:lex-minimal}
    Let $C$ be a parabolic double coset of some Coxeter system $(W,S)$. A
    \emph{presentation} of $C$ is a choice of $I, J \subseteq S$ and $w \in W$
    such that $C = W_I w W_J$. A presentation is \emph{lex-minimal} if
    $w$ is the minimal element of $C$, and $(|I|,|J|)$ is lexicographically minimal
    among all presentations of $C$. When $w$ is fixed, we will abuse terminology slightly to call $(I,J)$ a \emph{lex-minimal pair} for $w$ if $W_IwW_J$ is a lex-minimal presentation.
\end{defn}

In other words, if $C = W_I w W_J$ is lex-minimal, and $C = W_{I'} w' W_{J'}$ is another presentation, then $w \leq w'$. Furthermore, either $|I| < |I'|$, or $|I| = |I'|$ and $|J| \leq |J'|$.

In Section~\ref{sec:lex-min in symmetric group} and Section~\ref{sec:general groups}, we show that every parabolic
double coset has a unique lex-minimal presentation (the former section treats only the symmetric group, while the latter studies general Coxeter groups). Thus we focus our
attention on counting the lex-minimal presentations.  The main point
of this paper is to show that there exists a finite state automaton that encodes the lex-minimal conditions along rafts as allowable words in the automaton. Moreover, \emph{this same automaton can be used for all Coxeter systems}.

By the transfer matrix method (see \cite[Section 4.7]{ec1}), the number of allowable words of a given length in a language given by such an automaton has a rational generating function. Hence the sequences we use to count lex-minimal presentations satisfy finite linear recurrence relations. (In fact, we will see that there are only four different recurrences, the longest of which has six terms.) As in Theorem~\ref{main}, this allows us to compute $c_w$ more efficiently, and we
get an enumerative formula that generalizes Theorem~\ref{main} to any Coxeter group.

\begin{thm}\label{main.general.case}
There exists a finite family of sequences of positive integers
$b_m^{K}$ for $m \geq 0$, each determined by a
linear homogeneous constant-coefficient recurrence relation,  such that
for any Coxeter group $W$ and any $w \in W$, the number of parabolic
double cosets with minimal element $w$ is
\[
 c_w = 2^{|\floats(w)|} \sum_{T \subseteq \tethers(w) } \sum_{U \subseteq  \wharfs(w) }
  \prod_{\p R \in \rafts(w)} b_{|\p R|}^{K(\p R, T, U)}.
\]
\end{thm}

Here $\wharfs(w)$ is another subset of ascents of $w$, defined in
Section~\ref{sec:general groups}.  As evidence that the method is effective, in
Section~\ref{sec:general groups} we give formulae for $c_w$ when $w=e$ is the
identity element in each irreducible Weyl group of finite and affine
type.

As a byproduct of the proofs of Theorems~\ref{main} and~\ref{main.general.case}, we
introduce the $w$-ocean graph for each $w \in W$.  This graph encodes
all presentations $(I,w,J)$ for parabolic double cosets with minimal
element $w$.  This graph is independent of the edge labels of the Coxeter graph, as discussed in Remark~\ref{rem:oceans and edge labels}. Thus, for example, Theorem~\ref{main} applies exactly to type $B_n$.

In Theorem~\ref{thm:equivalence.moves}, we show that
$W_I w W_J = W_{I'} w W_{J'}$ if and only if $(I,J)$ can be obtained
from $(I',J')$ by plank moves which are defined in terms of moving
connected components on the $w$-ocean.  The lex-minimal presentations
are also characterized in terms of plank moves.

The paper is organized as follows. In Section~\ref{sec:background} we
give an overview of parabolic double cosets for Coxeter groups. The
enumeration and the marine model for $S_n$ is described in
Section~\ref{sec:symmetric group}.  In Section~\ref{sec:general
  groups}, we develop the marine model including the $w$-ocean in the general Coxeter group setting.

We finish with Section~\ref{sec:restricted}, which describes a way to think
about enumeration of cosets with minimal element $w \in W$ as equivalent to
enumeration of cosets with minimal element equal to the identity in a larger Coxeter group. The trade-off is that we may have to restrict our allowable set
of reflections, which poses some interesting questions. 

\section{Background} \label{sec:background}

In this section, we give definitions and relevant background information on the objects of interest in this paper. We begin with a discussion of Coxeter groups, and focus our attention first on parabolic cosets, and then on parabolic double cosets of these groups.

Coxeter groups are a broad family containing the symmetric groups,
as well as Weyl groups of Kac-Moody groups and all finite real reflection groups.

\begin{defn}
A \emph{Coxeter group} is a group $W$ with a presentation
\begin{equation*}
    \langle s \in S : s^2 = e \text{ for all } s \in S, \text{ and }
        (st)^{m(s,t)} = e \text{ for all } s,t \in S \rangle,
\end{equation*}
where $e$ is the identity element, and $m(s,t) = m(t,s)$ are values in the set $\{2,3,\ldots\} \cup \{+\infty\}$. (If $m(s,t) = +\infty$, then the corresponding relation is omitted.) The elements of $S$ are the
\emph{simple reflections} of $W$. The \emph{reflections} of $W$ are the
conjugates of elements of $S$. The set of reflections is usually denoted by $T$, so $T
= \{ wsw^{-1}: s \in S, w \in W\}$.
\end{defn}

In this paper, we study \emph{finitely generated} Coxeter groups. That is, we assume that the set $S$ is finite. In this case, $|S|$ is called the \emph{rank} of the group.

A Coxeter group $W$
can have more than one presentation. However, once a set $S$ of simple
reflections is chosen, the presentation is unique.  Consequently,
a pair $(W,S)$ is
referred to as a \emph{Coxeter system}.  The data of the presentation
corresponding to a Coxeter system can be encoded in an edge-labeled
graph.

\begin{defn}
Given a Coxeter system $(W,S)$, the \emph{Coxeter graph} has vertex set $S$, and edge set $\{\{s,t\} : m(s,t) > 2\}$. If $m(s,t) > 3$ then the edge $\{s,t\}$ is
labeled by $m(s,t)$.
\end{defn}

Note that pairs of nonadjacent vertices in a Coxeter graph correspond to pairs of commuting simple reflections.

\begin{defn}
A \emph{parabolic subgroup} $W_{I} \subseteq W$ is a subgroup generated by a subset $I
\subseteq S$.
\end{defn}

Parabolic subgroups are Coxeter groups in their own right.

\begin{example}\label{ex:symmetric group background}
For the symmetric group $S_n$, the simple reflections are usually chosen to
be the adjacent transpositions $(i\leftrightarrow i+1)$ for $1 \leq i
<n$.  The Coxeter graph for $S_n$ is a path with $n-1$ vertices labeled $1$, $2$, $\ldots$, $n-1$, consecutively.  The reflections in $S_n$ are the
transpositions $(i\leftrightarrow j)$ for $1\leq i < j\leq n$.
The parabolic subgroups of $S_{n}$, which are also known as \emph{Young subgroups},
are always products of smaller symmetric groups.
\end{example}

Because the simple reflections in a Coxeter system generate the Coxeter group, every element in the group can be written as a product of these generators.

\begin{defn}
For a Coxeter system $(W,S)$ and an element $w \in W$, the
\emph{length} $\ell(w)$ of $w$ is the minimum number of simple reflections needed to produce $w$; that is, $\ell(w)$ is minimal so that $w = s_1 \cdots s_{\ell(w)}$ for $s_i \in S$. Such a word $s_1 \cdots s_{\ell(w)}$ is a \emph{reduced
expression} for $w$.
\end{defn}

The \emph{Bruhat order} on $W$ is defined by taking the transitive closure of the relations $w < wt$, where $t \in T$ is a reflection of $W$ and $\ell(w) < \ell(wt)$. Bruhat order is ranked by the length function, and can be understood in terms of subwords of reduced expressions:
if $v = s_1 \cdots s_k$ is a reduced expression, then $u \leq v$ if and only if
$u = s_{i_1} \cdots s_{i_l}$ for some $1 \leq i_1 < \cdots < i_l \leq k$.

The computations in this paper are stated in terms of the ``marine model'', which we introduce in
the next section. A marine model is in turn made up of the following objects:

\begin{defn}
A simple reflection $s \in S$ is a \emph{right ascent} of $w$ if $\ell(ws) > \ell(w)$. Similarly, if $\ell(sw) > \ell(w)$, then $s$
is a \emph{left ascent}. By default, \emph{ascents} will refer to right ascents. We denote the set of (right) ascents by
\[
\asc(w) := \asc_R(w) = \{ s \in S : \ell(ws) > \ell(w)\}.
\]
Similarly, $\asc_L(w) = \{s \in S : \ell(sw) > \ell(w)\}$. Every reduced expression
for $w^{-1}$ is the reverse of a reduced expression for $w$, so $\asc_L(w)=\asc_R(w^{-1})$.
An element of $S$ that is not an ascent is a \emph{descent}. Again, by default \emph{descents} will refer to right descents. We denote the set of (right) descents by
\[
\Des(w) := \Des_R(w) = S \setminus \asc_R(w) =  \{ s \in S : \ell(ws) < \ell(w)\},
\]
where the relationship between $\Des_R$ and $\asc_R$ follows from the fact that $\ell(ws) \neq \ell(w)$ for all $w \in W$ and $s \in S$.
Similarly, the set of left descents is denoted $\Des_L(w)$, and $\Des_L(w) =
\Des_R(w^{-1})$.
\end{defn}

We refer the reader to \cite{b-b,Hum} for further background on Coxeter groups.

\subsection{Parabolic cosets of Coxeter groups}

In this work, we are concerned with parabolic double cosets. To give that study some context, and to emphasize the complexity of those objects, we first briefly state some facts about one-sided parabolic cosets.

Let $W_J$ be a parabolic subgroup of $W$. The left cosets in the quotient
$W/W_J$ each have a unique minimal-length element, and thus $W / W_J$ can be
identified with the set $W^{J}$ of all minimal-length left $W_J$-coset
representatives.  An element $w \in W$ belongs to $W^J$ if and only if $J$
contains no right descents of $w$; that is, if and only if $J \subseteq \asc(w)$. When
$S$ is fixed, we write $J^c = S-J$ for the complement of $J$. Thus we can also
say that $w \in W^J$ if and only if $\Des(w) \subseteq J^c$.

Every element $w \in W$ can be written uniquely as $w = uv$, where $u \in W^J$
and $v \in W_J$. This is the \emph{parabolic decomposition} of $w$ \cite[Section
5.12]{Hum}. The product $w =uv$ is a \emph{reduced factorization}, meaning
that $\ell(w) = \ell(u) + \ell(v)$. Moreover, as a poset under Bruhat order, every coset
$w W_J$ is isomorphic to $W_J$. If $W_J$ is finite, then every coset $w W_J$ is also finite.
In addition, $W_J$ has unique minimal and maximal elements, and $wW_{J}$ is a Bruhat interval.

Analogous statements can be made for right cosets.  We use the notation ${}^I
W$ for the set of minimal-length right coset representatives for $W_I
\backslash W$.

\subsection{Parabolic double cosets of Coxeter groups}

A \emph{parabolic double coset} is a subset $C \subseteq W$ of the form $C
= W_I w W_J$ for some $w \in W$ and $I, J \subseteq S$. Parabolic
double cosets inherit some of the nice properties of one-sided parabolic cosets
mentioned previously, including the following:

\begin{prop}\label{P:standard1}
  Let $(W,S)$ be a Coxeter system, and fix $I,J \subseteq S$.

\begin{enumerate}[(a)]
\item Every parabolic double coset in $W_I \backslash W / W_J$ has a
  unique minimal element with respect to Bruhat order. As Bruhat order is graded by length, this element is also the unique element of minimal length.

\item An element $w \in W$ is the minimal-length element of a
    double coset in $W_I \backslash W / W_J$ if and only if
    $w$ belongs to both ${}^I W$ and $W^J$. Thus $W_I \backslash W /
    W_J$ can be identified with $${}^I W^J:= {}^I W \cap W^J = \{ w \in W : I \subseteq \asc_L(w) \mbox{ and } J \subseteq \asc_R(w)\}.$$

\item  The parabolic double cosets in $W_I \backslash W / W_J$ are
  finite if and only if $W_I$ and $W_J$ are both finite.  In this
  case, each $C \in W_I \backslash W / W_J$ has a unique maximal-length element which is also the unique maximal element with respect
  to Bruhat order. In particular, if $C$ is finite then it is a Bruhat
  interval.
\end{enumerate}
\end{prop}

For the proof of Proposition \ref{P:standard1}, see \cite{Bourbaki} or
\cite{kobayashi}. The following statement is a consequence of
Proposition~\ref{P:standard1}. It has appeared, for instance, in
\cite{garsia-stanton.1984,curtis.1985}.

\begin{cor} [Double Parabolic Decomposition]\label{C:parabolic}
    Fix $I,J \subseteq S$ and $w \in {}^I W^J$.  Set
    \[H := I \cap (w J w^{-1}).\]
    Then $u w \in W^J$ for $u \in W_I$ if and only if $u \in
    W_I^{H}$, the minimal-length coset representatives in $W_I/W_H$.
    Consequently, every element of $W_I w W_J$ can be written uniquely
    as $u w v$, where $u \in W_{I}^{H}$, $v \in W_J$, and $\ell(uwv) =
    \ell(u) + \ell(w) + \ell(v)$.
\end{cor}

If $W_I$ and $W_J$ are finite parabolic subgroups, then
Corollary~\ref{C:parabolic} gives a bijective proof of the double
coset formula $|W_I w W_J| = |W_I||W_J|/|W_H|$ \cite[Theorem 2.5.1, and Exercise 40 on page 49]{Herstein}.

\begin{lem}\label{lem:b.order}
  Let $C$ be a parabolic double coset in $W_I \backslash W / W_J$ with minimal element  $w \in {}^I W^J$.
  For $y \in C$, fix a reduced factorization $y = uwv$ where $u \in
  W_I$, $v \in W_J$.  For all $x \in W$,  the following are
  equivalent:
            \begin{enumerate}
                \item [(i)]$x \in C$ and $x \leq y$ in Bruhat order,
                \item [(ii)]$w \leq x \leq y$, \text{ and }
                \item [(iii)]$x = u' w v'$ for some $u' \leq u$ and $v' \leq v$.
            \end{enumerate}
\end{lem}

\begin{proof} It is straightforward to check that (i) implies (ii) and
(iii) implies (i).  We need to show that (ii) implies (iii).

Take a
reduced expression $y = s_1 \cdots s_k$, where $u = s_1 \cdots s_a$,
$w = s_{a+1} \cdots s_b$, and $v = s_{b+1} \cdots s_k$.  This means that $s_i \in I$ for all $i \in [1,a]$, and $s_i \in J$ for all $j \in [b+1,k]$, while $s_{a+1} \not\in I$ and $s_b \not\in J$.

If
$x \leq y$ then $x$ has a reduced expression $s_{i_1} \cdots s_{i_l}$
for $1 \leq i_1 < i_2 < \ldots < i_l \leq k$.  If $w \leq x$
then $w$ has a reduced expression $w = s_{i_{j_1}} \cdots
s_{i_{j_m}}$, where $1 \leq j_1 < \ldots < j_m \leq l$.  Since $w \in
{}^I W^J$, we must have $s_{i_{j_1}} \not\in I$ and $s_{i_{j_m}} \not\in J$.
Thus $i_{j_1} \geq a+1$ and $i_{j_m} \leq b$. In other words,
\[
i_{j_1-1} < a+1 \leq i_{j_1} < \cdots < i_{j_m} \leq b < i_{j_m+1}.
\]
Since $m = \ell(w) = b -a$, and there are only $b-a$ letters from $a+1$ to $b$, we have
\[
\{i_{j_1} < \cdots < i_{j_m}\} = \{a+1 < \cdots < b\},
\]
and $u' := s_{i_1} \cdots s_{i_{j_1-1}} \in W_I$,
$v' := s_{i_{j_m+1}} \cdots s_{i_l} \in W_J$, as desired.
\end{proof}

In the remainder of the paper, we will assume that we know the unique
minimal-length element of a parabolic double coset.  The following
corollary shows that this is computationally easy to find from any
presentation of the coset.  The algorithm can also be used to test if
an arbitrary Bruhat interval is a parabolic double coset.

\begin{cor}\label{cor:min.max.elems}
Given any parabolic double coset $C=W_{I}wW_{J}$ with $w$ not
nec\-es\-sar\-i\-ly minimal, one can find the unique minimal element in $C$ by
applying a simple greedy algorithm to $w$. The algorithm proceeds by recursively multiplying $w$
by either $s_{i}$ on the left for any $s_i \in I\cap\Des_{L} (w)$, or $s_{j}$ on
the right for any $s_j \in J\cap \Des_{R} (w)$.  The algorithm terminates in at
most $\ell(w)$ steps with an element
$\mathrm{min}(C)=s_{i_{p}}\cdots s_{i_{1}} w s_{j_{1}} \cdots
s_{j_{q}}$ that has no left descents in $I$, nor right descents in $J$.

If $C$ is finite, the maximal element $\mathrm{max}(C)$ of $C$ can be
found in the analogous way by using ascent sets instead of descent sets
for $w$.
\end{cor}

That $\mathrm{min}(C)$ (respectively, $\mathrm{max}(C)$) is the unique minimal (respectively, maximal) element of $C$ follows from Corollary~\ref{lem:b.order}.

\begin{cor}\label{cor:interval.test}
Let $[u,v]$ be a finite interval in Bruhat order in any Coxeter group
$W$.  Then $[u,v]$ is a parabolic double coset if and only if $u =
\mathrm{min}(C)$ where
\begin{align*}
C &= W_{I}vW_{J},\\
I &= \asc_{L}(u) \cap\des_{L}(v),\\
J &= \asc_{R}(u) \cap \des_{R}(v), \text{ and}
\end{align*}
$\mathrm{min}(C)$ is found via the greedy algorithm described in Corollary~\ref{cor:min.max.elems}, starting at $v$.
\end{cor}

Another property of finite intervals $[u,v]$ that are parabolic double cosets is that they are rank-symmetric (see \cite{kobayashi}). It is natural to wonder if this rank-symmetry follows because the interval is self-dual, but this is not generally true. For example, the interval from the identity to the permutation $54312=s_3s_2s_1s_4s_3s_2s_4s_3s_4$ in the symmetric group $S_5$ can be written as
$$[e, 54312] = W_{\{s_2,s_3,s_4\}} e W_{\{s_1,s_2,s_3\}},$$
where $s_i$ denotes the $i$th adjacent transposition. Computer verification shows this interval is not self-dual.

\section{Parabolic double cosets in the symmetric group}\label{sec:symmetric group}

In this section we describe one of the main tools and results of this paper: the
marine model for $S_n$, and the accompanying formula for the number
$c_w$ of parabolic double cosets with a fixed minimal permutation $w
\in S_n$. At this stage, we focus on motivating the marine model, and
consequently some aspects are discussed only informally. All the facts
used in the enumeration will be discussed more formally in the next
section on general Coxeter groups.

\subsection{Ascents and descents in the symmetric group}

By Proposition~\ref{P:standard1}, we know that $w$ is the minimal element of a parabolic double coset $W_I w W_J$ if and only if $I \subseteq \asc_L(w)$ and $J \subseteq \asc_R(w)$. In the symmetric group, ascents have a well-known combinatorial description which we now describe.

First, recall from Example~\ref{ex:symmetric group background} that the symmetric group $S_n$ is a Coxeter group with generating set $S= \{ s_1,\ldots,s_{n-1}\}$, where $s_i$ denotes the
$i$th adjacent transposition. Elements of $S_n$ are encoded by
permutations of the set $\{1,2,\ldots,n\}$. We will usually write a
permutation $w \in S_n$ in one-line notation $w= w(1) \cdots
w(n)$. Thus right action permutes positions, while
left action permutes values.

\begin{example}
If $w = 7123546 \in S_7$, then $ws_6 = 7123564$ and $s_6w = 6123547$.
\end{example}

The length function for $S_n$ is the inversion statistic for permutations
\[
 \ell(w) = \inv(w):=\#\{ (i,j) : i < j \mbox{ and } w(i) > w(j) \}.
\]
Since $w$ and $ws_j$ differ only in that the letters $w(j)$ and $w(j+1)$ have
swapped positions, we have $\ell(w s_j) > \ell(w)$ if and only if
$w(j+1) > w(j)$. Thus, in a standard abuse of notation, we can write \[
 \asc_R(w) = \{ 1\leq j \leq n-1 : w(j) < w(j+1) \},
\]
and similarly
\[
 \Des_R(w)= \{1\leq j \leq n-1 : w(j) > w(j+1) \}.
\]
In the symmetric group, then, we can think of ascents and descents in terms of
\emph{positions} in permutations, in addition to the standard interpretation in terms
of simple generators. The study of these combinatorial notions of ascents and
descents goes (at least) as far back as the work of MacMahon in the early
twentieth century (see, for example, \cite{CarlitzRiordan,MacMahon}).

Left multiplication by $s_i$ swaps the positions of the letters $i$ and $i+1$, so
$s_i$ is a left ascent of $w$ if and only if the value $i$ appears to the left
of $i+1$ in the word $w(1) \cdots w(n)$, and we can think of
$\asc_L(w)=\asc_R(w^{-1})$ and $\Des_L(w) = \Des_R(w^{-1})$ as values of the
permutation.

To understand parabolic double cosets, we need to look at a particular kind of ascent which is, in a sense, ``small.''

\begin{defn}
A \emph{small right ascent} of $w$ is an index $j$ such that
$w(j+1)=w(j)+1$. If $j$ is a small right ascent, then $w s_j w^{-1} = s_i$,
where $i = w(j)$, and we say that $i$ is a \emph{small left ascent} of $w$.  Any
ascent that is not a small ascent is a \emph{large ascent}.
\end{defn}

\begin{example}\label{ex:7123546.pt1}
    For $w = 7123546$, we have $\asc_R(w) = \{ 2, 3, 4, 6\}$, $\Des_R(w) = \{
    1,5\}$, $\asc_L(w) = \{ 1,2,3,5\}$, and $\Des_L(w) = \{4, 6\}$. The small right
    ascents of $w$ are $\{2,3\}$, and the large right ascents are $\{4,6\}$. The small left
    ascents of $w$ are $\{1,2\}$, and the large left ascents are $\{3,5\}$.
\end{example}

\subsection{Balls in boxes}

Parabolic double cosets in the symmetric group can be represented by
\emph{balls-in-boxes pictures}, in which a number of balls are placed
in a two-dimensional grid of boxes separated by some solid vertical
and horizontal ``walls.'' This idea, attributed to Nantel Bergeron,
appears in work of Diaconis and Gangolli \cite[Proof of Theorem
  3.1]{Diaconis-Gangoli-1994} (see also
\cite{Pe15,PetersenBallsinBoxes}).

We construct the balls-in-boxes picture for a permutation $w$ by
placing balls as one would do in a permutation matrix.  To be precise, if
$w(a)=b$, we put a ball in column $a$ of row $b$, where columns are
labeled left-to-right and rows are labeled bottom-to-top in Cartesian
coordinates.  The symmetric group acts on the left by
permuting rows of such pictures, and on the right by permuting columns.

We can consider a parabolic double coset $C = W_I w W_J$ as a collection
of pictures.  While $C$ is not invariant under the full action of the
symmetric group, it is invariant under the left action of $W_I$ and
the right action of $W_J$. Thus walls are added to the picture to
indicate which simple transpositions are allowed to act on $C$.  For $i
\in I^c$, the complement of $I$, we put a horizontal wall between rows
$i$ and $i+1$. For $j \in J^c$, we put a vertical wall between columns
$j$ and $j+1$. We also draw walls around the boundary of the entire
permutation because no simple transpositions act in those positions. If
no sets $I$ and $J$ are specified, then we will assume that $I=\asc_L(w)$ and
$J=\asc_R(w)$; that is, we will draw walls in the left and right descent
positions.  Thus $C$ can be represented by the balls-in-boxes
pictures with walls, for $(I,w,J)$.

\begin{example}\label{ex:7123546.pt2}
Continuing Example~\ref{ex:7123546.pt1}, the permutation $w = 7123546$ is the minimal-length representative for any parabolic
double coset $W_IwW_J$ with $I \subseteq \{ 1,2,3,5\} = \asc_L(w)$ and $J \subseteq \{ 2,3,4,6\} = \asc_R(w)$. Figure~\ref{fig:double} depicts the balls-in-boxes model for $w=7123546$, with $I = \{ 1,2,3,5\}$ and $J = \{ 2,3,4,6\}$.
\end{example}

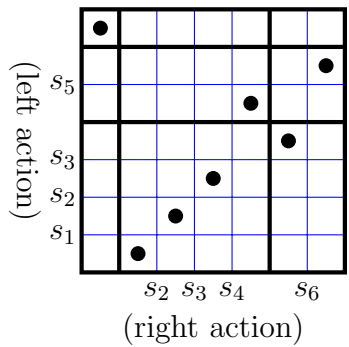
\begin{figure}[htbp]
\[
 \begin{tikzpicture}[scale=.25]
  \draw[blue] (0,0) grid[step=2] (14,14);
  \draw (1,13) node[circle,fill=black,inner sep =2] {};
  \draw (3,1) node[circle,fill=black,inner sep =2] {};
  \draw (5,3) node[circle,fill=black,inner sep =2] {};
  \draw (7,5) node[circle,fill=black,inner sep =2] {};
  \draw (9,9) node[circle,fill=black,inner sep =2] {};
  \draw (11,7) node[circle,fill=black,inner sep =2] {};
  \draw (13,11) node[circle,fill=black,inner sep =2] {};
  \draw[ultra thick] (2,0)--(2,14);
  \draw[ultra thick] (10,0)--(10,14);
  \draw[ultra thick] (0,8)--(14,8);
  \draw[ultra thick] (0,12)--(14,12);
  \draw[ultra thick] (0,0) rectangle (14,14);
  \draw (4,-1) node {$s_2$};
  \draw (6,-1) node {$s_3$};
  \draw (8,-1) node {$s_4$};
  \draw (12,-1) node {$s_6$};
  \draw (-1,2) node {$s_1$};
  \draw (-1,4) node {$s_2$};
  \draw (-1,6) node {$s_3$};
  \draw (-1,10) node {$s_5$};
  \draw (7,-3) node {(right action)};
  \draw (-3,7) node[rotate=-90] {(left action)};
 \end{tikzpicture}
\]
\caption{The balls-in-boxes model for a parabolic double coset. The light blue lines indicate which simple transpositions preserve the parabolic double coset.}\label{fig:double}
\end{figure}

Given a parabolic double coset $C = W_I w W_J$, we would like to find its
minimal and maximal elements as in Corollary \ref{cor:min.max.elems}.  The
balls-in-boxes picture can help with this task: The walls in the picture
partition the $n \times n$ grid into rectangular
enclosures.  By using the adjacent transpositions of rows and columns
that are not separated by walls, sort the balls between each parallel pair of adjacent
walls so that they create no inversions.  In this
sorting process, no ball leaves its enclosure; that is, the number of balls contained in any given enclosure does not change.  The
resulting picture will represent $(I,u,J)$ where $u$ is the minimal
length representative for the parabolic double coset $W_IwW_J$. Similarly, if we sort the balls so to maximize the number of inversions between each parallel pair of adjacent walls, we get the maximal-length representative $v$ for
$W_IwW_J$. The parabolic double coset is the Bruhat interval $[u,v]$ with $u$ and $v$ as just defined.

\begin{example}\label{ex:3512467}
Let $w=3512467$, with $I = \{1,3,4\}$ and $J=\{2,3,5,6\}$. Figure
\ref{fig:minmaxreps} shows the balls-in-boxes pictures for the triples
$(I,w,J)$, $(I,u,J)$, and $(I,v,J)$, where $u=3124567$ is the minimal element
of the parabolic double coset and $v=5421763$ is the maximal element.
\end{example}

\begin{figure}[htbp]
\[
\begin{tikzpicture}
\draw (-4,0) node (a) {
 \begin{tikzpicture}[scale=.25]
  \draw[blue] (0,0) grid[step=2] (14,14);
  \draw (1,5) node[circle,fill=black,inner sep =2] {};
  \draw (3,9) node[circle,fill=black,inner sep =2] {};
  \draw (5,1) node[circle,fill=black,inner sep =2] {};
  \draw (7,3) node[circle,fill=black,inner sep =2] {};
  \draw (9,7) node[circle,fill=black,inner sep =2] {};
  \draw (11,11) node[circle,fill=black,inner sep =2] {};
  \draw (13,13) node[circle,fill=black,inner sep =2] {};
  \draw[ultra thick] (2,0)--(2,14);
  \draw[ultra thick] (8,0)--(8,14);
  \draw[ultra thick] (0,4)--(14,4);
  \draw[ultra thick] (0,10)--(14,10);
  \draw[ultra thick] (0,12)--(14,12);
  \draw[ultra thick] (0,0) rectangle (14,14);
  \draw (4,-1) node {$s_2$};
  \draw (6,-1) node {$s_3$};
   \draw (10,-1) node {$s_5$};
  \draw (12,-1) node {$s_6$};
  \draw (-1,2) node {$s_1$};
  \draw (-1,6) node {$s_3$};
  \draw (-1,8) node {$s_4$};
  \end{tikzpicture}
 };
\draw (4,-3) node (b) {
\begin{tikzpicture}[scale=.25]
  \draw[blue] (0,0) grid[step=2] (14,14);
  \draw (1,5) node[circle,fill=black,inner sep =2] {};
  \draw (3,1) node[circle,fill=black,inner sep =2] {};
  \draw (5,3) node[circle,fill=black,inner sep =2] {};
  \draw (7,7) node[circle,fill=black,inner sep =2] {};
  \draw (9,9) node[circle,fill=black,inner sep =2] {};
  \draw (11,11) node[circle,fill=black,inner sep =2] {};
  \draw (13,13) node[circle,fill=black,inner sep =2] {};
  \draw[ultra thick] (2,0)--(2,14);
  \draw[ultra thick] (8,0)--(8,14);
  \draw[ultra thick] (0,4)--(14,4);
  \draw[ultra thick] (0,10)--(14,10);
  \draw[ultra thick] (0,12)--(14,12);
  \draw[ultra thick] (0,0) rectangle (14,14);
  \draw (4,-1) node {$s_2$};
  \draw (6,-1) node {$s_3$};
   \draw (10,-1) node {$s_5$};
  \draw (12,-1) node {$s_6$};
  \draw (-1,2) node {$s_1$};
  \draw (-1,6) node {$s_3$};
  \draw (-1,8) node {$s_4$};
 \end{tikzpicture}
 };
 \draw (4,3) node (c) {
\begin{tikzpicture}[scale=.25]
  \draw[blue] (0,0) grid[step=2] (14,14);
  \draw (1,9) node[circle,fill=black,inner sep =2] {};
  \draw (3,7) node[circle,fill=black,inner sep =2] {};
  \draw (5,3) node[circle,fill=black,inner sep =2] {};
  \draw (7,1) node[circle,fill=black,inner sep =2] {};
  \draw (9,13) node[circle,fill=black,inner sep =2] {};
  \draw (11,11) node[circle,fill=black,inner sep =2] {};
  \draw (13,5) node[circle,fill=black,inner sep =2] {};
  \draw[ultra thick] (2,0)--(2,14);
  \draw[ultra thick] (8,0)--(8,14);
  \draw[ultra thick] (0,4)--(14,4);
  \draw[ultra thick] (0,10)--(14,10);
  \draw[ultra thick] (0,12)--(14,12);
  \draw[ultra thick] (0,0) rectangle (14,14);
  \draw (4,-1) node {$s_2$};
  \draw (6,-1) node {$s_3$};
   \draw (10,-1) node {$s_5$};
  \draw (12,-1) node {$s_6$};
  \draw (-1,2) node {$s_1$};
  \draw (-1,6) node {$s_3$};
  \draw (-1,8) node {$s_4$};
 \end{tikzpicture}
 };
 \draw[->] (a)-- node[midway,below,fill=white,inner sep=2] {minimal} (b);
 \draw[->] (a)-- node[midway,above,fill=white,inner sep=2] {maximal} (c);
 \end{tikzpicture}
\]
\caption{A parabolic double coset, with its minimal and maximal representatives. Again, lines of reflection that represent allowable simple transpositions are drawn in blue.}\label{fig:minmaxreps}
\end{figure}
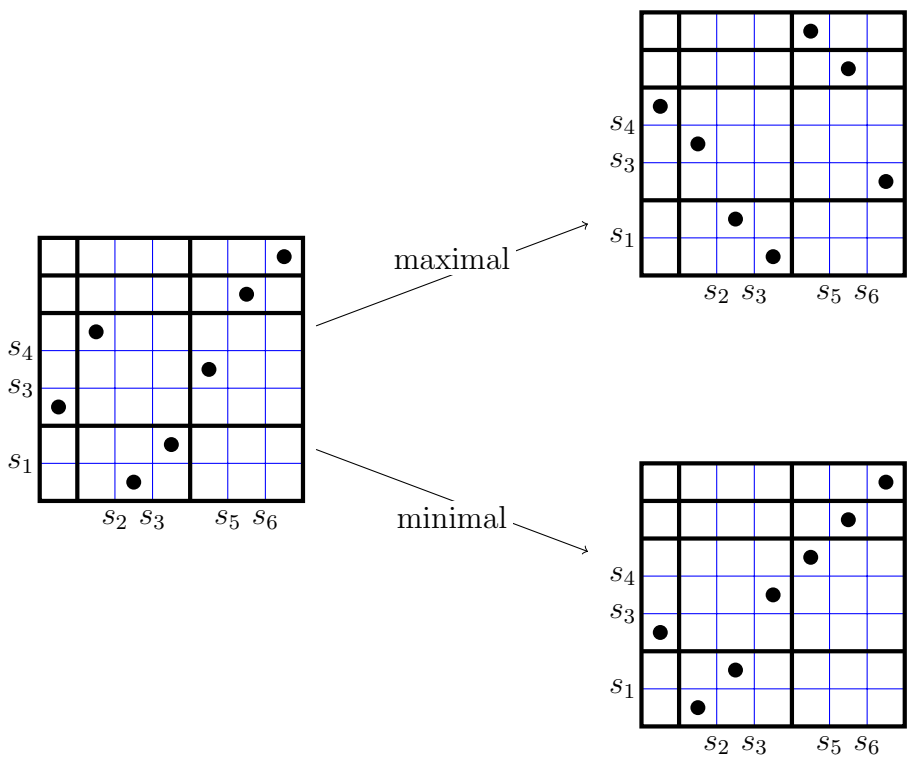

\subsection{Identifying canonical presentations in the symmetric group}\label{sec:lex-min in symmetric group}

Before we count parabolic double cosets having $w$ as a minimal element, we need to know
how to identify a canonical presentation. Parabolic double cosets are intervals in the Bruhat order, so they are uniquely identified by their maximal and minimal elements. In Theorem~\ref{T:minimalsym}, we will characterize the lex-minimal presentation of a parabolic double coset in the symmetric group. This foreshadows a more general result for all Coxeter groups, appearing in Theorem~\ref{T:minimal}, and it is this lex-minimal presentation that we will use in the enumeration in Theorem~\ref{main}. First, though, we will take a few moments to describe, quite simply, a canonical presentation for a parabolic double coset that is, in a sense, ``maximal.'' The generalized version of this will be discussed in Proposition~\ref{P:maxpres}.

\begin{defn}
For a parabolic double coset $C$ in $S_n$, let its minimal and maximal elements be $\Cmin$ and $\Cmax$, respectively, and set
\begin{align*}
M_L(C) & := \asc_L(\Cmin) \cap \des_L(\Cmax) \text{ and}\\
M_R(C) & := \asc_R(\Cmin) \cap \des_R(\Cmax).
\end{align*}
\end{defn}

\begin{prop}\label{prop:maxpressym}
Let $C$ be a parabolic double coset in $S_n$. There is a presentation
$$C = W_{M_L(C)}wW_{M_R(C)},$$
and this is the largest possible presentation for $C$, in the sense that if $C = W_I w' W_J$, then $I \subseteq M_L(C)$ and $J \subseteq M_R(C)$.
\end{prop}

\begin{proof}
Since $S_n$ is finite, the coset $C$ is a finite Bruhat interval, and Corollary \ref{cor:interval.test} gives the desired presentation.

Now suppose that $C = W_IwW_J$ is another presentation of this parabolic double coset. Since $\Cmin$ is the minimal element of $C$, we must have $I \subseteq \asc_L(\Cmin)$ and $J \subseteq \asc_R(\Cmin)$. Similarly, since $\Cmax$ is the maximal element of $C$, we must have $I \subseteq \des_L(\Cmax)$ and $J \subseteq \des_R(\Cmax)$. (In other words, no element of $I$ or $J$ can take us up in Bruhat order.) Hence, $I \subseteq \asc_L(\Cmin) \cap \des_L(\Cmax)$ and $J \subseteq \asc_R(\Cmin) \cap \des_R(\Cmax)$, as desired.
\end{proof}

The maximal presentation
$$C = W_{M_L(C)}wW_{M_R(C)}$$
of a parabolic double coset $C$ is quite natural. In the balls-in-boxes context, it describes the minimum number of walls necessary to produce the desired coset.

\begin{example}
Continuing Example~\ref{ex:3512467}, the parabolic double coset $C$ has $\Cmin = 3124567$ and $\Cmax = 5421763$, and $M_L(C) = \{1,3,4,6\}$ and $M_R(C) = \{2,3,5,6\}$. The balls-in-boxes pictures for the triples
$$(M_L(C),w,M_R(C)), \ (M_L(C),\Cmin,M_R(C)), \text{ and } (M_L(C),\Cmax,M_R(C))$$
appear in Figure~\ref{fig:minmaxreps-maxpres}.
\end{example}

\begin{figure}[htbp]
\[
\begin{tikzpicture}
\draw (-4,0) node (a) {
 \begin{tikzpicture}[scale=.25]
  \draw[blue] (0,0) grid[step=2] (14,14);
  \draw (1,5) node[circle,fill=black,inner sep =2] {};
  \draw (3,9) node[circle,fill=black,inner sep =2] {};
  \draw (5,1) node[circle,fill=black,inner sep =2] {};
  \draw (7,3) node[circle,fill=black,inner sep =2] {};
  \draw (9,7) node[circle,fill=black,inner sep =2] {};
  \draw (11,11) node[circle,fill=black,inner sep =2] {};
  \draw (13,13) node[circle,fill=black,inner sep =2] {};
  \draw[ultra thick] (2,0)--(2,14);
  \draw[ultra thick] (8,0)--(8,14);
  \draw[ultra thick] (0,4)--(14,4);
  \draw[ultra thick] (0,10)--(14,10);
  \draw[ultra thick] (0,0) rectangle (14,14);
  \draw (4,-1) node {$s_2$};
  \draw (6,-1) node {$s_3$};
   \draw (10,-1) node {$s_5$};
  \draw (12,-1) node {$s_6$};
  \draw (-1,2) node {$s_1$};
  \draw (-1,6) node {$s_3$};
  \draw (-1,8) node {$s_4$};
  \draw (-1,12) node {$s_6$};
  \end{tikzpicture}
 };
\draw (4,-3) node (b) {
\begin{tikzpicture}[scale=.25]
  \draw[blue] (0,0) grid[step=2] (14,14);
  \draw (1,5) node[circle,fill=black,inner sep =2] {};
  \draw (3,1) node[circle,fill=black,inner sep =2] {};
  \draw (5,3) node[circle,fill=black,inner sep =2] {};
  \draw (7,7) node[circle,fill=black,inner sep =2] {};
  \draw (9,9) node[circle,fill=black,inner sep =2] {};
  \draw (11,11) node[circle,fill=black,inner sep =2] {};
  \draw (13,13) node[circle,fill=black,inner sep =2] {};
  \draw[ultra thick] (2,0)--(2,14);
  \draw[ultra thick] (8,0)--(8,14);
  \draw[ultra thick] (0,4)--(14,4);
  \draw[ultra thick] (0,10)--(14,10);
  \draw[ultra thick] (0,0) rectangle (14,14);
  \draw (4,-1) node {$s_2$};
  \draw (6,-1) node {$s_3$};
   \draw (10,-1) node {$s_5$};
  \draw (12,-1) node {$s_6$};
  \draw (-1,2) node {$s_1$};
  \draw (-1,6) node {$s_3$};
  \draw (-1,8) node {$s_4$};
  \draw (-1,12) node {$s_6$};
 \end{tikzpicture}
 };
 \draw (4,3) node (c) {
\begin{tikzpicture}[scale=.25]
  \draw[blue] (0,0) grid[step=2] (14,14);
  \draw (1,9) node[circle,fill=black,inner sep =2] {};
  \draw (3,7) node[circle,fill=black,inner sep =2] {};
  \draw (5,3) node[circle,fill=black,inner sep =2] {};
  \draw (7,1) node[circle,fill=black,inner sep =2] {};
  \draw (9,13) node[circle,fill=black,inner sep =2] {};
  \draw (11,11) node[circle,fill=black,inner sep =2] {};
  \draw (13,5) node[circle,fill=black,inner sep =2] {};
  \draw[ultra thick] (2,0)--(2,14);
  \draw[ultra thick] (8,0)--(8,14);
  \draw[ultra thick] (0,4)--(14,4);
  \draw[ultra thick] (0,10)--(14,10);
  \draw[ultra thick] (0,0) rectangle (14,14);
  \draw (4,-1) node {$s_2$};
  \draw (6,-1) node {$s_3$};
   \draw (10,-1) node {$s_5$};
  \draw (12,-1) node {$s_6$};
  \draw (-1,2) node {$s_1$};
  \draw (-1,6) node {$s_3$};
  \draw (-1,8) node {$s_4$};
  \draw (-1,12) node {$s_6$};
 \end{tikzpicture}
 };
 \draw[->] (a)-- node[midway,below,fill=white,inner sep=2] {minimal} (b);
 \draw[->] (a)-- node[midway,above,fill=white,inner sep=2] {maximal} (c);
 \end{tikzpicture}
\]
\caption{The maximal presentation of a parabolic double coset, with its minimal and maximal representatives. (Cf.~Figure~\ref{fig:minmaxreps}.)}\label{fig:minmaxreps-maxpres}
\end{figure}

The enumeration in Theorem~\ref{main} will be done in terms of another canonical presentation of a parabolic double coset; namely, the lex-minimal presentation (see Definition \ref{defn:lex-minimal}). We spend a few moments now exploring some features of that presentation. We postpone a completely rigorous analysis of lex-minimal presentations until Section~\ref{sec:general groups}.

Suppose we begin with the balls-in-boxes picture of a parabolic double coset $C=W_IwW_J$. Assume that $I\subseteq \asc_L(w)$
and $J \subseteq \asc_R(w)$, so that $w$ is minimal. We can then ask whether any other walls could be inserted without changing $C$.

Suppose that $w$ has consecutive small left ascents $\{a,a+1,\ldots,b\} \subseteq I$, while $\{a-1,b+1\} \cap I = \emptyset$. Thus there are horizontal walls in positions $a-1$ and
$b+1$ of the balls-in-boxes picture, and a southwest-to-northeast diagonal of $b-a+1$ balls appears between them, spanning exactly $b-a+1$ columns. Suppose, further, that no vertical walls appear between these balls. In other words, the small right ascents $\{w^{-1}(a),w^{-1}(a+1) = w^{-1}(a)+1,\ldots,w^{-1}(b)\}$ are in the set
$J$. This is illustrated in the figure below.
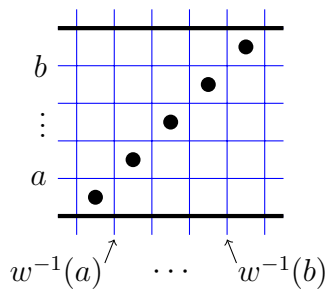
\[\begin{tikzpicture}[scale=.25]
  \draw[blue] (-1,-1) grid[step=2] (11,11);
  \draw (1,1) node[circle,fill=black,inner sep =2] {};
  \draw (3,3) node[circle,fill=black,inner sep =2] {};
  \draw (5,5) node[circle,fill=black,inner sep =2] {};
  \draw (7,7) node[circle,fill=black,inner sep =2] {};
  \draw (9,9) node[circle,fill=black,inner sep =2] {};
  \draw[ultra thick] (-1,0)--(11,0);
  \draw[ultra thick] (-1,10)--(11,10);
  \draw (-1,2) node[left] {$a$};
  \draw (-1,5.25) node[left] {$\vdots$};
  \draw (-1,8) node[left] {$b$};
  \draw (2,-3) node[left] {$w^{-1}(a)$};
  \draw[->] (1.5,-2.5) -- (2,-1.25);
  \draw (5,-3) node {$\,\cdots$};
  \draw[->] (8.5,-2.5) -- (8,-1.25);
  \draw (8,-3) node[right] {$w^{-1}(b)$};
 \end{tikzpicture}
\]
Now consider the same balls-in-boxes picture, but with $I'= I \setminus \{a,\ldots,b\}$. Locally, this appears as follows.
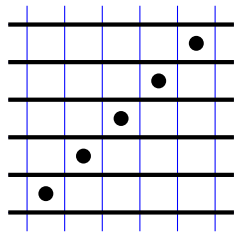
\[
\begin{tikzpicture}[scale=.25]
  \draw[blue] (-1,-1) grid[step=2] (11,11);
  \draw (1,1) node[circle,fill=black,inner sep =2] {};
  \draw (3,3) node[circle,fill=black,inner sep =2] {};
  \draw (5,5) node[circle,fill=black,inner sep =2] {};
  \draw (7,7) node[circle,fill=black,inner sep =2] {};
  \draw (9,9) node[circle,fill=black,inner sep =2] {};
  \draw[ultra thick] (-1,0)--(11,0);
  \draw[ultra thick] (-1,2)--(11,2);
  \draw[ultra thick] (-1,4)--(11,4);
  \draw[ultra thick] (-1,6)--(11,6);
  \draw[ultra thick] (-1,8)--(11,8);
  \draw[ultra thick] (-1,10)--(11,10);
 \end{tikzpicture}
\]
The new horizontal walls clearly do not impede the movement of balls outside this portion of the picture under either the left or right actions. Thus to determine if the two balls-in-boxes pictures represent the same parabolic double coset, it suffices to check that the balls in the picture can be sorted in the same way. In both pictures, right actions are enough to sort the balls into decreasing order. Hence both parabolic double cosets $W_IwW_J$ and $W_{I'}wW_J$ have the same maximal element, and so $W_IwW_J=W_{I'}wW_J$. Because $|I'| < |I|$, we conclude from this that the presentation $W_IwW_J$ is not lex-minimal.

Swapping the roles of $I$ and $J$ and of ``left'' and ``right'' yields an analogous conclusion about non lex-minimality when small right ascents are squeezed between consecutive vertical walls.

Another scenario we can easily analyze involves both removing and inserting walls. Suppose that $a,a+1,\ldots,b$ is a sequence of consecutive small right ascents of $w$, none of which are in $J$. Suppose, further, that neither $a-1$ nor $b+1$ are in $J$. In other words, there are vertical walls in each of the gaps from $a-1$ to $b+1$. Now suppose that the corresponding left ascents, $w(a),\ldots,w(b)$ are in $I$, but $w(a-1)$ and $w(b+1)$ are not in $I$. Consider removing all the vertical walls in $a,\ldots,b$ and inserting horizontal walls in $w(a),\ldots,w(b)$. This is illustrated below.
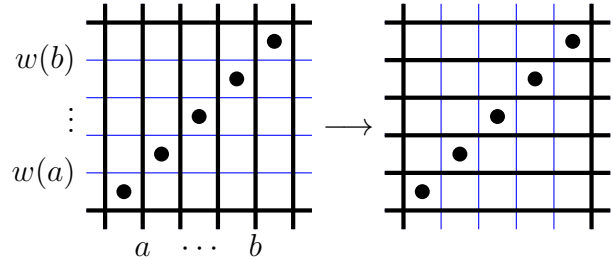
\[
\begin{tikzpicture}[scale=.25,baseline=1cm]
  \draw[blue] (-1,-1) grid[step=2] (11,11);
  \draw (1,1) node[circle,fill=black,inner sep =2] {};
  \draw (3,3) node[circle,fill=black,inner sep =2] {};
  \draw (5,5) node[circle,fill=black,inner sep =2] {};
  \draw (7,7) node[circle,fill=black,inner sep =2] {};
  \draw (9,9) node[circle,fill=black,inner sep =2] {};
  \draw[ultra thick] (-1,0)--(11,0);
  \draw[ultra thick] (2,-1)--(2,11);
  \draw[ultra thick] (4,-1)--(4,11);
  \draw[ultra thick] (6,-1)--(6,11);
  \draw[ultra thick] (8,-1)--(8,11);
  \draw[ultra thick] (-1,10)--(11,10);
  \draw[ultra thick] (0,-1)--(0,11);
  \draw[ultra thick] (10,-1)--(10,11);
  \draw (-1,2) node[left] {$w(a)$};
  \draw (-1,5.25) node[left] {$\vdots$};
  \draw (-1,8) node[left] {$w(b)$};
  \draw (2,-3) node[above] {$a$};
  \draw (5,-3) node[above] {$\,\cdots$};
  \draw (8,-3) node[above] {$b$};
 \end{tikzpicture}
 \longrightarrow
\begin{tikzpicture}[scale=.25,baseline=1cm]
  \draw[blue] (-1,-1) grid[step=2] (11,11);
  \draw (1,1) node[circle,fill=black,inner sep =2] {};
  \draw (3,3) node[circle,fill=black,inner sep =2] {};
  \draw (5,5) node[circle,fill=black,inner sep =2] {};
  \draw (7,7) node[circle,fill=black,inner sep =2] {};
  \draw (9,9) node[circle,fill=black,inner sep =2] {};
  \draw[ultra thick] (-1,0)--(11,0);
  \draw[ultra thick] (-1,2)--(11,2);
  \draw[ultra thick] (-1,4)--(11,4);
  \draw[ultra thick] (-1,6)--(11,6);
  \draw[ultra thick] (-1,8)--(11,8);
  \draw[ultra thick] (-1,10)--(11,10);
  \draw[ultra thick] (0,-1)--(0,11);
  \draw[ultra thick] (10,-1)--(10,11);
 \end{tikzpicture}
 \]
This change has no impact on the balls outside of this local area, and in both cases, we can sort the balls in decreasing order. Hence the two parabolic double cosets are the same. Let $I' = I\setminus \{w(a),\ldots, w(b)\}$ and let $J'=J \cup \{a,\ldots,b\}$. The picture on the left is $W_IwW_J$, while the picture on the right is $W_{I'}wW_{J'}$. Because $|I'|<|I|$, the presentation $W_IwW_J$ is not lex-minimal.

In fact, the preceding analysis characterizes lex-minimal presentations for $S_n$.

\begin{thm}\label{T:minimalsym}
    Let $w \in S_n$ and let $I$ and $J$ be subsets of the left and right ascent sets of $w$,  respectively.  Then $W_{I}w W_{J}$ is a lex-minimal presentation of a
    parabolic double coset of $S_n$ if and only if
    \begin{itemize}
          \item if $\{a-1,a,\ldots,b,b+1\} \cap I = \{a,\ldots,b\}$ and these are all small left ascents, then $\{w^{-1}(a), \ldots, w^{-1}(b)\} \not\subseteq J$ and $\{w^{-1}(a)-1,w^{-1}(a), \ldots, w^{-1}(b),w^{-1}(b)+1\} \cap J \neq \emptyset$, and
          \item if $\{a-1,a,\ldots,b,b+1\} \cap J = \{a,\ldots,b\}$ and these are all small right ascents, then $\{w(a),\ldots,w(b)\}\not\subseteq I$.
    \end{itemize}
\end{thm}

One direction of the proof of Theorem~\ref{T:minimalsym} was discussed above. The other direction is a specialization of the upcoming result, Theorem~\ref{T:minimal}, for general Coxeter groups.

\subsection{The marine model in the symmetric group}\label{sec:marinemodelforSn}

We are now ready to introduce the marine model. For $S_n$, we
illustrate this model using balls-in-boxes pictures. Later, in
Definition~\ref{defn:marine model pictures}, we will use a different
method of illustration that applies to any Coxeter group.

\begin{defn}\label{defn:marine model for S_n}
The following objects comprise the \emph{marine model} for a permutation $w \in S_n$. For each, we will give both a combinatorial and a pictorial description.
\begin{itemize}
\item A \emph{raft} of $w$ is an interval $[a,b]=\{a,\ldots,b\}$ such
  that $a,a+1,\ldots, b$ are all small right ascents of $w$, while
  $a-1$ and $b+1$ are not. In other words, a raft is a maximal
  increasing run of consecutive values: $w(a), w(a+1)=w(a)+1, \ldots,
  w(b)=w(a)+(b-a)$.  The \emph{size} of such a raft is
  $|[a,b]|=b-a+1$.

In a balls-in-boxes picture, a raft looks like a copy of an identity permutation, and we will connect the balls of a raft as follows.
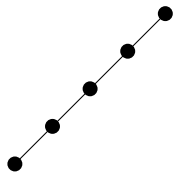
\[
\begin{tikzpicture}[scale=.5]
 \foreach \x in {1,...,5} {
  \draw (\x,\x) node[circle,fill=black, inner sep=2] {};
 }
 \draw (1,1)--(5,5);
\end{tikzpicture}
\]
We write $\rafts(w)$ for the set of intervals of small (right) ascents that make up the rafts of $w$.
\end{itemize}

\noindent The ends of a raft may connect to other balls in various ways.

\begin{itemize}
\item A \emph{tether} of $w$ is an ascent that is adjacent to two rafts. Because rafts are maximal, tethers are necessarily large ascents. If desired, we can specify a ``left'' tether or a ``right'' tether, corresponding to whether the large ascent in question is a left or a right ascent. We let $\tethers_L(w)$ and $\tethers_R(w)$ denote the sets of positions of the left and right tethers of $w$, respectively.

In pictures, we draw tethers with squiggly lines. The left-hand picture below depicts a left tether, and the right-hand figure depicts a right tether.
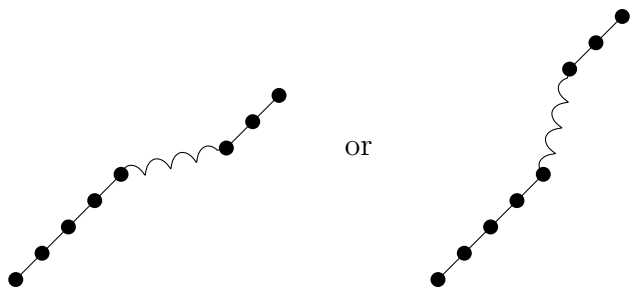
\[
\begin{tikzpicture}[baseline=2cm,scale=.35]
 \foreach \x in {1,...,5} {
  \draw (\x,\x) node[circle,fill=black, inner sep=2] {};
 }
 \foreach \x in {1,...,3} {
  \draw (\x+8,\x+5) node[circle,fill=black, inner sep=2] {};
 }
 \draw (1,1)--(5,5);
 \draw (9,6)--(11,8);
 \draw[snake=coil] (5,5)--(9,6);
\end{tikzpicture}
\hspace{.25in}
\mbox{ or }
\hspace{.25in}
 \begin{tikzpicture}[baseline=2cm,scale=.35]
 \foreach \x in {1,...,5} {
  \draw (\x,\x) node[circle,fill=black, inner sep=2] {};
 }
 \foreach \x in {1,...,3} {
  \draw (\x+5,\x+8) node[circle,fill=black, inner sep=2] {};
 }
 \draw (1,1)--(5,5);
 \draw (6,9)--(8,11);
 \draw[snake=coil] (5,5)--(6,9);
\end{tikzpicture}
\]
Note that while two rafts can be connected by at most one tether, it is possible to have a raft with both types of tethers emanating from the same ball, as shown below.
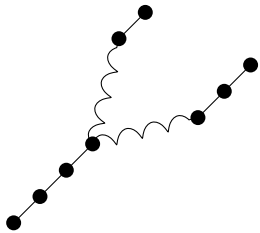
\[
\begin{tikzpicture}[baseline=2.5cm,scale=.35]
 \foreach \x in {2,...,5} {
  \draw (\x,\x) node[circle,fill=black, inner sep=2] {};
 }
 \foreach \x in {1,...,3} {
  \draw (\x+8,\x+5) node[circle,fill=black, inner sep=2] {};
 }
 \foreach \x in {1,...,2} {
  \draw (\x+5,\x+8) node[circle,fill=black, inner sep=2] {};
 }
 \draw (2,2)--(5,5);
 \draw (9,6)--(11,8);
 \draw (6,9)--(7,10);
 \draw[snake=coil] (5,5)--(9,6);
 \draw[snake=coil] (5,5)--(6,9);
\end{tikzpicture}
\]

\item A \emph{rope} of $w$ is a large ascent that is adjacent to exactly one raft.
Again, we can specify that a given rope is a ``left'' rope or a ``right'' rope. We let $\ropes_L(w)$ and $\ropes_R(w)$ denote the sets of positions of the left and right ropes, respectively.

In pictures, we draw ropes with dashed lines. The two leftmost figures below depict left ropes, whereas the two rightmost depict right ropes. Because a rope is adjacent to exactly one raft, the ``$\times$'' symbols in the figures below indicate locations where balls may not appear.
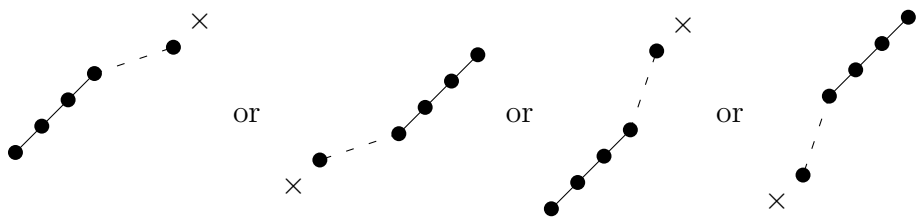
\[
\begin{tikzpicture}[baseline=.75cm,scale=.35]
 \foreach \x in {1,...,4} {
  \draw (\x,\x) node[circle,fill=black, inner sep=2] {};
 }
 \draw (7,5) node[circle,fill=black, inner sep=2] {};
 \draw (1,1)--(4,4);
 \draw[loosely dashed] (4,4)--(7,5);
 \draw (8,6) node {$\times$};
\end{tikzpicture}
\mbox{ or }
\begin{tikzpicture}[baseline=.5cm,scale=.35]
 \foreach \x in {1,...,4} {
  \draw (\x,\x) node[circle,fill=black, inner sep=2] {};
 }
 \draw (-2,0) node[circle,fill=black, inner sep=2] {};
 \draw (1,1)--(4,4);
 \draw[loosely dashed] (1,1)--(-2,0);
 \draw (-3,-1) node {$\times$};
\end{tikzpicture}
\mbox{ \ or }
\begin{tikzpicture}[baseline=1.5cm,scale=.35]
 \foreach \x in {1,...,4} {
  \draw (\x,\x) node[circle,fill=black, inner sep=2] {};
 }
 \draw (5,7) node[circle,fill=black, inner sep=2] {};
 \draw (1,1)--(4,4);
 \draw[loosely dashed] (4,4)--(5,7);
  \draw (6,8) node {$\times$};
\end{tikzpicture}
\mbox{ or }
\begin{tikzpicture}[baseline=0cm,scale=.35]
 \foreach \x in {1,...,4} {
  \draw (\x,\x) node[circle,fill=black, inner sep=2] {};
 }
 \draw (0,-2) node[circle,fill=black, inner sep=2] {};
 \draw (1,1)--(4,4);
 \draw[loosely dashed] (1,1)--(0,-2);
 \draw (-1,-3) node {$\times$};
\end{tikzpicture}
 \]

\item A \emph{float} of $w$ is a large ascent that is not adjacent to any rafts. Once again, we can specify that a given float is a ``left'' float or a ``right'' float. The sets of positions of left and right floats are denoted by $\floats_L(w)$ and $\floats_R(w)$, respectively.

In pictures, a float connects two isolated balls, and we draw floats with a dotted line. As was the case for ropes, we mark locations that cannot have a ball by ``$\times$'' symbols. The following figures depict left and right floats, respectively.
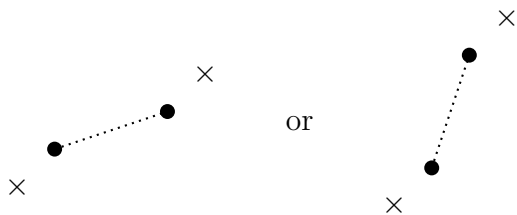
\[
\begin{tikzpicture}[baseline=.5cm,scale=.5]
  \draw[thick,dotted] (7,.5) node[circle,fill=black, inner sep=2] {} -- (10,1.5) node[circle,fill=black, inner sep=2] {};
  \foreach \x in {(6,-.5),(11,2.5)} {\draw \x node {$\times$};}
 \end{tikzpicture}
\hspace{.25in}
\mbox{ or }
\hspace{.25in}
 \begin{tikzpicture}[baseline=.5cm,scale=.5]
  \draw[thick,dotted] (0,0) node[circle,fill=black, inner sep=2] {} -- (1,3) node[circle,fill=black, inner sep=2] {};
  \foreach \x in {(-1,-1),(2,4)} {\draw \x node {$\times$};}
\end{tikzpicture}
\]
\end{itemize}
\end{defn}

Before examining the utility of this model, we describe two permutations whose marine models are, in a sense, extreme.

\begin{example}
\ \begin{itemize}
\item If $w \in S_n$ is the identity permutation, in which all positions are small ascents, then $w$ has one raft, $[1,n-1]$, and no floats, ropes, or tethers.
\item If $w \in S_n$ is the longest permutation, in which no positions are ascents, then $w$ has no rafts, floats, ropes, or tethers.
\end{itemize}
\end{example}

To streamline notation, we will write left large ascents with tick marks $1',2',\ldots$ and set
\begin{align*}
\tethers(w)&:=\{ i' : i \in \tethers_L(w)\} \cup \tethers_R(w),\\
\ropes(w) &:= \{ i' : i \in \ropes_L(w)\} \cup \ropes_R(w), \text{ and}\\
\floats(w) &:= \{ i' : i \in \floats_L(w)\} \cup \floats_R(w).
\end{align*}

\begin{example}\label{ex:7123546.pt3}
Continuing Example~\ref{ex:7123546.pt2}, for the permutation $w = 7123546$, the marine model is overlaid on the balls-in-boxes picture of $w = 7123546$ in Figure~\ref{fig:marine model on grid 7123546}. We have $\rafts(w) = \{[2,3]\}$, $\tethers(w) = \emptyset$, $\ropes(w) = \{3',4\}$, and $\floats(w) = \{5',6\}$.
\begin{figure}[htbp]
$$\begin{tikzpicture}[scale=.25]
  \draw[blue] (0,0) grid[step=2] (14,14);
  \draw (1,13) node[circle,fill=black,inner sep =2] {};
  \draw (3,1) node[circle,fill=black,inner sep =2] {};
  \draw (5,3) node[circle,fill=black,inner sep =2] {};
  \draw (7,5) node[circle,fill=black,inner sep =2] {};
  \draw (9,9) node[circle,fill=black,inner sep =2] {};
  \draw (11,7) node[circle,fill=black,inner sep =2] {};
  \draw (13,11) node[circle,fill=black,inner sep =2] {};
  \draw[ultra thick] (2,0)--(2,14);
  \draw[ultra thick] (10,0)--(10,14);
  \draw[ultra thick] (0,8)--(14,8);
  \draw[ultra thick] (0,12)--(14,12);
  \draw[ultra thick] (0,0) rectangle (14,14);
  \draw (3,1) -- (7,5);
  \draw[loosely dashed] (7,5) -- (9,9);
  \draw[loosely dashed] (7,5) -- (11,7);
  \draw[thick,dotted] (11,7) -- (13,11);
  \draw[thick,dotted] (9,9) -- (13,11);
  \foreach \x in {1,...,6}{
   \draw (2*\x,-2) node[above] {$\x$};
   \draw (-1,2*\x) node {$\x'$};
  }
 \end{tikzpicture}$$
 \caption{The permutation $7123546$ with its lone raft, two ropes, and two floats marked.  There are no tethers. Walls are drawn in left and right descent positions.}\label{fig:marine model on grid 7123546}
\end{figure}
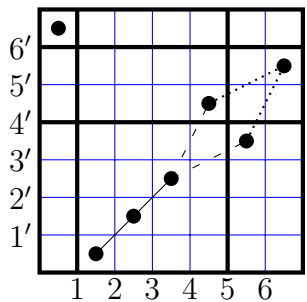
\end{example}

\subsection{Starting to count}\label{sec:startcounting}

We now illustrate how the marine model enables enumeration of parabolic double cosets with a given minimal element.

We start from the idea that rafts are, in a sense, well-behaved, since they look like copies of the identity permutation. Suppose we know the number of parabolic double cosets for an identity permutation (we will study this number in Section~\ref{sec:Snenumeration}). We would next want to identify how any relationships between rafts and isolated balls will affect the total number of parabolic double cosets whose minimal element is the permutation $w$.

Consider the balls-in-boxes picture for $w$, with $I=\asc_L(w)$ and $J=\asc_R(w)$ as large as possible. The parabolic double coset $W_IwW_J$ will contain all other cosets for which $w$ is the minimal element. We want to insert walls in this picture, yielding all the lex-minimal representations of these cosets.

We start to understand how to do this with a simple observation.

\begin{lem}\label{lem:descent}
If two balls are connected by a tether, rope, or float, then the balls occupy different boxes; that is, there is a wall between them. If the connector is of ``left" type, then the balls are separated by a vertical wall; if the connector is of ``right" type, then the balls are separated by a horizontal wall.
\end{lem}

Floats exhibit particularly interesting behavior, which we highlight here.

\begin{lem}\label{lem:floats}
If two nodes are connected by a float, then inserting a wall in this position will result in a different parabolic double coset, independent of all other choices for the walls.
\end{lem}

\begin{proof}
Suppose, without loss of generality, that two balls are connected by a right float. Then Lemma~\ref{lem:descent} says there is a horizontal wall between them. Suppose the balls correspond to $w(i)$ and $w(i+1)$, with $w(i) < w(i+1)$. If there is no vertical wall between them, then the maximal element for this parabolic double coset, call it $v$, has $v(i) > v(i+1)$, obtained by acting on the right by $s_i$ at some point to swap the columns these balls occupy. If, on the other hand, there is a vertical wall in position $i$, then the maximal element must have $v(i) < v(i+1)$, since there are both vertical and horizontal bars between the two balls.
\end{proof}

Lemma~\ref{lem:floats} means that floats are independent actors in our counting. In other words, if $w$ has $m$ floats, then the formula for $c_w$ has the form
\[
 c_w = 2^m \cdot (\mbox{something}).
\]
If $w$ has no tethers, then each raft contributes independently to the formula, and that ``something" will be a product of terms related to each raft. If $w$ does have tethers, then these contributions will vary depending on whether we choose to place a wall in a tether position.

\begin{example}\label{ex:7123546.pt4}
Continuing Example~\ref{ex:7123546.pt3}, the permutation $w = 7123546$ has two floats, so $c_w$ is a multiple of four. If we ignore these floats, then we have only three potential horizontal walls (in positions $1', 2', 3'$) and three potential vertical walls (in positions $2, 3, 4$) whose insertion can possibly give rise to new cosets. This is equivalent to counting the parabolic double cosets for $x=12354$, pictured below.
$$
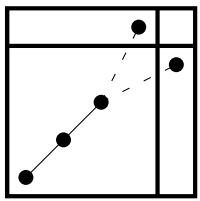
\begin{tikzpicture}[scale=.25]
  \draw (3,1) node[circle,fill=black,inner sep =2] {};
  \draw (5,3) node[circle,fill=black,inner sep =2] {};
  \draw (7,5) node[circle,fill=black,inner sep =2] {};
  \draw (9,9) node[circle,fill=black,inner sep =2] {};
  \draw (11,7) node[circle,fill=black,inner sep =2] {};
  \draw[ultra thick] (2,0)--(2,10);
  \draw[ultra thick] (10,0)--(10,10);
  \draw[ultra thick] (2,8)--(12,8);
  \draw[ultra thick] (2,0) rectangle (12,10);
  \draw (3,1) -- (7,5);
  \draw[loosely dashed] (7,5) -- (9,9);
  \draw[loosely dashed] (7,5) -- (11,7);
 \end{tikzpicture}$$
Through brute force, we find that $c_x = 36$. Thus $w$ is the minimal element of $2^2\cdot 36 = 144$ parabolic double cosets in $S_7$.
\end{example}

We next consider interplay between rafts.

\begin{defn}
When multiple rafts are connected by tethers of the same orientation, the resulting structure is a (horizontal or vertical) \emph{flotilla}.
\end{defn}

\begin{example}\label{ex:biggun}
Let
$$w = 1\ 3\ 4\ 5\ 7\ 8\ 2\ 6\ 14\ 15\ 16\ 9\ 10\ 11\ 12\ 13 \in S_{16}.$$
We see the balls-in-boxes picture of the marine model for $w$ in Figure~\ref{fig:marine model overlaid on big example}. We have $\rafts(w) = \{[2,3],[5,5],[9,10],[12,15]\}$, $\tethers(w)=\{8',4\}$, $\ropes(w)=\{5',1,8\}$, and $\floats(w)=\{1',7\}$.

Label the rafts $A$, $B$, $C$, and $D$ from left to right in Figure~\ref{fig:marine model overlaid on big example}. There is a horizontal flotilla consisting of rafts $B$ and $D$, and a vertical flotilla consisting of rafts $A$ and $B$.
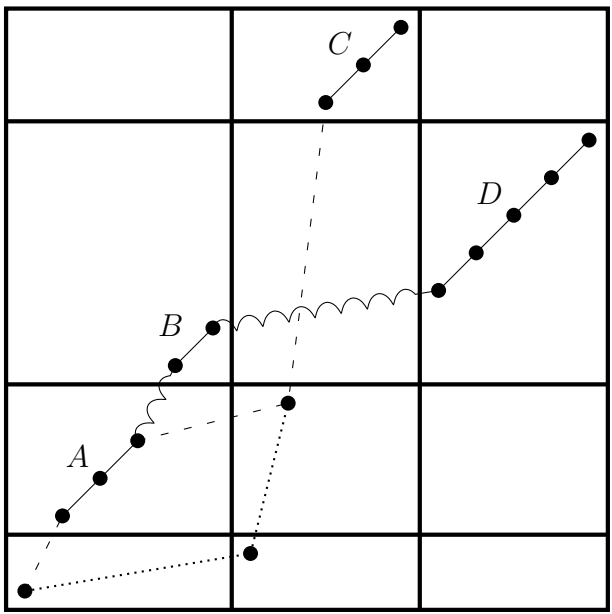
\begin{figure}[htbp]
\begin{tikzpicture}[scale=.25]
\draw[ultra thick] (0,0) rectangle (32,32);
\draw (1,1) node[circle,fill=black,inner sep =2] {};
\draw (3,5) node[circle,fill=black,inner sep =2] {};
\draw (5,7) node[circle,fill=black,inner sep =2] {};
\draw (7,9) node[circle,fill=black,inner sep =2] {};
\draw (9,13) node[circle,fill=black,inner sep =2] {};
\draw (11,15) node[circle,fill=black,inner sep =2] {};
\draw (13,3) node[circle,fill=black,inner sep =2] {};
\draw (15,11) node[circle,fill=black,inner sep =2] {};
\draw (17,27) node[circle,fill=black,inner sep =2] {};
\draw (19,29) node[circle,fill=black,inner sep =2] {};
\draw (21,31) node[circle,fill=black,inner sep =2] {};
\draw (23,17) node[circle,fill=black,inner sep =2] {};
\draw (25,19) node[circle,fill=black,inner sep =2] {};
\draw (27,21) node[circle,fill=black,inner sep =2] {};
\draw (29,23) node[circle,fill=black,inner sep =2] {};
\draw (31,25) node[circle,fill=black,inner sep =2] {};
\draw (3,5)-- node[midway,above left] {$A$} (7,9);
\draw (9,13)-- node[midway,above left] {$B$} (11,15);
\draw (17,27)-- node[midway,above left] {$C$} (21,31);
\draw (23,17)-- node[midway,above left] {$D$} (31,25);
\draw[snake=coil] (7,9)--(9,13);
\draw[snake=coil] (11,15)--(23,17);
\draw[loosely dashed] (1,1)--(3,5);
\draw[loosely dashed] (7,9)--(15,11)--(17,27);
\draw[thick,dotted] (1,1)--(13,3);
\draw[thick,dotted] (13,3)--(15,11);
\draw[ultra thick] (12,0)--(12,32);
\draw[ultra thick] (22,0)--(22,32);
\draw[ultra thick] (0,4)--(32,4);
\draw[ultra thick] (0,12)--(32,12);
\draw[ultra thick] (0,26)--(32,26);
\end{tikzpicture}
\caption{The permutation $1\ 3\ 4\ 5\ 7\ 8\ 2\ 6\ 14\ 15\ 16\ 9\ 10\ 11\ 12\ 13$ with its (four) rafts, (one right and one left) tethers, (two right and one left) ropes, and (one right and one left) floats marked.}\label{fig:marine model overlaid on big example}
\end{figure}
\end{example}

We now examine two adjacent rafts in a flotilla, and the tether connecting them. For example, consider the rafts $A$ and $B$ in Figure~\ref{fig:marine model overlaid on big example}. In the maximal representative of the corresponding parabolic double coset, the balls of raft $B$ will be sorted above and to the left of the balls of raft $A$. On the other hand, if a vertical wall is inserted between these two rafts, as depicted in Figure~\ref{fig:bigex-red}, then the balls of raft $B$ will appear above and to the right of the balls in raft $A$ in the maximal representative of the corresponding parabolic double coset. Thus the two parabolic double cosets are distinct.

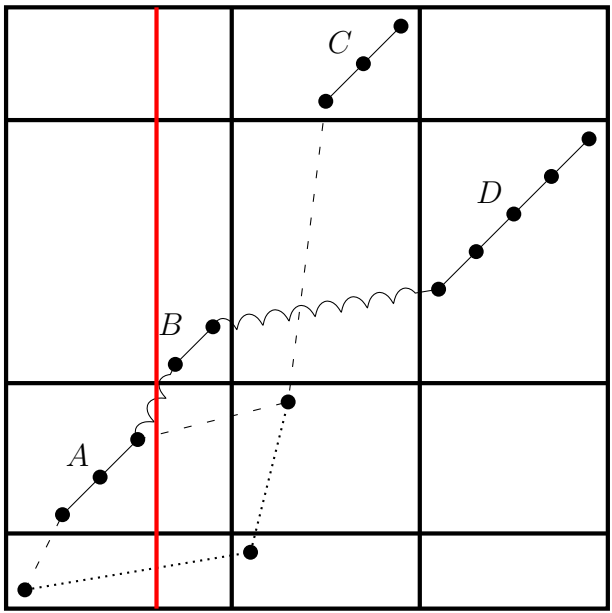
\begin{figure}[htbp]
\begin{tikzpicture}[scale=.25]
\draw[ultra thick] (0,0) rectangle (32,32);
\draw (1,1) node[circle,fill=black,inner sep =2] {};
\draw (3,5) node[circle,fill=black,inner sep =2] {};
\draw (5,7) node[circle,fill=black,inner sep =2] {};
\draw (7,9) node[circle,fill=black,inner sep =2] {};
\draw (9,13) node[circle,fill=black,inner sep =2] {};
\draw (11,15) node[circle,fill=black,inner sep =2] {};
\draw (13,3) node[circle,fill=black,inner sep =2] {};
\draw (15,11) node[circle,fill=black,inner sep =2] {};
\draw (17,27) node[circle,fill=black,inner sep =2] {};
\draw (19,29) node[circle,fill=black,inner sep =2] {};
\draw (21,31) node[circle,fill=black,inner sep =2] {};
\draw (23,17) node[circle,fill=black,inner sep =2] {};
\draw (25,19) node[circle,fill=black,inner sep =2] {};
\draw (27,21) node[circle,fill=black,inner sep =2] {};
\draw (29,23) node[circle,fill=black,inner sep =2] {};
\draw (31,25) node[circle,fill=black,inner sep =2] {};
\draw (3,5)-- node[midway,above left] {$A$} (7,9);
\draw (9,13)-- node[midway,above left] {$B$} (11,15);
\draw (17,27)-- node[midway,above left] {$C$} (21,31);
\draw (23,17)-- node[midway,above left] {$D$} (31,25);
\draw[snake=coil] (7,9)--(9,13);
\draw[snake=coil] (11,15)--(23,17);
\draw[loosely dashed] (1,1)--(3,5);
\draw[loosely dashed] (7,9)--(15,11)--(17,27);
\draw[thick,dotted] (1,1)--(13,3);
\draw[thick,dotted] (13,3)--(15,11);
\draw[ultra thick] (12,0)--(12,32);
\draw[ultra thick] (22,0)--(22,32);
\draw[ultra thick] (0,4)--(32,4);
\draw[ultra thick] (0,12)--(32,12);
\draw[ultra thick] (0,26)--(32,26);
\draw[ultra thick,red] (8,0)--(8,32);
\end{tikzpicture}
\caption{Inserting a vertical wall (marked in red) through the vertical tether between rafts $A$ and $B$ in Figure~\ref{fig:marine model overlaid on big example}.}\label{fig:bigex-red}
\end{figure}

In fact, the same argument can be used to show this phenomenon holds generally.

\begin{lem}\label{lem:tethers}
If two rafts are connected by a left (respectively, right) tether, then inserting a horizontal (respectively, vertical) wall in that position will result in a different parabolic double coset, independent of all other choices for the walls.
\end{lem}

Despite the similarities between Lemmas~\ref{lem:floats} and~\ref{lem:tethers}, the contribution from the rafts can be different depending on the subset of tethers that are cut by walls. This is because each tether is incident to two rafts, whereas a float is incident to none. This gives the enumeration of $c_w$ the form:
\[
  c_w = 2^{|\floats(w)|}\sum_{ T \subseteq \tethers(w)} (\mbox{something depending on $T$}),
\]
where the ``something" will look like a product of contributions from the rafts.

The only piece of the marine model that we have yet to discuss is a rope. Ropes
occur at the end of a raft. These are better behaved, in the sense that
choosing whether or not to cut a rope only affects one raft's boundary.

\subsection{The Coxeter-theoretic picture}\label{sub:a.sequences}

We step briefly away from enumeration to expand our combinatorial model.

Foreshadowing our treatment of parabolic double cosets in other Coxeter groups, it will be useful to be able to represent the marine model in terms of the simple reflections that generate $S_n$. This more general visual representation will be based on the Coxeter graph $G$ of the group, which, in the case of the symmetric group $S_n$, is a path of vertices labeled $\{1,\ldots, n-1\}$. We now revisit the objects in Definition~\ref{defn:marine model for S_n}.

\begin{defn}\label{defn:marine model pictures}
To a permutation $w \in S_n$, we associate a diagram called
the \emph{$w$-ocean}, formed as follows.
\begin{itemize}
\item Draw two rows of $n-1$ vertices, represented as open dots. We will think of these as being labeled $1,2,\ldots, n-1$ from
  left to right, representing two copies of the set of adjacent transpositions that generates $S_n$.
\item Cross out each dot in the top (respectively, bottom) row that corresponds to a
  right (respectively, left) descent of $w$.  Thus the remaining dots correspond to
  left or right ascents of $w$.
\item Circle each dot in the top (respectively, bottom) row that corresponds to a large
  right (respectively, left) ascent in $w$.
\item For each $k$ which appears in a raft of $w$, draw a line from the $k$th dot in the top row to the $w(k)$th dot in the bottom row. Such lines are \emph{planks}.
\item Draw horizontal lines connecting consecutive small ascents.
\item If the $i$th dot in the top row is a right rope or tether, then draw an edge(s) horizontally from it to its adjacent small ascent(s). Do the same in the bottom row for the left ropes and tethers.
\end{itemize}
\end{defn}
In the $w$-ocean, we can represent a pair $(I,J)$, where $I \subset \asc_L(w)$
and $J \subset \asc_R(w)$, by filling in the corresponding open dots. More
precisely, elements of $I$ are filled in the bottom row, and elements of $J$
are filled in the top row. Such fillings will be the basis for our enumeration
of lex-minimal pairs for $w$.

As we mentioned at the end of Section \ref{sec:startcounting}, the enumeration
of lex-minimal fillings reduces to individual rafts. For each raft, the number
of lex-minimal fillings is controlled by the fillings of adjacent ropes and tethers.
The only nodes of a raft which can be adjacent to a rope or tether are those nodes
in the boundary of the raft. Hence we make the following (somewhat informal) definition:
\begin{defn}\label{defn:apparatus}
    A \emph{boundary apparatus} for a particular raft consists of the arrangement of
    ropes and tethers adjacent to the raft, along with a choice of fillings for these
    ropes and tethers.
\end{defn}

\begin{example}\label{ex:7123546.pt5}
Continuing Example~\ref{ex:7123546.pt4}, the $w$-ocean for $w=7123546$ is shown below. There is a raft of size
2, with one rope on each row on the right-hand side. There are also two floats.
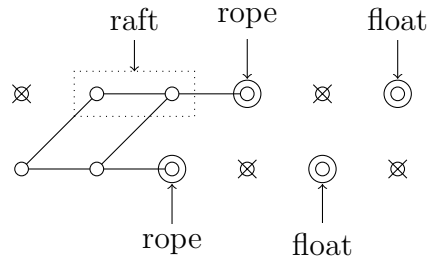
\[
\begin{tikzpicture}[scale = 1]
 \draw  (2,1) -- (1,0)--(2,0)--(3,1) -- (2,1);
 \draw (6,1) circle (1ex);
 \draw (5,0) circle (1ex);
 \draw (3,0) circle (1ex);
 \draw (4,1) circle (1ex);
 \draw (4,1) -- (3,1); \draw (3,0) -- (2,0);
 \draw[->] (2.5,2) node[fill=white] {raft} -- (2.5,1.3);
 \draw[dotted] (1.7,.7) rectangle (3.3,1.3);
 \draw[->] (4,2) node[fill=white] {rope} -- (4,1.2);
 \draw[->] (6,2) node[fill=white] {float} -- (6,1.2);
 \draw[->] (3,-1) node[fill=white] {rope} -- (3,-.2);
 \draw[->] (5,-1) node[fill=white] {float} -- (5,-.2);
 \foreach \x in {1,...,6}
     {\draw[fill=white] (\x,1) circle (.5ex);\draw[fill=white] (\x,0) circle (.5ex);} 
 \draw (4,0) node[cross=.75ex]{};
 \draw (6,0) node[cross=.75ex]{};
 \draw (1,1) node[cross=.75ex]{};
 \draw (5,1) node[cross=.75ex]{};
\end{tikzpicture}
\]

\noindent
To represent the triple $(I,w,J)$ with $I=\{1,2\}$ and $J=\{3,4,6\}$,
we fill the appropriate dots in the $w$-ocean. In this case, the boundary apparatus
of the raft consists of the filled in rope attached to the upper right corner,
and the unfilled rope attached to the lower right corner.
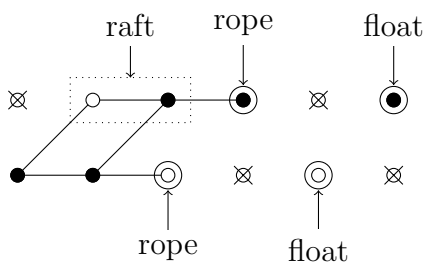
\[
\begin{tikzpicture}[scale = 1]
 \draw  (2,1) -- (1,0)--(2,0)--(3,1) -- (2,1);
 \draw (6,1) circle (1ex);
 \draw (5,0) circle (1ex);
 \draw (3,0) circle (1ex);
 \draw (4,1) circle (1ex);
 \draw (4,1) -- (3,1); \draw (3,0) -- (2,0);
\draw[->] (2.5,2) node[fill=white] {raft} -- (2.5,1.3);
 \draw[dotted] (1.7,.7) rectangle (3.3,1.3);
 \draw[->] (4,2) node[fill=white] {rope} -- (4,1.2);
 \draw[->] (6,2) node[fill=white] {float} -- (6,1.2);
 \draw[->] (3,-1) node[fill=white] {rope} -- (3,-.2);
 \draw[->] (5,-1) node[fill=white] {float} -- (5,-.2);
 \foreach \x in {1,...,6}
          {\draw[fill=white] (\x,1) circle (.5ex);\draw[fill=white] (\x,0) circle (.5ex);} 
 \draw[fill] (1,0) circle (.5ex);
 \draw[fill] (2,0) circle (.5ex);
 \draw[fill] (3,1) circle (.5ex);
 \draw[fill] (4,1) circle (.5ex);
 \draw[fill] (6,1) circle (.5ex);
 \draw (4,0) node[cross=.75ex]{};
 \draw (6,0) node[cross=.75ex]{};
 \draw (1,1) node[cross=.75ex]{};
 \draw (5,1) node[cross=.75ex]{};
\end{tikzpicture}
\]
\end{example}

To fully appreciate the marine model as represented in the $w$-ocean, we consider a larger example.

\begin{example}
Continuing Example~\ref{ex:biggun}, Figure~\ref{fig:marine-model picture example} shows the $w$-ocean for the permutation $w = 1\ 3\ 4\ 5\ 7\ 8\ 2\ 6\ 14\ 15\ 16\ 9\ 10\ 11\ 12\ 13 \in S_{16}$. Compare with Figure~\ref{fig:marine model overlaid on big example}.

\begin{figure}[htbp]
\begin{center}
\begin{tikzpicture}[scale = 1]
 \node at (1,1.92) {rope};
 \draw[->] (1,1.72) -- (1,1.2);
 \node at (2.5,2) {raft};
 \draw[dotted] (1.7,.7) rectangle (3.3,1.3);
 \draw[->] (2.5,1.8) -- (2.5,1.3);
 \node at (4,2) {tether};
 \draw[->] (4,1.8) -- (4,1.2);
 \node at (5,2) {raft};
 \draw[->] (5,1.8) -- (5,1.2);
 \node at (7,2) {float};
 \draw[->] (7,1.8) -- (7,1.2);
 \node at (8,1.92) {rope};
 \draw[->] (8,1.72) -- (8,1.2);
 \node at (9.5,2) {raft};
 \draw[dotted] (8.7,.7) rectangle (10.3,1.3);
 \draw[->] (9.5,1.8) -- (9.5,1.3);
 \node at (13.5,2) {raft};
 \draw[dotted] (11.7,.7) rectangle (15.3,1.3);
 \draw[->] (13.5,1.8) -- (13.5,1.3);
 \draw[->] (1,-1) node[fill=white] {float} -- (1,-.3);
 \draw[->] (5,-1) node[fill=white] {rope} -- (5,-.2);
 \draw[->] (8,-1) node[fill=white] {tether} -- (8,-.2);
 \draw  (2,1) -- (3,0); \draw (3,1) -- (4,0); \draw (2,1) -- (3,1); \draw (3,0) -- (4,0);
 \draw (5,1) -- (7,0);
 \draw (9,1) -- (14,0); \draw (10,1) -- (15,0); \draw (9,1) -- (10,1); \draw (14,0) -- (15,0);
 \draw (12,1) -- (9,0); \draw (13,1) -- (10,0); \draw (14,1) -- (11,0); \draw (15,1) -- (12,0); \draw (12,1) -- (15,1); \draw (9,0) -- (12,0);
 \draw (3,1)--(5,1);
 \draw (7,0)--(9,0);
\draw (8,0) circle (1ex);
\draw (4,1) circle (1ex);
 \draw (7,1) circle (1ex);
 \draw (1,0) circle (1ex);
 \draw (1,1) -- (2,1); \draw (4,0) -- (5,0);
 \draw (1,1) circle (1ex);  \draw (8,1) circle (1ex);  \draw (5,0) circle (1ex);
 \draw (8,1) -- (9,1);
 \foreach \x in {1,...,15}
     {\draw[fill=white] (\x,1) circle (.5ex);\draw[fill=white] (\x,0) circle (.5ex);}
 \draw (2,0) node[cross=.75ex]{};
 \draw (6,0) node[cross=.75ex]{};
 \draw (13,0) node[cross=.75ex]{};
 \draw (6,1) node[cross=.75ex]{};
 \draw (11,1) node[cross=.75ex]{};
\end{tikzpicture}
\end{center}
\caption{The $w$-ocean of rafts, tethers, floats, and ropes for the permutation $w = 1\ 3\ 4\ 5\ 7\ 8\ 2\ 6\ 14\ 15\ 16\ 9\ 10\ 11\ 12\ 13 \in S_{16}$.} \label{fig:marine-model picture example}
\end{figure}
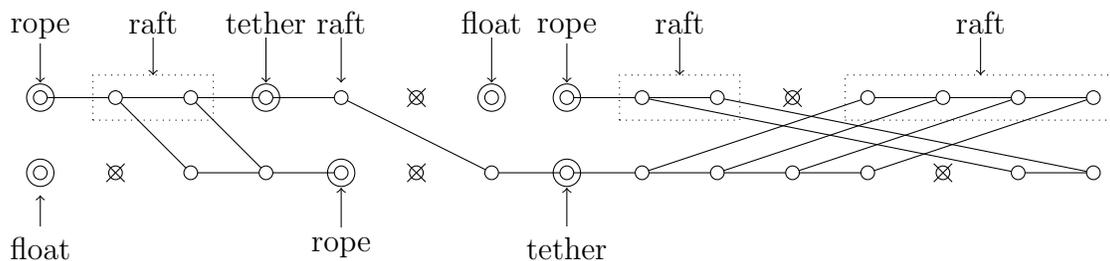
\end{example}

To enumerate the lex-minimal presentations of parabolic double cosets with a
fixed $w \in S_n$ as the minimal element, we want to count the ways of choosing
subsets $I$ and $J$ from the lower and upper row of the $w$-ocean,
respectively, such that the conditions in Theorem~\ref{T:minimalsym} are
satisfied.
\begin{defn}
    Let $\p R$ be a raft in a $w$-ocean. A \emph{lex-minimal filling} of
    $\p R$ is a lex-minimal filling of the $w$-ocean such that the only
    filled vertices belong either to $\p R$, or to the ropes and tethers
    adjacent to $\p R$.
\end{defn}
It should be clear that the number of lex-minimal fillings of a raft does not
depend on the full permutation $w$, only on the length of the raft and the
adjacent ropes and tethers.

For the next lemma, we need one more temporary definition. Let $(I,J)$ be a
filling of the $w$-ocean. The \emph{restriction} of $(I,J)$ to a raft $\p R$ is
a new filling in which all nodes outside of $\p R$ and its boundary apparatus
are left unfilled (and the nodes in $\p R$ and its boundary apparatus are left
unchanged).
\begin{lem}\label{lem:restriction}
    A filling of the $w$-ocean is lex-minimal if and only if the restriction to
    every raft $\p R$ is a lex-minimal filling of $\p R$.
\end{lem}

We can now prove the first version of our main enumeration formula, leaving the
number of lex-minimal fillings as a black box.

\begin{prop}\label{prop:raft.rope.tether.formula} For any $w\in S_n$,
\begin{equation}\label{eq:raft.formula}
  c_w = 2^{|\floats(w)|}\sum_{ T \subseteq \tethers(w)} \sum_{S
    \subseteq \ropes(w)} \prod_{\p R \in \rafts(w)} a(\p R,S,T)
\end{equation}
where $a(\p R,S,T)$ is the number of lex-minimal fillings of raft $\p R$ with
the boundary apparatus determined by $S$ and $T$.
\end{prop}

\begin{proof}
As described above, we want to choose appropriate subsets $I$ and $J$ from the lower and upper row of the $w$-ocean, as required by Theorem~\ref{T:minimalsym}. We can think of such pairs $(I,J)$ as arising from an arbitrary choice of any of the floats, ropes, and
tethers for $w$ (that is, any large ascents), followed by an appropriate filling
of the vertices in the rafts.  Once the floats, ropes, and tethers are
specified, the vertices on each raft are filled in such a way that they satisfy
the lex-minimal conditions of Theorem~\ref{T:minimalsym}. By Lemma
\ref{lem:restriction}, the choices made on each raft are independent of the
choice made on other rafts.

\end{proof}

We have reduced the problem of calculating $c_w$ to
finding a way to compute $a(\p R,S,T)$ for a single raft $\p R$ with fixed
apparatus on its boundary.

\begin{example}\label{ex:7123546.pt6}
Recall the $w$-ocean for $w=7123546$ shown in Example~\ref{ex:7123546.pt5}, which
has a rope on each row. If we choose the upper rope, but not the lower rope,
then there are nine ways to fill in vertices on the raft with lex-minimal
presentation.  These nine filling are in bijection with the following pictures.
$$
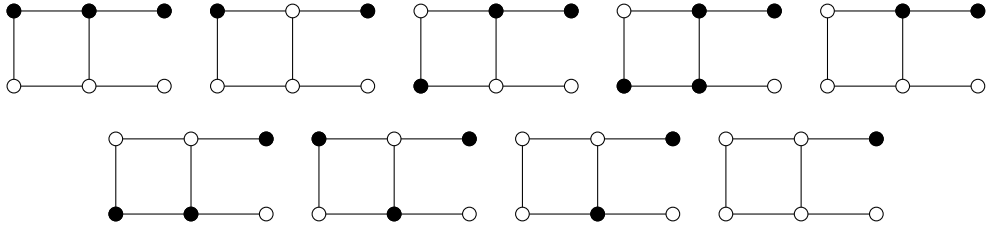
\begin{tikzpicture}
  \draw (1,1) -- (1,0);
  \draw (0,1) -- (0,0);
  \draw (0,0)--(2,0);
  \draw (0,1)--(2,1);
  \foreach \x in {0,...,2}
     {\draw[fill=white] (\x,1) circle (.5ex);\draw[fill=white] (\x,0) circle (.5ex);}
\draw[fill] (0,1) circle (.5ex);
\draw[fill] (1,1) circle (.5ex);
\draw[fill] (2,1) circle (.5ex);
\end{tikzpicture}\hspace{.2in}\begin{tikzpicture}
  \draw (0,0)--(2,0);
  \draw (0,1)--(2,1);
  \draw (1,1) -- (1,0);
  \draw (0,1) -- (0,0);
  \foreach \x in {0,...,2}
     {\draw[fill=white] (\x,1) circle (.5ex);\draw[fill=white] (\x,0) circle (.5ex);}
\draw[fill] (0,1) circle (.5ex);
\draw[fill] (2,1) circle (.5ex);
\end{tikzpicture}\hspace{.2in}\begin{tikzpicture}
  \draw (0,0)--(2,0);
  \draw (0,1)--(2,1);
  \draw (1,1) -- (1,0);
  \draw (0,1) -- (0,0);
  \foreach \x in {0,...,2}
     {\draw[fill=white] (\x,1) circle (.5ex);\draw[fill=white] (\x,0) circle (.5ex);}
\draw[fill] (0,0) circle (.5ex);
\draw[fill] (1,1) circle (.5ex);
\draw[fill] (2,1) circle (.5ex);
\end{tikzpicture}\hspace{.2in}\begin{tikzpicture}
  \draw (0,0)--(2,0);
  \draw (0,1)--(2,1);
  \draw (1,1) -- (1,0);
  \draw (0,1) -- (0,0);
  \foreach \x in {0,...,2}
     {\draw[fill=white] (\x,1) circle (.5ex);\draw[fill=white] (\x,0) circle (.5ex);}
\draw[fill] (0,0) circle (.5ex);
\draw[fill] (1,0) circle (.5ex);
\draw[fill] (1,1) circle (.5ex);
\draw[fill] (2,1) circle (.5ex);
\end{tikzpicture}\hspace{.2in}\begin{tikzpicture}
  \draw (0,0)--(2,0);
  \draw (0,1)--(2,1);
  \draw (1,1) -- (1,0);
  \draw (0,1) -- (0,0);
  \foreach \x in {0,...,2}
     {\draw[fill=white] (\x,1) circle (.5ex);\draw[fill=white] (\x,0) circle (.5ex);}
\draw[fill] (1,1) circle (.5ex);
\draw[fill] (2,1) circle (.5ex);
\end{tikzpicture}
$$
$$
\begin{tikzpicture}
  \draw (0,0)--(2,0);
  \draw (0,1)--(2,1);
  \draw (1,1) -- (1,0);
  \draw (0,1) -- (0,0);
  \foreach \x in {0,...,2}
     {\draw[fill=white] (\x,1) circle (.5ex);\draw[fill=white] (\x,0) circle (.5ex);}
\draw[fill] (0,0) circle (.5ex);
\draw[fill] (1,0) circle (.5ex);
\draw[fill] (2,1) circle (.5ex);
\end{tikzpicture}\hspace{.2in}\begin{tikzpicture}
  \draw (0,0)--(2,0);
  \draw (0,1)--(2,1);
  \draw (1,1) -- (1,0);
  \draw (0,1) -- (0,0);
  \foreach \x in {0,...,2}
     {\draw[fill=white] (\x,1) circle (.5ex);\draw[fill=white] (\x,0) circle (.5ex);}
\draw[fill] (1,0) circle (.5ex);
\draw[fill] (0,1) circle (.5ex);
\draw[fill] (2,1) circle (.5ex);
\end{tikzpicture}\hspace{.2in}\begin{tikzpicture}
  \draw (0,0)--(2,0);
  \draw (0,1)--(2,1);
  \draw (1,1) -- (1,0);
  \draw (0,1) -- (0,0);
  \foreach \x in {0,...,2}
     {\draw[fill=white] (\x,1) circle (.5ex);\draw[fill=white] (\x,0) circle (.5ex);}
\draw[fill] (1,0) circle (.5ex);
\draw[fill] (2,1) circle (.5ex);
\end{tikzpicture}\hspace{.2in}\begin{tikzpicture}
  \draw (0,0)--(2,0);
  \draw (0,1)--(2,1);
  \draw (1,1) -- (1,0);
  \draw (0,1) -- (0,0);
  \foreach \x in {0,...,2}
     {\draw[fill=white] (\x,1) circle (.5ex);\draw[fill=white] (\x,0) circle (.5ex);}
\draw[fill] (2,1) circle (.5ex);
\end{tikzpicture}
$$

This proves that $a(\p R, \{4\},\emptyset)=9$.  The reader is encouraged to
verify that $a(\p R, \{3'\},\emptyset)=9,$\ $ a(\p R, \{4,3'\},\emptyset) =
12,$ and $a(\p R, \emptyset,\emptyset) =6$ using the conditions in
Theorem~\ref{T:minimalsym}.  Thus,
\begin{align*}
c_w &= 2^2\big( a(\p R, \{4\},\emptyset) + a(\p R, \{3'\},\emptyset) + a(\p R,
\{4,3'\},\emptyset)+ a(\p R, \emptyset,\emptyset) \big)\\
&= 4 \cdot (9+9+12+6) \\
&=144.
\end{align*}
\end{example}

\subsection{Enumeration for rafts}\label{sec:Snenumeration}

Fix a raft $\p R$ and an accompanying boundary apparatus, composed of the selected ropes $S$ and tethers $T$.  To compute $a(\p R,S,T)$, we need to
consider all pairs of subsets $(I,J)$ of the lower and upper vertices of
the raft that satisfy the lex-minimal conditions in
Theorem~\ref{T:minimalsym}.  We will show that these lex-minimal
subsets can be recognized by a finite state automaton, pictured in
Figure~\ref{fig:aut}.  As a result, the transfer-matrix method can be
used to find a recurrence for the number of choices for $I$ and $J$,
depending on the size of the raft $\p R$ and the boundary apparatus.
Recall that the size of a raft is
equal to
the number of planks in the raft.

To make this idea precise, suppose that the tuple $(i_1,i_2,i_3,i_4)$
represents the apparatus attached to the lower-left, upper-left,
lower-right, and upper-right outer corners, respectively, of a raft as follows.
$$
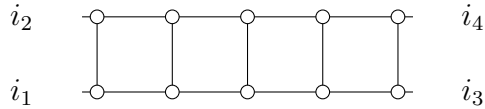
\begin{tikzpicture}
\draw (0,0) node[fill=white, inner sep=2] {$i_1$}; \draw (6,0) node[fill=white, inner sep=7] {$i_3$};
\draw (0,1) node[fill=white, inner sep=2] {$i_2$}; \draw (6,1) node[fill=white, inner sep=2] {$i_4$};
\draw (.8,0)--(5.2,0);
\draw (.8,1)--(5.2,1);
 \foreach \x in {1,...,5}
     {\draw (\x,1) -- (\x,0);\draw[fill=white] (\x,1) circle (.5ex);\draw[fill=white] (\x,0) circle (.5ex);} 
\end{tikzpicture}
$$
The indicators are $1$ (if selected and filled) or $0$ (if not selected
and not filled), depending on whether or not a rope or tether attached
at that point appears in $S$ or $T$.  In terms of the
balls-in-boxes picture, $i_1=0$ indicates the presence of a horizontal
wall immediately below a raft, $i_2=0$ indicates the presence of a
vertical wall immediately to the left of the raft, $i_3=0$ indicates the presence of
a horizontal wall immediately above, and $i_4=0$ indicates the presence
of a vertical wall immediately to the right.

Up to symmetry, there are seven cases:
\begin{itemize}
\item $i_1+i_2+i_3+i_4 = 0$, i.e., four walls,
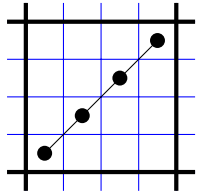
\[\begin{tikzpicture}[scale=.25]
  \draw[blue] (-1,-1) grid[step=2] (9,9);
  \draw (1,1) node[circle,fill=black,inner sep =2] {};
  \draw (3,3) node[circle,fill=black,inner sep =2] {};
  \draw (5,5) node[circle,fill=black,inner sep =2] {};
  \draw (7,7) node[circle,fill=black,inner sep =2] {};
  \draw (1,1) -- (7,7);
  \draw[ultra thick] (-1,0)--(9,0);
  \draw[ultra thick] (0,-1) -- (0,9);
  \draw[ultra thick] (-1,8)--(9,8);
  \draw[ultra thick] (8,-1) -- (8,9);
 \end{tikzpicture}
\]

\item $i_1+i_2+i_3+i_4 = 1$, i.e., three walls,
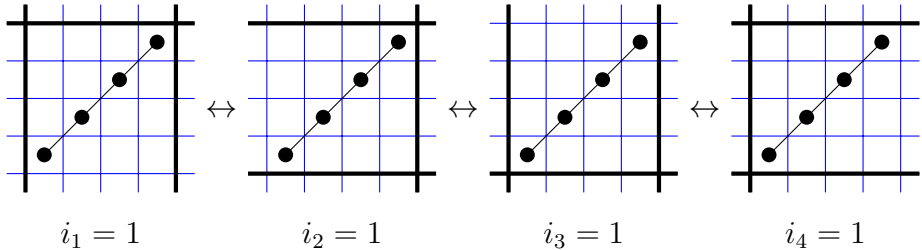
\[
\begin{tikzpicture}[scale=.25,baseline=.75cm]
  \draw[blue] (-1,-1) grid[step=2] (9,9);
  \draw (1,1) node[circle,fill=black,inner sep =2] {};
  \draw (3,3) node[circle,fill=black,inner sep =2] {};
  \draw (5,5) node[circle,fill=black,inner sep =2] {};
  \draw (7,7) node[circle,fill=black,inner sep =2] {};
  \draw (1,1) -- (7,7);
  \draw[ultra thick] (0,-1) -- (0,9);
  \draw[ultra thick] (-1,8)--(9,8);
  \draw[ultra thick] (8,-1) -- (8,9);
  \draw(4,-2) node[below] {$i_1=1$};
 \end{tikzpicture}
 \leftrightarrow
\begin{tikzpicture}[scale=.25,baseline=.75cm]
  \draw[blue] (-1,-1) grid[step=2] (9,9);
  \draw (1,1) node[circle,fill=black,inner sep =2] {};
  \draw (3,3) node[circle,fill=black,inner sep =2] {};
  \draw (5,5) node[circle,fill=black,inner sep =2] {};
  \draw (7,7) node[circle,fill=black,inner sep =2] {};
  \draw (1,1) -- (7,7);
  \draw[ultra thick] (-1,0)--(9,0);
  \draw[ultra thick] (-1,8)--(9,8);
  \draw[ultra thick] (8,-1) -- (8,9);
  \draw(4,-2) node[below] {$i_2=1$};
 \end{tikzpicture}
 \leftrightarrow
 \begin{tikzpicture}[scale=.25,baseline=.75cm]
  \draw[blue] (-1,-1) grid[step=2] (9,9);
  \draw (1,1) node[circle,fill=black,inner sep =2] {};
  \draw (3,3) node[circle,fill=black,inner sep =2] {};
  \draw (5,5) node[circle,fill=black,inner sep =2] {};
  \draw (7,7) node[circle,fill=black,inner sep =2] {};
  \draw (1,1) -- (7,7);
  \draw[ultra thick] (-1,0)--(9,0);
  \draw[ultra thick] (0,-1) -- (0,9);
  \draw[ultra thick] (8,-1) -- (8,9);
  \draw(4,-2) node[below] {$i_3=1$};
 \end{tikzpicture}
 \leftrightarrow
\begin{tikzpicture}[scale=.25,baseline=.75cm]
  \draw[blue] (-1,-1) grid[step=2] (9,9);
  \draw (1,1) node[circle,fill=black,inner sep =2] {};
  \draw (3,3) node[circle,fill=black,inner sep =2] {};
  \draw (5,5) node[circle,fill=black,inner sep =2] {};
  \draw (7,7) node[circle,fill=black,inner sep =2] {};
  \draw (1,1) -- (7,7);
  \draw[ultra thick] (-1,0)--(9,0);
  \draw[ultra thick] (0,-1) -- (0,9);
  \draw[ultra thick] (-1,8)--(9,8);
  \draw(4,-2) node[below] {$i_4=1$};
 \end{tikzpicture}
 \]
\item $i_1+i_2+i_3+i_4 = 3$, i.e., one wall,
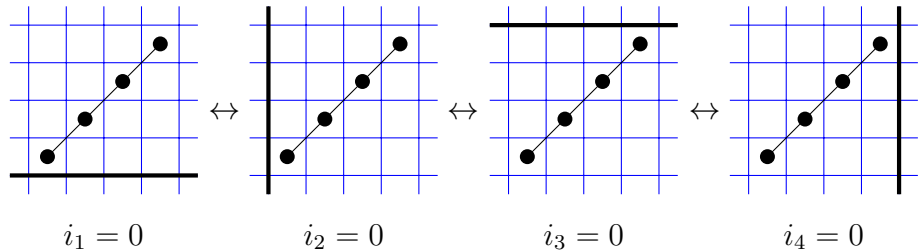
\[
\begin{tikzpicture}[scale=.25,baseline=.75cm]
  \draw[blue] (-1,-1) grid[step=2] (9,9);
  \draw (1,1) node[circle,fill=black,inner sep =2] {};
  \draw (3,3) node[circle,fill=black,inner sep =2] {};
  \draw (5,5) node[circle,fill=black,inner sep =2] {};
  \draw (7,7) node[circle,fill=black,inner sep =2] {};
  \draw (1,1) -- (7,7);
  \draw[ultra thick] (-1,0)--(9,0);
  \draw(4,-2) node[below] {$i_1=0$};
 \end{tikzpicture}
 \leftrightarrow
\begin{tikzpicture}[scale=.25,baseline=.75cm]
  \draw[blue] (-1,-1) grid[step=2] (9,9);
  \draw (1,1) node[circle,fill=black,inner sep =2] {};
  \draw (3,3) node[circle,fill=black,inner sep =2] {};
  \draw (5,5) node[circle,fill=black,inner sep =2] {};
  \draw (7,7) node[circle,fill=black,inner sep =2] {};
  \draw (1,1) -- (7,7);
  \draw[ultra thick] (0,-1) -- (0,9);
  \draw(4,-2) node[below] {$i_2=0$};
 \end{tikzpicture}
 \leftrightarrow
 \begin{tikzpicture}[scale=.25,baseline=.75cm]
  \draw[blue] (-1,-1) grid[step=2] (9,9);
  \draw (1,1) node[circle,fill=black,inner sep =2] {};
  \draw (3,3) node[circle,fill=black,inner sep =2] {};
  \draw (5,5) node[circle,fill=black,inner sep =2] {};
  \draw (7,7) node[circle,fill=black,inner sep =2] {};
  \draw (1,1) -- (7,7);
  \draw[ultra thick] (-1,8)--(9,8);
  \draw(4,-2) node[below] {$i_3=0$};
 \end{tikzpicture}
 \leftrightarrow
\begin{tikzpicture}[scale=.25,baseline=.75cm]
  \draw[blue] (-1,-1) grid[step=2] (9,9);
  \draw (1,1) node[circle,fill=black,inner sep =2] {};
  \draw (3,3) node[circle,fill=black,inner sep =2] {};
  \draw (5,5) node[circle,fill=black,inner sep =2] {};
  \draw (7,7) node[circle,fill=black,inner sep =2] {};
  \draw (1,1) -- (7,7);
  \draw[ultra thick] (8,-1) -- (8,9);
  \draw(4,-2) node[below] {$i_4=0$};
 \end{tikzpicture}
 \]

\item $i_1+i_2+i_3+i_4 = 4$, i.e., zero walls,
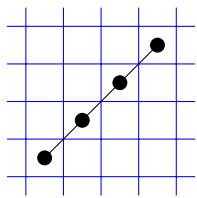
\[\begin{tikzpicture}[scale=.25]
  \draw[blue] (-1,-1) grid[step=2] (9,9);
  \draw (1,1) node[circle,fill=black,inner sep =2] {};
  \draw (3,3) node[circle,fill=black,inner sep =2] {};
  \draw (5,5) node[circle,fill=black,inner sep =2] {};
  \draw (7,7) node[circle,fill=black,inner sep =2] {};
  \draw (1,1) -- (7,7);
 \end{tikzpicture}
\]

\item and $i_1+i_2+i_3+i_4 = 2$, i.e., two walls, in three distinct ways.
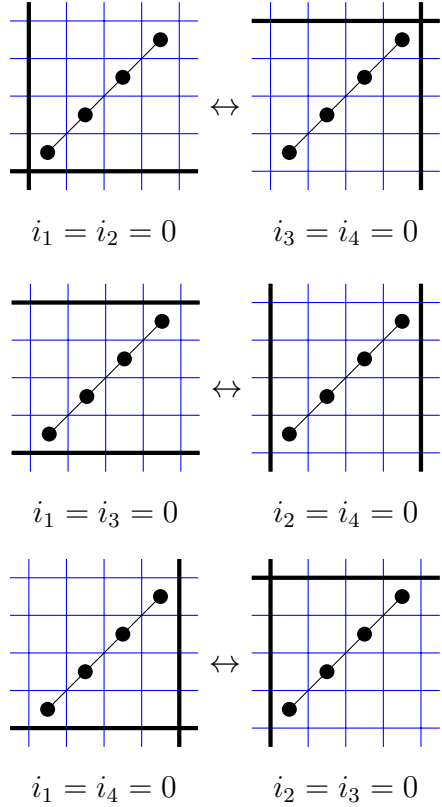
\[
\begin{tikzpicture}[scale=.25,baseline=.75cm]
  \draw[blue] (-1,-1) grid[step=2] (9,9);
  \draw (1,1) node[circle,fill=black,inner sep =2] {};
  \draw (3,3) node[circle,fill=black,inner sep =2] {};
  \draw (5,5) node[circle,fill=black,inner sep =2] {};
  \draw (7,7) node[circle,fill=black,inner sep =2] {};
  \draw (1,1) -- (7,7);
  \draw[ultra thick] (-1,0)--(9,0);
  \draw[ultra thick] (0,-1) -- (0,9);
  \draw(4,-2) node[below] {$i_1 = i_2 = 0$};
 \end{tikzpicture}
 \leftrightarrow
\begin{tikzpicture}[scale=.25,baseline=.75cm]
  \draw[blue] (-1,-1) grid[step=2] (9,9);
  \draw (1,1) node[circle,fill=black,inner sep =2] {};
  \draw (3,3) node[circle,fill=black,inner sep =2] {};
  \draw (5,5) node[circle,fill=black,inner sep =2] {};
  \draw (7,7) node[circle,fill=black,inner sep =2] {};
  \draw (1,1) -- (7,7);
  \draw[ultra thick] (-1,8)--(9,8);
  \draw[ultra thick] (8,-1) -- (8,9);
  \draw(4,-2) node[below] {$i_3 = i_4 =0$};
 \end{tikzpicture}
\]

\[
\begin{tikzpicture}[scale=.25,baseline=.75cm]
  \draw[blue] (-1,-1) grid[step=2] (9,9);
  \draw (1,1) node[circle,fill=black,inner sep =2] {};
  \draw (3,3) node[circle,fill=black,inner sep =2] {};
  \draw (5,5) node[circle,fill=black,inner sep =2] {};
  \draw (7,7) node[circle,fill=black,inner sep =2] {};
  \draw (1,1) -- (7,7);
  \draw[ultra thick] (-1,0)--(9,0);
  \draw[ultra thick] (-1,8)--(9,8);
  \draw(4,-2) node[below] {$i_1 = i_3 = 0$};
 \end{tikzpicture}
 \leftrightarrow
\begin{tikzpicture}[scale=.25,baseline=.75cm]
  \draw[blue] (-1,-1) grid[step=2] (9,9);
  \draw (1,1) node[circle,fill=black,inner sep =2] {};
  \draw (3,3) node[circle,fill=black,inner sep =2] {};
  \draw (5,5) node[circle,fill=black,inner sep =2] {};
  \draw (7,7) node[circle,fill=black,inner sep =2] {};
  \draw (1,1) -- (7,7);
  \draw[ultra thick] (0,-1) -- (0,9);
  \draw[ultra thick] (8,-1) -- (8,9);
  \draw(4,-2) node[below] {$i_2 = i_4 =0$};
 \end{tikzpicture}
\]

\[
\begin{tikzpicture}[scale=.25,baseline=.75cm]
  \draw[blue] (-1,-1) grid[step=2] (9,9);
  \draw (1,1) node[circle,fill=black,inner sep =2] {};
  \draw (3,3) node[circle,fill=black,inner sep =2] {};
  \draw (5,5) node[circle,fill=black,inner sep =2] {};
  \draw (7,7) node[circle,fill=black,inner sep =2] {};
  \draw (1,1) -- (7,7);
  \draw[ultra thick] (-1,0)--(9,0);
  \draw[ultra thick] (8,-1) -- (8,9);
  \draw(4,-2) node[below] {$i_1 = i_4 = 0$};
 \end{tikzpicture}
 \leftrightarrow
\begin{tikzpicture}[scale=.25,baseline=.75cm]
  \draw[blue] (-1,-1) grid[step=2] (9,9);
  \draw (1,1) node[circle,fill=black,inner sep =2] {};
  \draw (3,3) node[circle,fill=black,inner sep =2] {};
  \draw (5,5) node[circle,fill=black,inner sep =2] {};
  \draw (7,7) node[circle,fill=black,inner sep =2] {};
  \draw (1,1) -- (7,7);
  \draw[ultra thick] (0,-1) -- (0,9);
  \draw[ultra thick] (-1,8)--(9,8);
  \draw(4,-2) node[below] {$i_2 = i_3 =0$};
 \end{tikzpicture}
\]

\end{itemize}

These seven symmetry types can be encoded by the function
\begin{equation*}
    k(i_1,i_2,i_3,i_4)
        = \begin{cases}
         i_1 + i_2 + i_3 + i_4 & \mbox{if } i_1 + i_2 + i_3 + i_4 \neq 2, \\
                              2 & \mbox{if } i_1 + i_2 + i_3 + i_4 = 2 \mbox{ and } i_1 = i_2,\\
                                2' & \mbox{if }  i_1 + i_2 + i_3 + i_4 = 2 \mbox{ and } i_1 = i_3, \text{ and}\\
                                2'' & \mbox{if } i_1 + i_2 + i_3 + i_4 = 2 \mbox{ and } i_1 = i_4.
                               \end{cases}
\end{equation*}

In the following statement, the seven possibilities for
$k:=k(i_1,i_2,i_3,i_4)$ each determine an integer sequence denoted
$(a_m^k)_{m \in \N}$, which we call the \emph{$a$-sequences}.  These sequences are at the core of our enumerative results --- not just for symmetric groups, but for general Coxeter
groups as well.

\begin{thm}\label{thm:amk}
Consider a raft $\p R$ of size $m>0$ and its boundary apparatus,
composed of the selected ropes $S$ and tethers $T$. Let
$(i_1,i_2,i_3,i_4)$ represent this choice of apparatus attached to
the lower-left, upper-left, lower-right, and upper-right outer corners
of the raft, respectively, with $k:=k(i_1,i_2,i_3,i_4)$.  Then
$$a(\p R,S,T)=a_m^k,$$
where the family of sequences $(a_m^k)_{m
  \in \N}$, for $k \in \{0,1,2,2',2'',3,4\}$, are defined by the
recurrence
\begin{equation*}
    a_m  = 6 a_{m-1} - 13 a_{m-2} + 16 a_{m-3} - 11 a_{m-4} + 4 a_{m-5} \quad \mbox{for} \quad m \geq 5,
\end{equation*}
with initial conditions given in Table~\ref{table:a.table}.
\begin{table}[htbp]
  $$\begin{array}{r|ccccc}
   m=& 0 & 1 & 2 & 3 & 4\\
   \hline
   k=0\phantom{''} & 1 & 2 & 6 & 20 & 66 \\
   1\phantom{''} & 1 & 3 & 9 & 28 & 89 \\
   2\phantom{''} & 1 & 4 & 12 & 36 & 112 \\
   2' \phantom{'}& 1 & 3 & 11 & 37 & 119 \\
   2'' & 1 & 4 & 12 & 37 & 118 \\
   3\phantom{''} & 1 & 4 & 14 & 46 & 148 \\
   4\phantom{''} & 1 & 4 & 16 & 56 & 184
  \end{array}$$
\caption{Initial conditions for the $a$-sequences $a_m^k$.}\label{table:a.table}
\end{table}
\end{thm}

Before proving Theorem~\ref{thm:amk}, we pause for commentary and an example.

\begin{remark}\label{rem:factor.rec}
     The characteristic polynomial corresponding to the recurrence is
     $$1 - 6 t + 13t^2 - 16 t^3 + 11 t^4 - 4t^5 = (1 - t + t^2) (1 - 5 t + 7 t^2 - 4 t^3).$$
    In fact, the sequences $(a_m^k)_{m \in \N}$ for $k = 0,1,2,3,4$ (but not for $k = 2', 2''$)
    actually satisfy a recurrence of order $3$: $$a_m = 5 a_{m-1} - 7 a_{m-2}
    + 4 a_{m-3} \quad \mbox{for} \quad m \geq 3.$$
    For the sake of brevity, however, we opt for stating the result in terms of a single (higher order) recurrence.
    \end{remark}

To prove Theorem~\ref{thm:amk}, we will use a finite automaton to
recognize pairs $(I,J)$ that give \emph{lex-minimal fillings} for rafts with
a given boundary apparatus.  Our automaton will read the raft from left
to right, one vertical pair of nodes at a time, so the alphabet will
be the set of tiles

\begin{equation}\label{eqn:the alphabet A}
\mathcal{A} = \Bigg\{
\begin{tikzpicture}[scale=.5,baseline=.1cm]
\draw (-.5,-.5) rectangle (.5,1.5);
\draw (0,0) node [circle,fill=black, inner sep=2.5] {};
\draw (0,1) node [circle,fill=black, inner sep=2.5] {};
\end{tikzpicture}\, ,
\begin{tikzpicture}[scale=.5,baseline=.1cm]
\draw (-.5,-.5) rectangle (.5,1.5);
\draw (0,0) node [circle,draw=black,fill=white, inner sep=2.5] {};
\draw (0,1) node [circle,fill=black, inner sep=2.5] {};
\end{tikzpicture}\, ,
\begin{tikzpicture}[scale=.5,baseline=.1cm]
\draw (-.5,-.5) rectangle (.5,1.5);
\draw (0,0) node [circle,fill=black, inner sep=2.5] {};
\draw (0,1) node [circle,draw=black,fill=white, inner sep=2.5] {};
\end{tikzpicture}\, ,
\begin{tikzpicture}[scale=.5,baseline=.1cm]
\draw (-.5,-.5) rectangle (.5,1.5);
\draw (0,0) node [circle,draw=black,fill=white, inner sep=2.5] {};
\draw (0,1) node [circle,draw=black,fill=white, inner sep=2.5] {};
\end{tikzpicture}
\Bigg\}.
\end{equation}
The first tile represents the boundary apparatus on the left of the raft, and the
last tile represents the apparatus on the right.

\begin{example}
Consider a raft of length $m=7$ with no apparatus on
either side, and with $I=\{1,3,4,5\}$ and $J=\{2,3,4,5,6\}$. This is drawn
as:
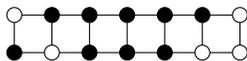
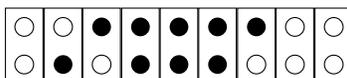
\[
\begin{tikzpicture}[scale=.5]
\draw (1,0)--(7,0);
\draw (1,1)--(7,1);
\foreach \x in {1,...,7} {
\draw (\x,0) node [circle,draw=black,fill=white,inner sep=2] {} -- (\x,1) node [circle,draw=black,fill=white,inner sep=2] {};
}
\draw (1,0) node [circle,fill=black, inner sep=2] {};
\draw (3,0) node [circle,fill=black, inner sep=2] {};
\draw (4,0) node [circle,fill=black, inner sep=2] {};
\draw (5,0) node [circle,fill=black, inner sep=2] {};
\draw (2,1) node [circle,fill=black, inner sep=2] {};
\draw (3,1) node [circle,fill=black, inner sep=2] {};
\draw (4,1) node [circle,fill=black, inner sep=2] {};
\draw (5,1) node [circle,fill=black, inner sep=2] {};
\draw (6,1) node [circle,fill=black, inner sep=2] {};
\end{tikzpicture}
\]

or
\[
\begin{tikzpicture}[scale=.5,baseline=.1cm]
\draw (-.5,-.5) rectangle (.5,1.5);
\draw (0,0) node [circle,draw=black,fill=white, inner sep=2.5] {};
\draw (0,1) node [circle,draw=black,fill=white, inner sep=2.5] {};
\end{tikzpicture}
\begin{tikzpicture}[scale=.5,baseline=.1cm]
\draw (-.5,-.5) rectangle (.5,1.5);
\draw (0,0) node [circle,fill=black, inner sep=2.5] {};
\draw (0,1) node [circle,draw=black,fill=white, inner sep=2.5] {};
\end{tikzpicture}
\begin{tikzpicture}[scale=.5,baseline=.1cm]
\draw (-.5,-.5) rectangle (.5,1.5);
\draw (0,0) node [circle,draw=black,fill=white, inner sep=2.5] {};
\draw (0,1) node [circle,fill=black, inner sep=2.5] {};
\end{tikzpicture}
\begin{tikzpicture}[scale=.5,baseline=.1cm]
\draw (-.5,-.5) rectangle (.5,1.5);
\draw (0,0) node [circle,fill=black, inner sep=2.5] {};
\draw (0,1) node [circle,fill=black, inner sep=2.5] {};
\end{tikzpicture}
\begin{tikzpicture}[scale=.5,baseline=.1cm]
\draw (-.5,-.5) rectangle (.5,1.5);
\draw (0,0) node [circle,fill=black, inner sep=2.5] {};
\draw (0,1) node [circle,fill=black, inner sep=2.5] {};
\end{tikzpicture}
\begin{tikzpicture}[scale=.5,baseline=.1cm]
\draw (-.5,-.5) rectangle (.5,1.5);
\draw (0,0) node [circle,fill=black, inner sep=2.5] {};
\draw (0,1) node [circle,fill=black, inner sep=2.5] {};
\end{tikzpicture}
\begin{tikzpicture}[scale=.5,baseline=.1cm]
\draw (-.5,-.5) rectangle (.5,1.5);
\draw (0,0) node [circle,draw=black,fill=white, inner sep=2.5] {};
\draw (0,1) node [circle,fill=black, inner sep=2.5] {};
\end{tikzpicture}
\begin{tikzpicture}[scale=.5,baseline=.1cm]
\draw (-.5,-.5) rectangle (.5,1.5);
\draw (0,0) node [circle,draw=black,fill=white, inner sep=2.5] {};
\draw (0,1) node [circle,draw=black,fill=white, inner sep=2.5] {};
\end{tikzpicture}
\begin{tikzpicture}[scale=.5,baseline=.1cm]
\draw (-.5,-.5) rectangle (.5,1.5);
\draw (0,0) node [circle,draw=black,fill=white, inner sep=2.5] {};
\draw (0,1) node [circle,draw=black,fill=white, inner sep=2.5] {};
\end{tikzpicture}
\]
where prefix and suffix tiles have been appended to indicate the (lack of an) apparatus.
Notice that this presentation is not lex-minimal, because the interval $[3,5] \subset I$ is a proper subset of $J$.
\end{example}

A priori, from Theorem~\ref{T:minimalsym}, we can realize any lex-minimal pair $(I,J)$ for a raft of size $m$ by a walk on the graph in Figure~\ref{fig:all.words}.    However, not every walk will correspond to a lex-minimal pair.

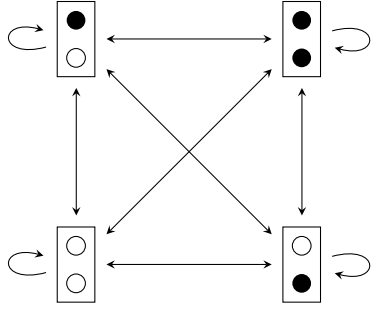
\begin{figure}[htbp]
\[
\begin{tikzpicture}[>=stealth]
 \draw (0,0) node (a) {
 \begin{tikzpicture}[scale=.5,baseline=.1cm]
\draw (-.5,-.5) rectangle (.5,1.5);
\draw (0,0) node [circle,draw=black,fill=white, inner sep=2.5] {};
\draw (0,1) node [circle,draw=black,fill=white, inner sep=2.5] {};
\end{tikzpicture}};
\draw (3,0) node (b) {
 \begin{tikzpicture}[scale=.5,baseline=.1cm]
\draw (-.5,-.5) rectangle (.5,1.5);
\draw (0,0) node [circle,fill=black, inner sep=2.5] {};
\draw (0,1) node [circle,draw=black,fill=white, inner sep=2.5] {};
\end{tikzpicture}};
\draw (0,3) node (c) {
 \begin{tikzpicture}[scale=.5,baseline=.1cm]
\draw (-.5,-.5) rectangle (.5,1.5);
\draw (0,0) node [circle,draw=black,fill=white, inner sep=2.5] {};
\draw (0,1) node [circle,fill=black, inner sep=2.5] {};
\end{tikzpicture}};
\draw (3,3) node (d) {
 \begin{tikzpicture}[scale=.5,baseline=.1cm]
\draw (-.5,-.5) rectangle (.5,1.5);
\draw (0,0) node [circle,fill=black, inner sep=2.5] {};
\draw (0,1) node [circle,fill=black, inner sep=2.5] {};
\end{tikzpicture}};
\draw[<->] (a) -- (b);
\draw[<->] (a) -- (c);
\draw[<->] (a) -- (d);
\draw[<->] (b) -- (c);
\draw[<->] (b)--(d);
\draw[<->] (c)--(d);
\path[->] (a) edge [loop left] (a);
\path[->] (b) edge [loop right] (b);
\path[->] (c) edge [loop left] (c);
\path[->] (d) edge [loop right] (d);
 \end{tikzpicture}
\]
\caption{The finite state automaton for all words on the alphabet $\mathcal{A}$ given in Equation~\eqref{eqn:the alphabet A}.}\label{fig:all.words}
\end{figure}

We are now ready to prove the recurrence relation for $a(\p R,S,T)$.

\begin{proof}[Proof of Theorem~\ref{thm:amk}]

Recall from Proposition~\ref{prop:raft.rope.tether.formula} that
$a(\p R,S,T)$ is the number of ways to fill the vertices in the upper and lower rows
of raft $\p R$ with the boundary apparatus determined by $S$ and $T$ in a
lex-minimal presentation.  Since $a(\p R,S,T)$ only depends on the
boundary apparatus in $S\cup T$ and the size of $\p R$, we will simplify
the notation by setting $a_m(s,t) := a(\p R,S,T)$
where $m$ is the size of $\p R$ and $s,t \in
\mathcal{A}$ are chosen accordingly.  Set $a_0(s,t) =1$ for all $s,t \in \mathcal{A}$.
This defines $4\times 4$ auxiliary sequences $a_m(s,t)$ for $m\geq0$.

Theorem~\ref{T:minimalsym} gives local conditions whose avoidance
characterizes lex-minimal pairs. There are three types of
local configurations to avoid, shown in Figure~\ref{fig:avoid}. In each type, the ellipsis represents arbitrarily long repetition of the adjacent tile. Types
(i) and (ii) correspond to part (a) of Theorem~\ref{T:minimalsym},
where $\{a,\ldots,b\}$ is an interval of small left ascents in $I$,
with neither $a-1$ nor $b+1$ in $I$. Type (iii) corresponds to part
(b), in which $\{a,\ldots,b\}$ is an interval of small right ascents
in $J$, with neither $a-1$ nor $b+1$ in $J$. In the pictures for Types (i) and (iii)
of Figure~\ref{fig:avoid}, the pattern is forbidden whether or not the nodes marked ``$\star$" are filled.

\begin{figure}[htbp]
\[
\begin{array}{|c| c| c|}
\hline
\begin{tikzpicture}[scale=.5]
\draw (0.5,0) node {\begin{tikzpicture}[scale=.5,baseline=.1cm]
\draw (-.5,-.5) rectangle (.5,1.5);
\draw (0,0) node [circle,draw=black,fill=white, inner sep=2.5] {};
\draw (0,1) node {$\star$};
\end{tikzpicture}
};
\draw (1.5,0) node {\begin{tikzpicture}[scale=.5,baseline=.1cm]
\draw (-.5,-.5) rectangle (.5,1.5);
\draw (0,0) node [circle,fill=black, inner sep=2.5] {};
\draw (0,1) node [circle,fill=black, inner sep=2.5] {};
\end{tikzpicture}
};
\draw (4.5,0) node {
\begin{tikzpicture}[scale=.5,baseline=.1cm]
\draw (-.5,-.5) rectangle (.5,1.5);
\draw (0,0) node [circle,fill=black, inner sep=2.5] {};
\draw (0,1) node [circle,fill=black, inner sep=2.5] {};
\end{tikzpicture}
};
\draw (5.5,0) node {\begin{tikzpicture}[scale=.5,baseline=.1cm]
\draw (-.5,-.5) rectangle (.5,1.5);
\draw (0,0) node [circle,draw=black,fill=white, inner sep=2.5] {};
\draw (0,1) node {$\star$};
\end{tikzpicture}
};
\draw (3,0) node {$\cdots$};
\end{tikzpicture}
&
\begin{tikzpicture}[scale=.5]
\draw (0.5,0) node {\begin{tikzpicture}[scale=.5,baseline=.1cm]
\draw (-.5,-.5) rectangle (.5,1.5);
\draw (0,0) node [circle,draw=black,fill=white, inner sep=2.5] {};
\draw (0,1) node [circle,draw=black,fill=white, inner sep=2.5] {};
\end{tikzpicture}
};
\draw (1.5,0) node {\begin{tikzpicture}[scale=.5,baseline=.1cm]
\draw (-.5,-.5) rectangle (.5,1.5);
\draw (0,0) node [circle,fill=black, inner sep=2.5] {};
\draw (0,1) node [circle,draw=black,fill=white, inner sep=2.5] {};
\end{tikzpicture}
};
\draw (4.5,0) node {
\begin{tikzpicture}[scale=.5,baseline=.1cm]
\draw (-.5,-.5) rectangle (.5,1.5);
\draw (0,0) node [circle,fill=black, inner sep=2.5] {};
\draw (0,1) node [circle,draw=black,fill=white, inner sep=2.5] {};
\end{tikzpicture}
};
\draw (5.5,0) node {\begin{tikzpicture}[scale=.5,baseline=.1cm]
\draw (-.5,-.5) rectangle (.5,1.5);
\draw (0,0) node [circle,draw=black,fill=white, inner sep=2.5] {};
\draw (0,1) node [circle,draw=black,fill=white, inner sep=2.5] {};
\end{tikzpicture}
};
\draw (3,0) node {$\cdots$};
\end{tikzpicture}
&
\begin{tikzpicture}[scale=.5]
\draw (0.5,0) node {\begin{tikzpicture}[scale=.5,baseline=.1cm]
\draw (-.5,-.5) rectangle (.5,1.5);
\draw (0,1) node [circle,draw=black,fill=white, inner sep=2.5] {};
\draw (0,0) node {$\star$};
\end{tikzpicture}
};
\draw (1.5,0) node {\begin{tikzpicture}[scale=.5,baseline=.1cm]
\draw (-.5,-.5) rectangle (.5,1.5);
\draw (0,0) node [circle,fill=black, inner sep=2.5] {};
\draw (0,1) node [circle,fill=black, inner sep=2.5] {};
\end{tikzpicture}
};
\draw (4.5,0) node {
\begin{tikzpicture}[scale=.5,baseline=.1cm]
\draw (-.5,-.5) rectangle (.5,1.5);
\draw (0,0) node [circle,fill=black, inner sep=2.5] {};
\draw (0,1) node [circle,fill=black, inner sep=2.5] {};
\end{tikzpicture}
};
\draw (5.5,0) node {\begin{tikzpicture}[scale=.5,baseline=.1cm]
\draw (-.5,-.5) rectangle (.5,1.5);
\draw (0,1) node [circle,draw=black,fill=white, inner sep=2.5] {};
\draw (0,0) node {$\star$};
\end{tikzpicture}
};
\draw (3,0) node {$\cdots$};
\end{tikzpicture} \\
(\mbox{i}) & (\mbox{ii}) & (\mbox{iii}) \\
\hline
\end{array}
\]
\caption{Patterns that are forbidden for lex-minimal presentations.}\label{fig:avoid}
\end{figure}
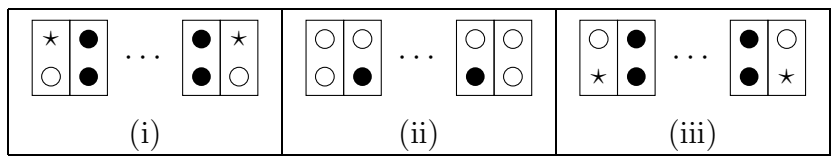

In order to construct an automaton for the allowed configurations, we
need a more refined graph (as opposed to the graph in Figure~\ref{fig:all.words}) labeled by the tiles of $\mathcal{A}$.  The first and last tiles can be any one of the four tiles in $\mathcal{A}$
corresponding to the boundary apparatus.  The tiles in between these two must not introduce any forbidden configurations. This involves some tedious case analysis, but the result is given in
Figure~\ref{fig:aut}. While there are still only four tile types, there are
eight states, reflecting the need for allowed walks to avoid
introducing any of the patterns shown in Figure
\ref{fig:avoid}.

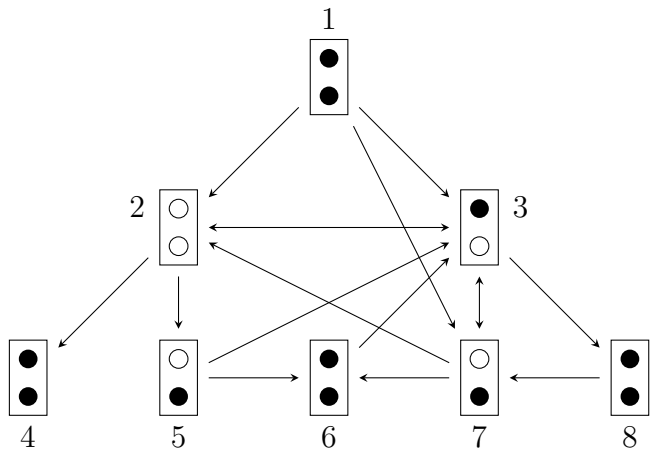
\begin{figure}[htbp]
\begin{tikzpicture}[>=stealth]
\draw (0,0) node[fill=white] (a1) {
\begin{tikzpicture}[scale=.5]
\draw (-.5,-.5) rectangle (.5,1.5);
\draw (0,0) node [circle,fill=black, inner sep=2.5] {};
\draw (0,1) node [circle,fill=black, inner sep=2.5] {};
\end{tikzpicture}
};
\draw (2,0) node[fill=white] (a2) {
\begin{tikzpicture}[scale=.5]
\draw (-.5,-.5) rectangle (.5,1.5);
\draw (0,0) node [circle,fill=black, inner sep=2.5] {};
\draw (0,1) node [circle,draw=black,fill=white, inner sep=2.5] {};
\end{tikzpicture}
};
\draw (4,0) node[fill=white] (a3) {
\begin{tikzpicture}[scale=.5]
\draw (-.5,-.5) rectangle (.5,1.5);
\draw (0,0) node [circle,fill=black, inner sep=2.5] {};
\draw (0,1) node [circle,fill=black, inner sep=2.5] {};
\end{tikzpicture}
};
\draw (6,0) node[fill=white] (a4) {
\begin{tikzpicture}[scale=.5]
\draw (-.5,-.5) rectangle (.5,1.5);
\draw (0,0) node [circle,fill=black, inner sep=2.5] {};
\draw (0,1) node [circle,draw=black,fill=white, inner sep=2.5] {};
\end{tikzpicture}
};
\draw (8,0) node[fill=white] (a5) {
\begin{tikzpicture}[scale=.5]
\draw (-.5,-.5) rectangle (.5,1.5);
\draw (0,0) node [circle,fill=black, inner sep=2.5] {};
\draw (0,1) node [circle,fill=black, inner sep=2.5] {};
\end{tikzpicture}
};
\draw (2,2) node[fill=white] (b1) {
\begin{tikzpicture}[scale=.5]
\draw (-.5,-.5) rectangle (.5,1.5);
\draw (0,0) node [circle,draw=black,fill=white, inner sep=2.5] {};
\draw (0,1) node [circle,draw=black,fill=white, inner sep=2.5] {};
\end{tikzpicture}
};
\draw (6,2) node[fill=white] (b2) {
\begin{tikzpicture}[scale=.5]
\draw (-.5,-.5) rectangle (.5,1.5);
\draw (0,0) node [circle,draw=black,fill=white, inner sep=2.5] {};
\draw (0,1) node [circle,fill=black, inner sep=2.5] {};
\end{tikzpicture}
};
\draw (4,4) node[fill=white] (c) {
\begin{tikzpicture}[scale=.5]
\draw (-.5,-.5) rectangle (.5,1.5);
\draw (0,0) node [circle,fill=black, inner sep=2.5] {};
\draw (0,1) node [circle,fill=black, inner sep=2.5] {};
\end{tikzpicture}
};
\draw[->] (c) -- (b1);
\draw[->] (c) -- (b2);
\draw[->] (c) -- (a4);
\draw[<->] (b1) -- (b2);
\draw[->] (b1) -- (a1);
\draw[->] (b1) -- (a2);
\draw[->] (b2) -- (a5);
\draw[<->] (b2) -- (a4);
\draw[->] (a2) -- (b2);
\draw[->] (a2) -- (a3);
\draw[->] (a3) -- (b2);
\draw[->] (a4) -- (b1);
\draw[->] (a4) -- (a3);
\draw[->] (a5) -- (a4);
\draw (c) node[above,yshift=5mm] {$1$};
\draw (b1) node[above left,xshift=-3mm] {$2$};
\draw (b2) node[above right,xshift=3mm] {$3$};
\draw (a1) node[below,yshift=-5mm] {$4$};
\draw (a2) node[below,yshift=-5mm] {$5$};
\draw (a3) node[below,yshift=-5mm] {$6$};
\draw (a4) node[below,yshift=-5mm] {$7$};
\draw (a5) node[below,yshift=-5mm] {$8$};
\end{tikzpicture}
\caption{The finite automaton for lex-minimal presentations for a raft. Loops are allowed at each node, and have only been omitted for the sake of readability.}\label{fig:aut}
\end{figure}

Having built the automaton, enumerating walks in this graph is now a straightforward application of the transfer matrix method. See \cite[Section 4.7]{ec1}. The matrix of adjacencies for the automaton is
\[
 A = \left[
 \begin{array}{c ccccccc}
 1 & 1 & 1 & 0 & 0 & 0 & 1 & 0 \\
 0 & 1 & 1 & 1 & 1 & 0 & 0 & 0 \\
 0 & 1 & 1 & 0 & 0 & 0 & 1 & 1 \\
 0 & 0 & 0 & 1 & 0 & 0 & 0 & 0 \\
 0 & 0 & 1 & 0 & 1 & 1 & 0 & 0 \\
 0 & 0 & 1 & 0 & 0 & 1 & 0 & 0 \\
 0 & 1 & 1 & 0 & 0 & 1 & 1 & 0 \\
 0 & 0 & 0 & 0 & 0 & 0 & 1 & 1
 \end{array}
 \right].
\]
Denote the generating function for walks that begin at node $i$ and
end at node $j$ in the automaton by $R_{i,j}$. That is,
\[
 R_{i,j}(t) = \sum_{m\geq 0} r(i,j;m) t^m,
\]
where $r(i,j;m)$ denotes the number of walks of length $m$ that begin at node $i$ and end at node $j$. It is a well-known result that
\begin{equation}\label{eq:Rij}
 R_{i,j}(t) = (-1)^{i+j}\frac{ \det(\mathcal{I}-tA:j,i)
 }{\det(\mathcal{I}-tA) },
\end{equation}
where $\mathcal{I}$ denotes the identity matrix and $\det(B:r,s)$ denotes the determinant of the matrix $B$ after deleting row $r$ and column $s$.

Consider walks of length $m+1$ on the graph given in Figure
\ref{fig:aut}, including loops which are not drawn, starting in states
$1$, $2$, $3$, or $7$, and ending at any tile.  We claim that such walks are in bijective
correspondence with lex-minimal fillings of rafts of length $m$ with fixed boundary apparatus.
The starting state, $1$, $2$, $3$, or $7$, is uniquely determined by the starting tile $s \in \mathcal{A}$.  The final state can be any one of the eight states, but it
must correspond with the raft's boundary apparatus, meaning
$t \in \mathcal{A}$.  The bijection sends such a walk $v_0v_1\ldots
v_{m+1}$ with $v_0=s$ and $v_{m+1}=t$ to $(I,J)$, with $k \in I$ if and
only if the node in the bottom row of $v_k$ is filled, and $k \in J$ if and only if
the node in the top row of $v_k$ is filled, for all $1\leq k \leq m$.  The claim
now follows from the fact that the forbidden configurations will never
occur in such a walk, and conversely a sequence of tiles avoiding the
forbidden configurations can be realized in a unique way by such a walk.

Thus, Table~\ref{table:4x4.table} expresses the generating functions for the
$4 \times 4$ sequences $a_m(s,t)$ for $s,t \in \mathcal{A}$, in terms
of the finite automaton, where we abbreviate
$$R_{i,j_1 \ldots j_k} : = R_{i,j_1}(t) + \cdots + R_{i,j_k}(t)$$
for legibility. Each sequence is shifted by
a factor of $t$ because we are relating walks of length $m+1$ to
fillings of rafts of size $m$.  Along the diagonal, we first subtract
the constant term because we are only considering walks of length at
least 1.  Note that the forbidden configurations in
Figure~\ref{fig:avoid} are symmetric under reversal, so $a_m(s,t)
= a_m(t,s)$.  Therefore, we only have to describe ten sequences.

\begin{table}[htbp]
$$\begin{array}{c|cccc}
  & \oo & \xo & \ox &  \xx  \\ \hline
\oo & \frac{R_{2,2}-1}t & \frac{R_{2,3}}t & \frac{R_{2,57}}t & \frac{R_{2,468}}t   \Tstrut\\[0.3cm]
\xo & & \frac{R_{3,3}-1}t & \frac{R_{3,57}}t & \frac{R_{3,468}}t \\[0.3cm]
\ox & & & \frac{R_{7,57} - 1}{t} & \frac{R_{7,468}}{t} \\[0.3cm]
\xx & & & & \frac{R_{1,1468}-1}{t}
\end{array}$$
\caption{Generating functions for the sequences $a_m(s,t)$.}\label{table:4x4.table}
\end{table}

By applying Equation~\eqref{eq:Rij}, we can compute the explicit form for each
rational generating function $R_{i,j}$ and use that to compute each
$a_m(s,t)$ in Table~\ref{table:4x4.table}.  We leave the computation to the
reader, and instead we give the first terms and recurrences which is
an equivalent formulation. All $a_m(s,t)$ sequences satisfy one of the
following recurrences:
\begin{description}
 \item[R1] $a_m = 5a_{m-1} - 7a_{m-2} + 4a_{m-3}$ for $m \geq 3$
 \item[R2] $a_m = 6a_{m-1} - 13a_{m-2} + 16a_{m-3} - 11a_{m-4} + 4a_{m-5}$ for $m \geq 5$
\end{description}
The recurrences and initial conditions are shown in Table~\ref{table:4x4.table solved}.
\begin{table}[htbp]
$$\begin{array}{c|cccc}
  & \oo & \ox & \xo & \xx  \\ \hline
\oo & \ri 1{1,2,6}  & \ri 1{1,3,9} & \ri 1{1,3,9}  & \ri 1{1,4,12}  \Tstrut\\[0.3cm]
\ox & & \ri 2{1,3,11,37,119} & \ri 2{1,4,12,37,118}  & \ri 1{1,4,14}  \\[0.3cm]
\xo & & & \ri 2{1,3,11,37,119}  & \ri 1{1,4,14}  \\[0.3cm]
\xx & & & & \ri 1{1,4,16}
\end{array}$$
\caption{Recurrence relations and initial conditions for the sequences $a_m(s,t)$.}\label{table:4x4.table solved}
\end{table}

One can observe from this table that
\begin{align*}
a_m\!\raisebox{.5ex}{\Big(}\oo\, ,\ox\raisebox{.5ex}{\Big)} &= a_m\!\raisebox{.5ex}{\Big(}\oo\, ,\xo\raisebox{.5ex}{\Big)},\\
a_m\!\raisebox{.5ex}{\Big(}\ox\, ,\xx\raisebox{.5ex}{\Big)} &= a_m\!\raisebox{.5ex}{\Big(}\xo\, ,\xx\raisebox{.5ex}{\Big)}, \text{ and}\\
a_m\!\raisebox{.5ex}{\Big(}\ox\, ,\ox\raisebox{.5ex}{\Big)} &= a_m\!\raisebox{.5ex}{\Big(}\xo\, ,\xo\raisebox{.5ex}{\Big)},
\end{align*}
so, in fact, the original sixteen sequences $a_m(s,t)$ fall into seven distinct families.

Comparing Table~\ref{table:4x4.table solved} to the statement in Theorem~\ref{thm:amk}
and Remark~\ref{rem:factor.rec} finishes the proof.
\end{proof}

\subsection{Finishing the enumeration}\label{sub:b-sequences}

We now have a complete answer for how many lex-minimal presentations a
raft can have, given a fixed choice of selected nodes on its
boundary. Whether these choices are available depends on the immediate
neighborhood of the raft in the $w$-ocean. We can lump together some of
these boundary cases, which differ only by the selection of ropes.

Suppose that $\p R$ is a raft of size $m$ in the $w$-ocean for some
permutation $w$, and $(i_1,i_2,i_3,i_4)$ are the indicators on the
boundary of the raft.  For a fixed set of selected tethers $T \subseteq
\tethers(w)$ on the boundary of $\p R$, consider
\[
 b_m (\p R, T) = \sum_S a(\p R,S,T) = \sum a_m^{k(i_1,i_2,i_3,i_4)},
\]
where the first sum is over all possible selections from the ropes $S$
adjacent to $\p R$ in the $w$-ocean, and the second sum is over all encodings for the
corresponding triples $(\p R,S,T)$. Note there are at most sixteen terms in the sum. We encode the terms in $b_m (\p R, T)$
by defining indicator sets $K(\p R, T)=(I_1,I_2,I_3,I_4)$ where
$$
I_r = \begin{cases}
  \{0,1\} &  i_r \in \ropes(w),\\
  \{1\}  &   i_r \in T, \text{ and}\\
  \{0\}  &   i_r \not \in T \cup \ropes(w).\\
\end{cases}
$$
Thus,
\[
b_m (\p R, T)  = \sum_{\substack{(i_1,i_2,i_3,i_4) \\ \in I_1\times I_2 \times I_3 \times I_4}} a_m^{k(i_1,i_2,i_3,i_4)}.
\]

Since $b_m (\p R, T)$ only depends on $m$, and $K(\p R, T)=(I_1,I_2,I_3,I_4)$, we
define $81=3^4$ sequences known as the \emph{$b$-sequences}:
\begin{equation}\label{eq:b.seq}
b_m^{(I_1,I_2,I_3,I_4)}  = \sum_{\substack{(i_1,i_2,i_3,i_4) \\ \in I_1\times I_2 \times I_3 \times I_4}} a_m^{k(i_1,i_2,i_3,i_4)},
\end{equation}
where each $I_r \in \{ \{0\},\{1\}, \{0,1\}\}$.  Similar to the
$a$-sequences, we can use open or filled dots to denote the possible
apparatus on either end of a raft.  Let the symbols $\nodot$\, , $\yesdot$\, , and
$\bothdot$ represent the three options $\{0\}$, $\{1\}$, and $\{0,1\}$, respectively.
Then the apparatus at each end of a raft with a selection of specified
tethers can be represented by one of the nine tiles in the alphabet

$$
\mathcal{B} =\raisebox{.5ex}{\Big\{}\oo\, ,\ \ox\, ,\ \xo\, ,\ \xx\, ,\ \bo\, ,\ \ob\, ,\ \bx\, ,\ \xb\, ,\ \bb\raisebox{.5ex}{\Big\}}.
$$

\noindent
Thus, each $b_m^{(I_1,I_2,I_3,I_4)}$ can be denoted by $b_m(s,t)$ for
$s,t \in \mathcal{B}$.  Up to symmetry among the $a$-sequences, there
are only 27 different $b$-sequences.  Initial terms in each case appear in the Appendix to this article.

\begin{example}
Suppose $\p R$ is the following raft of length $m=6$.
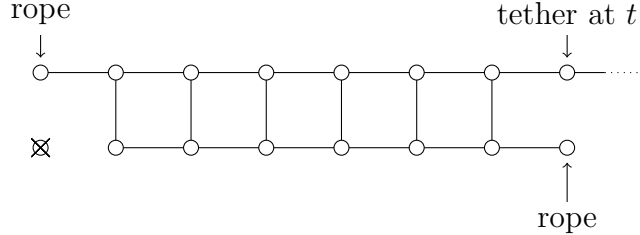
\[
 \begin{tikzpicture}
  \draw[dotted] (7.5,1)--(8,1);
  \draw[->] (0,1.5) node[above] {rope} -- (0,1.2);
  \draw[->] (7,1.5) node[above] {tether at $t$} -- (7,1.2);
  \draw[->] (7,-1) node[fill=white] {rope} -- (7,-.2);
  \draw (1,0)-- (7,0) node[draw=black,fill=white,circle,inner sep=2] {};
  \draw (0,1) node[draw=black,fill=white,circle,inner sep=2] {} -- (7,1) node[draw=black,fill=white,circle,inner sep=2] {} --(7.5,1);
  \draw (0,0) node[draw=black,fill=white,circle,inner sep=2] {};
  \foreach \x in {1,...,6}{
    \draw (\x,0) node [circle,draw=black,fill=white,inner sep=2] {} -- (\x,1) node [circle,draw=black,fill=white,inner sep=2] {};     \draw (0,0) node[cross=.75ex]{};
  }
 \end{tikzpicture}
\]
If the tether at $i_4$ in this raft is not selected, then, by considering whether or not the ropes are selected, we conclude that
\[
b_6({\p R}, \emptyset) = b_6^{(0,01,01,0)} = b_6\raisebox{.25ex}{\Big(}\bo\, , \ob\raisebox{.25ex}{\Big)} = a_6\raisebox{.25ex}{\Big(}\oo\, , \oo\raisebox{.25ex}{\Big)} + a_6\raisebox{.25ex}{\Big(}\oo\, , \ox\raisebox{.25ex}{\Big)} + a_6\raisebox{.25ex}{\Big(}\xo\, , \oo\raisebox{.25ex}{\Big)}  + a_6\raisebox{.25ex}{\Big(}\xo\, , \ox\raisebox{.25ex}{\Big)} =3732
\]
lex-minimal presentations for this raft, where we have abbreviated $\{0\}$ by ``$0$'' and $\{0,1\}$ by ``$01$'' for the sake of legibility. On the other hand, selecting the tether at $i_4$ yields
\[
 b_6 (\p R, \{t\}) = b_6^{(0,01,01,1)} = b_6\raisebox{.25ex}{\Big(}\bo\, , \xb\raisebox{.25ex}{\Big)} =a_6\raisebox{.25ex}{\Big(}\oo\, , \xo\raisebox{.25ex}{\Big)} + a_6\raisebox{.25ex}{\Big(}\oo\, , \xx\raisebox{.25ex}{\Big)} + a_6\raisebox{.25ex}{\Big(}\xo\, , \xo\raisebox{.25ex}{\Big)}  + a_6\raisebox{.25ex}{\Big(}\xo\, , \xx\raisebox{.25ex}{\Big)} = 4788
\]
lex-minimal presentations for $\p R$, with analogous abbreviations.
\end{example}

\begin{cor}\label{cor:b-rec}
  The sequences $b_m^{(I_1,I_2,I_3,I_4)}$  satisfy the
  linear recurrence
  \begin{equation*}
    b_m  = 6 b_{m-1} - 13 b_{m-2} + 16 b_{m-3} - 11 b_{m-4} + 4 b_{m-5} \quad \mbox{for} \quad m \geq 5,
  \end{equation*}
with initial conditions depending on $(I_1,I_2,I_3,I_4)$, which can be
deduced from the initial conditions for the $a$-sequences given in
Table~\ref{table:a.table}.
\end{cor}

\begin{proof}
The recurrence relation follows from the fact that the $b$-sequences
are each defined as a fixed finite sum of the $a$-sequences, and
the $a$-sequences all satisfy the same recurrence, as shown in
Theorem~\ref{thm:amk}.
\end{proof}

Now that we have developed all of the notation, we can prove our
main enumeration theorem for $S_n$, originally given in
Theorem~\ref{main}.  We restate the theorem here for the reader's
convenience before giving the proof.

\begin{reppthm}{main}
The number of parabolic double cosets with minimal element $w \in S_n$ is
\[
 c_w = 2^{|\floats(w)|} \sum_{T \subseteq \tethers(w) } \prod_{\p R \in \rafts(w)} b_{|\p R|}^{K(\p R,T)}.
\]
\end{reppthm}

\begin{proof}
By Proposition~\ref{prop:raft.rope.tether.formula}, we reduce the
computation of $c_w$ to finding $a(\p R,S,T)$, the number of lex-minimal
fillings of the raft $\p R$ with boundary apparatus
apparatus determined by $S$ and $T$.  By
definition of the $b$-sequences, we collect the terms in each sum corresponding to $\ropes(w)$ to get the stated formula.
\end{proof}

We finish our discussion of the symmetric group case by calculating $c_w$ for a large, somewhat generic example.

\begin{example}
Continuing Example~\ref{ex:biggun}, let
$$w = 1\ 3\ 4\ 5\ 7\ 8\ 2\ 6\ 14\ 15\ 16\ 9\ 10\ 11\ 12\ 13 \in S_{16},$$
whose balls-in-boxes picture is shown in Figure~\ref{fig:marine model overlaid on big example}, and whose $w$-ocean is shown in Figure~\ref{fig:marine-model picture example}. Recall that $\floats(w) = \{ 1', 7\}$, $\tethers = \{8',4\}$, and there are four rafts:
\[
 A = [2,3], \ B=[5,5], \ C = [9,10], \ \text{and} \ D=[12,15].
\]
Raft $A$ has ropes in positions $i_2$ and $i_3$, raft $B$ has no ropes, raft $C$ has a rope in position $i_2$, and raft $D$ has no ropes. Thus
\begin{align*}
c_w &= 2^{|\floats(w)|} \sum_{T \subseteq \tethers(w) } \prod_{\p R \in \rafts(w)} b_{\p R}( T),\\
 &= 4\cdot \sum_{T \subseteq \{8',4\} } b_A(T)b_B(T)b_C(T)b_D(T).
\end{align*}

We now consider each subset of $T$. If $T=\emptyset$, then
\begin{align*}
 b_A(\emptyset) &= b_2^{(0,01,01,0)} = a_2^0 + 2a_2^1 + a_2^{2''} = 6+2(9)+12=36,\\
 b_B(\emptyset) &= b_1^{(0,0,0,0)} = a_1^0 = 2,\\
 b_C(\emptyset) &= b_2^{(0,01,0,0)} = a_2^0 + a_2^1 = 6+9=15,\text{ and}\\
 b_D(\emptyset) &= b_4^{(0,0,0,0)} = a_4^0 = 66,
\end{align*}
and so
\[
\prod_{\p R \in \rafts(w)} b_{\p R}(\emptyset) =71280.
\]

\noindent
For $T=\{8'\}$, we find
\begin{align*}
 b_A(\{8'\}) &= b_2^{(0,01,01,0)} = a_2^0 + 2a_2^1 + a_2^{2''} = 6+2(9)+12=36,\\
 b_B(\{8'\}) &= b_1^{(0,0,1,0)} = a_1^1 = 3,\\
 b_C(\{8'\}) &= b_2^{(0,01,0,0)} = a_2^0 + a_2^1 = 6+9=15,\text{ and}\\
 b_D(\{8'\}) &= b_4^{(1,0,0,0)} = a_4^1 = 89,
\end{align*}
and so
\[
\prod_{\p R \in \rafts(w)} b_{\p R}(\{8'\}) =144180.
\]

\noindent
When $T=\{4\}$, we have
\begin{align*}
 b_A(\{4\}) &= b_2^{(0,01,01,1)} = a_2^1 + a_2^2 + a_2^{2'} +a_2^3= 9 + 12 + 11 + 14 =46,\\
 b_B(\{4\}) &= b_1^{(0,1,0,0)} = a_1^1 = 3,\\
 b_C(\{4\}) &= b_2^{(0,01,0,0)} = a_2^0 + a_2^1 = 6+9=15,\text{\ and}\\
 b_D(\{4\}) &= b_4^{(0,0,0,0)} = a_4^0 = 66,
\end{align*}
and so
\[
\prod_{\p R \in \rafts(w)} b_{\p R}(\{4\}) =136620.
\]

\noindent
Finally, $T = \{8',4\}$ yields
\begin{align*}
 b_A(\{8',4\}) &= b_2^{(0,01,01,1)} = a_2^1 + a_2^2 + a_2^{2'} +a_2^3= 9 + 12 + 11 + 14 =46,\\
  b_B(\{8',4\}) &= b_1^{(0,1,1,0)} = a_1^{2''} = 4,\\
 b_C(\{8',4\}) &= b_2^{(0,01,0,0)} = a_2^0 + a_2^1 = 6+9=15,\text{ and}\\
 b_D(\{8',4\}) &= b_4^{(1,0,0,0)} = a_4^1 = 89,
\end{align*}
and so
\[
\prod_{\p R \in \rafts(w)} b_{\p R}(\{8',4\}) =245640.
\]

\noindent
Therefore, the number of parabolic double cosets for which this $w$ is the minimal element is
\begin{align*}
c_w &= 4(71280+144180+136620+245640)\\
 &= 2,\!390,\!880.
\end{align*}
\end{example}

\subsection{Expected number of tethers}

Given a particular $w$, Theorem~\ref{main} can, in general, be computed
quickly using the linear recurrence relation for the $b$-sequences
given in Corollary~\ref{cor:b-rec}. However, one might be concerned
about computing the sum over all subsets of $\tethers(w)$.  Recall the claim, made in the
introduction, that tethers are rare (and hence the sum
has few terms). In fact, as stated in Proposition~\ref{prop:expected},
an exact formula for the expected number of tethers for $w\in S_n$
for all $n>2$ is
\[
 \frac{1}{n!}\sum_{w \in S_n} |\tethers(w)| = \frac{(n-3)(n-4)}{n(n-1)(n-2)}.
 \]

\begin{proof}[Proof of Proposition~\ref{prop:expected}]
For $w$ to have a right tether at position $k$ means that
$$(w(k-1),w(k),w(k+1),w(k+2)) = (i,i+1,j,j+1),$$
where $i+1 < j - 1$.
Let $RT_{k}(w) \in \{0,1\}$ be the number of right tethers of $w$ in position $k$.  The expected value of $RT_{k}$ as a
random variable can be computed by counting the number of pairs $i,j$
for which $1 \leq i < j-2 \leq n-3$, multiplied by the number of permutations
containing $(i,i+1,j,j+1)$ as consecutive values with $i+1$ in position
$k$.  Namely,
\[
\EE(RT_{k}) = \frac{\binom{n-3}2 (n-4)!}{n!} = \frac{(n-4)}{2n(n-1)(n-2)}.
\]
for all $2 \leq k \leq n-2$.  Expectation of random variables is
linear, so the expected number of right tethers is
\[
\EE(RT_{2}) + \cdots + \EE(RT_{n-2}) = \frac{(n-3)(n-4)}{2n(n-1)(n-2)}.
\]
An analogous argument proves that the expected number of left tethers
has the same value, so summing the two contributions proves the formula.
\end{proof}

While the set $\tethers(w)$ is typically small (on the order of $1/n$), there are some permutations for which $|\tethers(w)|$ can be quite large, as seen in the following example.

\begin{example}\label{ex:manytethers}
The permutation
$$w = 1 \ 2 \ 17 \ 18 \ 3 \ 4 \  19 \ 20 \ \ldots \ 15 \ 16 \ 31 \ 32$$ has fourteen left tethers and eight right tethers (see Figure~\ref{fig2}), leading to a sum with $2^{22}$ terms. In this case, the value $c_w = 632,\!371,\!867,\!544,\!102$ can be determined on a computer within a few minutes.
\end{example}

\begin{figure}[htbp]
\begin{center}
\begin{tikzpicture}[scale=.5]
 \draw (1,0) -- (15,0); \draw (17,0) -- (31,0);
 \foreach \x in {1,...,8}
     {\draw  (-1+4*\x,3) -- (-3+4*\x,3) -- (-1+2*\x,0); \draw (-1+4*\x,3) -- (15+2*\x,0);}
 \foreach \x in {1,...,8}
     {
     \draw (-2 + 4*\x,3) circle (1.5ex);
     }
 \foreach \x in {1,...,7}
     { {\draw (2*\x,0) circle (1.5ex); \draw (16 + 2*\x,0) circle (1.5ex);
}
 \foreach \x in {1,...,31}
     {\draw[draw=black,fill=white] (\x,3) circle (1ex);
     \draw[draw=black,fill=white] (\x,0) circle (1ex);}
 \foreach \x in {1,...,7}
     \draw (4*\x,3) node[cross=.75ex]{};
          }
     \draw (16,0) node[cross=.75ex]{};
\end{tikzpicture}
\end{center}
\caption{The $w$-ocean for the permutation of Example~\ref{ex:manytethers}.
 We see that $w$ has $22$ large ascents denoted by two concentric
 circles, each of which is a tether} \label{fig2}
\end{figure}
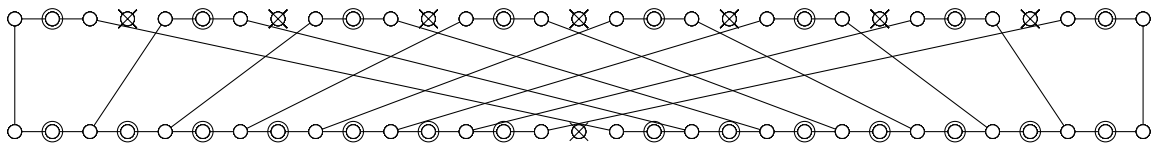

\section{Parabolic double cosets for Coxeter groups}\label{sec:general groups}

We now turn to the general setting of Coxeter groups, where our first task is to extend the
characterization of lex-minimal presentations in Theorem
\ref{T:minimalsym} to this context.

Fix a Coxeter system $(W,S)$ and recall the notation and terminology of Section~\ref{sec:background}. We first observe that if a coset has two different presentations, we get a third presentation by taking the union of the generators acting on the left and the right.

\begin{lem}\label{lem:newpres}
    Suppose that a parabolic double coset $C$ has two different presentations $W_IwW_J =
    C = W_{I'}wW_{J'}$. Then $W_{I\cup I'}wW_{J\cup J'}$ is also a presentation for
    $C$.
\end{lem}

\begin{proof}
    The lemma is an instance of a fact about general groups: if $H_i,K_i$, $i=1,2$ are subgroups of a group $G$ such
    that $C := H_1 g K_1 = H_2 g K_2$ for some $g \in G$, then $C = H g K$, where $H = \langle H_1,H_2 \rangle$
    is the subgroup generated by $H_1$ and $H_2$, and $K = \langle K_1, K_2 \rangle$ is the subgroup generated by
    $K_1$ and $K_2$.

    To prove this fact, observe that $H_i C K_i = H_i g K_i = C$. But this means that
    $C = H_1 H_2 g K_2 K_1$. Repeating this idea again, we get that $C = H_2 C K_2 = H_2 H_1 H_2 g K_2 K_1 K_2$.
    By induction, we conclude that $C = (H_1 H_2)^k g (K_2 K_1)^k$ for all $k \geq 1$. Taking the union across
    $k$, we get that $C = H g K$.

    To finish the lemma, note that if $H_1 = W_I$ and $H_2 = W_{I'}$, then
    $\langle H_1, H_2 \rangle = W_{I \cup I'}$.

\end{proof}

\begin{defn}
    Given a parabolic double coset $C$, set
    \begin{equation*}
        M_L(C) := \bigcup_{\substack{(I,w,J) \ : \\ C = W_I w W_J}} I \quad\text{ and }\quad M_R(C) := \bigcup_{\substack{(I,w,J) \ : \\ C = W_I w W_J}} J.
    \end{equation*}
\end{defn}

\begin{prop}\label{P:maxpres}
    Let $C$ be a parabolic double coset.
    \begin{enumerate}[(a)]
        \item $C$ has a presentation $C = W_{M_L(C)} w W_{M_R(C)}$, and this is
            the largest possible presentation for $C$, in the sense that if $C = W_I w'
            W_J$ then $I \subseteq M_L(C)$ and $J \subseteq M_R(C)$.
        \item The sets $M_L(C)$ and $M_R(C)$ can be determined by
            \begin{align*}
                M_L(C) & = \{s \in S : s x \in C \text{ for all } x \in C\} \text{ and } \\
                M_R(C) & = \{s \in S : x s \in C \text{ for all } x \in C\}.
            \end{align*}
    \end{enumerate}
\end{prop}

\begin{proof}
Part (a) follows immediately from the definition and Lemma \ref{lem:newpres}.

For part (b), let $I' = \{s \in S : s x \in C \text{ for all } x \in C\}$ and $J' = \{s \in S : x s \in C \text{ for all } x \in C\}.$ Then $M_L(C) \subseteq I'$ and $M_R(C) \subseteq J'$. Hence $C \subseteq W_{I'}wW_{J'}$. At the same time, $C$ is closed under left multiplication by members of $I'$ and right multiplication by members of $J'$. Hence $W_{I'}wW_{J'} \subseteq C$. Thus $W_{I'} w W_{J'}$ is a presentation for $C$, which means $I' \subseteq M_L(C)$ and $J'\subseteq M_R(C)$. Putting this all together means $I' = M_L(C)$ and $J' = M_R(C)$, as desired.
\end{proof}

In light of Proposition~\ref{P:maxpres}, we introduce the following terminology.

\begin{defn}
The presentation $W_{M_L(C)} w W_{M_R(C)}$ appearing in Proposition~\ref{P:maxpres} is the \emph{maximal presentation} for $C$.
\end{defn}

We will also want to identify presentations that are as small as possible.

\begin{defn}\label{D:minimal}
    A presentation $C = W_I w W_J$ is \emph{minimal} if
    \begin{enumerate}[(a)]
        \item $w \in {}^I W^J$,
        \item no connected component of $I$ is contained in $(w J w^{-1}) \cap S$, and
        \item no connected component of $J$ is contained in $(w^{-1} I w) \cap S$.
    \end{enumerate}
\end{defn}

Note that this is not the same as lex-minimality, which was introduced in Definition~\ref{defn:lex-minimal}.

If a connected component $I_0$ of $I$ is contained in $(w J w^{-1}) \cap S$ for
$w \in {}^I W^J$, then $W_{I} w W_J = W_{I \setminus I_0} w W_J$.  A similar argument applies to subsets of $J$, so every
presentation can be reduced to a minimal presentation. In Proposition
\ref{P:minimal}, we will show that our nomenclature is appropriate; that is,
minimal presentations have minimum size.

\begin{lem}\label{L:descentsets}
    Let $C = W_I w W_J$ be a minimal presentation of $C$. Then
    \begin{align*}
        M_L(C) & = I \cup \{s \in (w J w^{-1}) \cap S : s \text{ is not adjacent to } I\} \text{ and }\\
        M_R(C) & = J \cup \{s \in (w^{-1} I w) \cap S : s \text{ is not adjacent to } J\}.
    \end{align*}
\end{lem}

\begin{proof}
    That $M_L(C)$ includes its proposed reformulation is clear. For the
    other direction, suppose that $s \in M_L(C)$ and $s \not\in I$. By
    Proposition~\ref{P:maxpres}(b), $w < s w \in C$, and by Corollary
   ~\ref{C:parabolic} it follows that $s w = w t$ for some $t \in J$.

    We can extend this argument to show that $s$ is not adjacent to any element
    of $I$.  Indeed, suppose that $s$ is adjacent to some element of $r$ of
    $I$, and let $I_0$ be the connected component of $r$ in $I$. We
    argue that $I_0 \subseteq w J w^{-1}$, in contradiction of our hypothesis.
    Any element $r'$ of $I_0$ is connected to $s$ by a simple path $s = s_0,
    s_1,\ldots,s_k=r'$, $k\geq 1$, in the Coxeter graph of $W$, where
    $s_1,\ldots,s_k$ is entirely contained in $I_0$. Using
    Proposition~\ref{P:maxpres}(b) and Corollary~\ref{C:parabolic}, we
    can write $s_0 \cdots s_k w = u w v$, where $u \in W_I$ and $v \in
    {}^{(w^{-1} I w) \cap S} W_J$. Now, $s_0$ is the only left descent of $s_0
    \cdots s_k$, and because $I \cup \{s_0\} \subset \asc_L(w)$, we conclude
    that $s_0$ is the only left descent of $s_0 \cdots s_k w$ in $I \cup \{s_0\}$.
    But any left descent of $u$ will be a left descent of $s_0 \cdots s_k w$ in
    $I$, so we conclude that $u$ has no left descents, or in other words, $u = e$. As a result, $w^{-1} s_0 \cdots s_k w = v \in
    W_J$. The same argument shows that $w^{-1} s_0 \cdots s_{k-1} w \in
    W_J$. Hence we have $w^{-1} s_k w \in W_J$. Because $s_k w = w (w^{-1}
    s_k w)$, and $w \in {}^I W^J$, we conclude that $w^{-1} s_k w = w^{-1} r' w$ belongs to
    $J$. Thus $I_0 \subseteq w J w^{-1}$, yielding the desired contradiction.

    The equality for $M_R(C)$ is analogous.
\end{proof}

\begin{cor}\label{C:descentsets}
    Suppose that $C = W_I w W_J$ is a minimal presentation of $C$.  If $T$ is
    any connected subset of $M_L(C)$, then either $T \subseteq I$, or $T$ is
    disjoint and non-adjacent to $I$ and $T \subseteq (w J w^{-1}) \cap S$.
\end{cor}

\begin{proof}
    Suppose $s \in T$ is not contained in $I$. By Lemma~\ref{L:descentsets},
    this $s$ must be non-adjacent to $I$ and contained in $w J w^{-1}$. Because
    $T$ is connected, iterating this argument for the neighbors
    of $s$ yields the desired conclusion.
\end{proof}

Given subsets $X,Y,Z \subseteq S$, write
$$X = Y \naunion Z$$
to mean that $X$ is
the disjoint union of $Y$ and $Z$, and $Y$ and $Z$ are non-adjacent. (In other
words, the subgraph of the Coxeter graph induced by the vertex set $X$ is
isomorphic to the disjoint union of the vertex-induced subgraphs of $Y$ and
$Z$.) The following proposition is, roughly speaking, obtained by repeated
application of Corollary~\ref{C:descentsets}.

\begin{prop}\label{P:minimal}
    Fix $w \in {}^I W^J$. A presentation $C = W_I w W_J$ is minimal if and
    only if $|I| + |J| \leq |I'| + |J'|$ for all other presentations $C = W_{I'}
    w W_{J'}$ of $C$. Furthermore, if $C = W_I w W_J$ and $C = W_{I'} w W_{J'}$ are both minimal
    presentations, then there are sequences of connected components
    $I_1,\ldots,I_m$ and $J_1,\ldots,J_n$ of the subgraphs induced by $I$ and
    $J$, respectively, such that for each $i$ and $j$, $wJ_jw^{-1}$ and $w^{-1}I_iw$ are subsets of $S$, and
    \begin{align*}
        I' & = \left(I \naunion w J_1 w^{-1} \naunion \cdots \naunion w J_n w^{-1}\right) \setminus \bigcup_{i=1}^m I_i \text{ and } \\
        J' & = \left(J \naunion w^{-1} I_1 w \naunion \cdots \naunion w^{-1} I_m w\right) \setminus \bigcup_{j=1}^n J_j .
    \end{align*}
\end{prop}

Note that the order of operations in the identities of
Proposition~\ref{P:minimal} is significant, in that $w J_j w^{-1}$ is required
to be disjoint and non-adjacent to $I_i$ for all $i,j$.

\begin{proof}[Proof of Proposition~\ref{P:minimal}]
    We start by proving the second part of the proposition. Suppose that $C =
    W_I w W_J$ and $C = W_{I'} w W_{J'}$ are both minimal presentations, and
    let $I_0$ be a connected component of $I$. Then $I_0 \subseteq M_L(C)$, so
    by Corollary~\ref{C:descentsets} either $I_0 \subseteq I'$, or $I_0$ is
    disjoint and non-adjacent to $I'$. In the former case, $I_0$ must be
    contained in a connected component $I'_0$ of $I'$. Applying Corollary
   ~\ref{C:descentsets} again, we must have $I'_0 \subseteq I$, and hence $I_0 =
    I'_0$. If $I_0$ is disjoint and non-adjacent to $I'$, then Corollary~\ref{C:descentsets}
    also tells us that $I_0 \subseteq (w J'_0 w^{-1}) \cap S$,
    where $J'_0$ is a connected component of $J'$. Since $C = W_I w W_J$
    is minimal, it is impossible for $J'_0$ to be contained in $J$, because then we would
    have $I_0 \subseteq w J w^{-1}$. Applying Corollary~\ref{C:descentsets}
    one last time, we see that $J'_0$ is disjoint and non-adjacent to $J$,
    and $J'_0 \subseteq w^{-1} I_1 w$ for some connected component $I_1$ of
    $I$. But then $I_0 \subseteq I_1$, which means that $I_0 = I_1$ and
    $J'_0 = w^{-1} I_0 w$. Combining the two cases shows that
    either $I_0$ is a connected component of $I'$, or $w^{-1} I_0 w$
    is a connected component of $J'$, which is disjoint and non-adjacent
    to $J$. This proves the second part of the proposition.

    For the first part of the proposition, let $N$ be the minimum of
    $|\tilde{I}| + |\tilde{J}|$ across all presentations $C = W_{\tilde{I}} w
    W_{\tilde{J}}$ of the parabolic double coset $C$. Any presentation $C = W_{I'} w W_{J'}$
    with $|I'| + |J'| = N$ is clearly minimal. If $C = W_I w W_J$ is another
    minimal presentation, then $I,J$ and $I',J'$ are related as in the second
    part of the proposition, and consequently $|I| + |J| = |I'| + |J'| = N$.
\end{proof}

We now get the desired characterization of lex-minimal presentations as an
immediate corollary of Proposition~\ref{P:minimal}. This will also complete the proof of Theorem~\ref{T:minimalsym}.

\begin{thm}\label{T:minimal}
    Fix $w \in {}^I W^J$. Then $C = W_I w W_J$ is lex-minimal if and only if
    \begin{enumerate}[(a)]
        \item no connected component of $J$ is contained in $(w^{-1} I w) \cap S$, and
        \item if a connected component $I_0$ of $I$ is contained in $w S w^{-1}$,
            then $w^{-1} I_0 w$ is not contained in $J$, and there is an element of
            $J$ adjacent to or contained in $w^{-1} I_0 w$.
    \end{enumerate}
    Furthermore, every parabolic double coset has a unique lex-minimal presentation.
\end{thm}

When $w$ is the identity, Theorem~\ref{T:minimal} implies that
lex-minimal presentations are two-level staircase diagrams in the
sense of recent work by Richmond and Slofstra \cite{RS15}. Note that
while \cite{RS15} addresses the enumeration of staircase diagrams,
two-level staircase diagrams were not considered.

\subsection{The marine model for Coxeter groups}

The enumeration formula in Theorem~\ref{main} extends to general
Coxeter groups. To describe the formula in the general setting, we
need to extend the marine model.  In
Theorem~\ref{thm:equivalence.moves}, we present a visual test for
detecting when two parabolic double cosets with the same minimal
element are equal.

\begin{defn}
For $w \in W$, an ascent $s \in \asc_{R}(w)$ is a \emph{small ascent}
if $w s w^{-1} \in S$.  Otherwise, this $s$ is a \emph{large ascent}.
\end{defn}

Note that $s$ is a small ascent of $w$ if and only if $w s w^{-1}$ is
a small ascent of $w^{-1}$.  Also, if $s$ and $t$ are both small
ascents for $w$, then $m_{st} = m_{s't'}$, where $s' = w s w^{-1}$ and
$t' = w t w^{-1}$.

Definition~\ref{defn:marine model pictures} introduced the $w$-ocean of a permutation, built out of two copies of the Coxeter graph of $S_n$. For general Coxeter groups, we build an analogous graph.

\begin{defn}\label{defn:general marine model pictures}
    For $w \in W$, the \emph{$w$-ocean}
    is the graph $G(W,S,w)$ whose vertices are
    \begin{equation*}
        \big\{ (s,1) : s \text{ a right ascent of } w \big\} \cup \big\{ (s,0) : s \text{ a left ascent of } w \big\}.
    \end{equation*}
    There is an edge between vertices $(s,1)$ and $(t,1)$ (respectively, $(s,0)$ and $(t,0)$), if $s$ and $t$ are adjacent in the Coxeter graph of $W$, and
    at least one of them is a small ascent of $w$ (respectively, $w^{-1}$).  There is
    an edge between $(s,1)$ and $(t,0)$ if $s$ is a small ascent of $w$ and $t
    = w s w^{-1}$.
\end{defn}

In terms of Definition~\ref{defn:marine model pictures}, the vertices $(s,1)$ are
the top row of the $w$-ocean, and the vertices $(s,0)$ are the bottom
row.

\begin{defn}
Given a $w$-ocean $G = G(W,S,w)$, we identify certain vertices
and induced subgraphs. Note how these generalize the classifications
in Definition~\ref{defn:marine model for S_n}, and note also the new
type ``wharf.''
\begin{itemize}
\item A \emph{small ascent} of $G$ is a vertex of the form $(s,1)$ or
$(s,0)$ such that $s$ is a small ascent of $w$ or $w^{-1}$
respectively.

\item A \emph{large ascent} of $G$ is a vertex of the form $(s,1)$ or
$(s,0)$ such that $s$ is a large ascent of $w$ or $w^{-1}$
respectively.

\item A \emph{float} is an isolated vertex in $G$.  They are all large
ascents of $G$ not adjacent to any small ascents.

\item A \emph{rope} is a large ascent of $G$ which is adjacent to
exactly one small ascent of $G$.

 \item A \emph{tether} is a large ascent of $G$ which is adjacent to
at least two small ascents of $G$.

\item A \emph{plank} in $G$ is an induced subgraph consisting of
exactly two small ascents of the form $\{(s,1),(t,0) \}$ such that $t
= w s w^{-1}$.

 \item A \emph{wharf} is a plank in $G$ such that at least one of its
two vertices is adjacent to at least three other vertices of $G$, all
on the same row, and at least two of the adjacent vertices are small
ascents of $G$. In addition, any plank can be designated as a wharf if at least
one of its vertices is adjacent to two other small ascents of $G$ (on the same
row). If $G$ contains cycles, then we will always choose a number of additional
wharfs so that the graph with tethers and wharfs deleted is acyclic.

 \item A \emph{raft} is a connected component of the subgraph of $G$ induced by
the planks which are not wharfs.  The \emph{size} of a raft is the number of
planks it contains, or equivalently half the number of vertices.  If a raft $R$
has size $m$, then it is isomorphic to the $w$-ocean for $w=e \in S_{m+1}$.
\end{itemize}
\end{defn}

Compared to the definition of the marine model in Section
\ref{sec:marinemodelforSn}, floats, ropes and rafts are essentially
the same as before. Tethers are also essentially the same, although
they can be adjacent to more than two small ascents when the degree of
the corresponding vertex of the Coxeter graph is at least three. It is
also now possible to have small ascents that are connected to more
than two small ascents, or to two small ascents and a large ascent,
which is where wharfs come in.  Examples of all of these objects are
given in the next sections.

\begin{remark}\label{rem:oceans and edge labels}
  Observe that the construction of the $w$-ocean does not depend on
  edge labels $m(s,t) \ge 3$ in the Coxeter graph of $(W,S)$.
\end{remark}

The $w$-ocean is a useful tool for studying the presentations
$(I,w,J)$ for all $(I,J) \in \asc(w^{-1}) \times \asc(w)$.  Each such
pair $(I,J)$ can be represented by the subgraph of the $w$-ocean
containing the vertices $\{(s_i,0): s_i \in I\} \cup \{(s_j,1): s_j \in J\}$.
We denote the selected vertices by filled dots and the unselected
vertices by open dots in the $w$-ocean as in the type $A$ case.  By a
slight abuse of notation, we will equate the subsets of
$\{(s_i,0): s_i \in I\}$ in the $w$-ocean with subsets of $I$ by the
natural bijection.  In particular, the connected components of $I$ as
an induced subgraph of the Coxeter graph are in natural bijection with
the components of $\{(s_i,0): s_i \in I\}$ as an induced subgraph of the
$w$-ocean.  Similarly, there is a natural bijection between subsets of
$J$ and $\{(s_j,1): s_j \in J\}$.

\begin{defn}\label{defn:plank.moves}
  Assume $w\in W$ and $(I,J) \in \asc(w^{-1}) \times \asc(w)$.  Let
  $I_0$ be a connected component of $I$ consisting entirely of small
  left ascents.  Then, each vertex in $I_0$ is the endpoint of a plank
  in the $w$-ocean.  Let $J_0 =w^{-1}I_0 w$ be the corresponding
  connected set of endpoints on the top row of the $w$-ocean so $J_0$
  consists entirely of small right ascents.  We define the following
  three types of \emph{plank moves} when applicable along with their
  analogs obtained from switching the roles of $I$ and $J$.

\begin{itemize}
\item Contraction move: $(I,J) \rightarrow (I\setminus I_0,J)$ provided
  $J_0 \subset J$.

\item Expansion move: $(I\setminus I_0,J) \rightarrow (I,J)$ provided
  $J_0 \subset J$.

\item Slide move: $(I,J) \rightarrow (I\setminus I_0,J \cup J_0)$
  provided $J \cap J_0 = \emptyset$ and $J_0$ is a connected component
 of $J\cup J_0$ on the top row of the $w$-ocean.
\end{itemize}
\end{defn}

\begin{thm}\label{thm:equivalence.moves}
  Let $(I,J),(I',J') \in \asc(w^{-1}) \times \asc(w)$.  Then,
  $W_I w W_J = W_{I'} w W_{J'}$ if and only if $(I,J)$ can be obtained
  from $(I',J')$ by plank moves.
\end{thm}

\begin{proof}
  Let $(I,J) \in \asc(w^{-1}) \times \asc(w)$, and let $I_0 \subset I$
  be a connected component of $I$ as an induced subgraph of the
  Coxeter graph.  Then, the elements of $W_{I\setminus I_0}$ commute
  with $W_{I_0}$ so $W_Iw = W_{I\setminus I_0} W_{I_0} w$.  If in
  addition each vertex in $I_0$ is a small ascent or equivalently
  adjacent to a plank, then $W_Iw=W_{I\setminus I_0} w W_{J_0}$ where
  $J_0=w^{-1}I_0 w$ is the set of vertices on the other side of the
  planks attached to $I_0$.  Thus,
  $W_I w W_J = W_{I\setminus I_0} w W_{J_0} W_J$.  The product
  $W_{J_0} W_J$ equals the parabolic subgroup $W_{J}$ if
  $J_0 \subset J$ (expansion/contraction move) or $W_{J_{0} \cup J}$
  provided $W_{J_0}$ and $W_J$ are commuting subgroups (slide move).
  Recall, the ladder condition is equivalent to saying $J_0$ is a
  connected component of the induced subgraph of the $w$-ocean on
  vertices $J\cup J_0$ and $J \cap J_0 = \emptyset$.  Thus, both
  slides and contractions preserve the corresponding double cosets so
  if $(I',J')$ is connected to $(I,J)$ by plank moves, then
  $W_I w W_J = W_{I'} w W_{J'}$.

  Conversely, assume $W_I w W_J = W_{I'} w W_{J'}$.  By applying all
  possible contraction moves on $(I,J)$ and $(I',J')$, we arrive at two
  minimal presentations for the same double coset.  Then by
  Proposition~\ref{P:minimal}, these two minimal presentations are
  connected by slide moves.
\end{proof}

\begin{cor}\label{cor:lex-min-fillings}
  For $w \in W$ and $(I,J) \in \asc(w^{-1}) \times \asc(w)$, the
  lex-minimal presentation of the parabolic double coset $W_I w W_J$
  is obtained from $(I,J)$ by applying all possible contraction moves
  and then applying all possible slide moves on the remaining
  components in the bottom row.
\end{cor}

\begin{proof}
  The statement follows directly from Theorem~\ref{T:minimal} and
  Theorem~\ref{thm:equivalence.moves}.
\end{proof}

\begin{example}\label{ex:plank.moves}
  Consider the permutation
\[w = 2 \ 3 \ 4 \ \cdots \ 15 \ 16 \ 1 \ \in S_{16}\]
and the
  parabolic double coset $W_I w W_J$ with
\[I=\{s_i : i \in \{2,3,5,6,7,8,9,11,12,14,15\}\} \text{\ and \ }
  J=\{s_j : j \in \{1,2,3,4,5,7,9,10\}\}.\]
 Figure~\ref{fig:slide.1} shows the
  connected components of $I$ and $J$ as induced subgraphs of the
  $w$-ocean on the bottom and top rows respectively.  The arrows
  indicate possible plank moves; from left to right we see two
  contractions and a slide.  Applying these plank moves to $(I,J)$
  results in the pair $(I',J')$ for
\[I'=\{s_i : i \in \{5,6,7,8,9,11,12\}\} \text{\ and \ }
  J'=\{s_j : j \in \{1,2,3,4,5,9,10,13,14\}\}\]
whose connected components are shown
  in Figure~\ref{fig:slide.2}.  Thus, by
  Theorem~\ref{thm:equivalence.moves}, $W_I w W_J = W_{I'} w W_{J'}$.
  One can also verify that $(I',J')$ is a lex-minimal pair for $w$ by
  Corollary~\ref{cor:lex-min-fillings} or directly from
  Theorem~\ref{T:minimal}.
\end{example}

\begin{figure}
\begin{tikzpicture}[scale = .7]
    \draw[fill,color=myblue,rounded corners] (.75,.8) rectangle (5.25,1.2);
    \draw[fill,color=myblue,rounded corners] (6.75,.8) rectangle (7.25,1.2);
    \draw[fill,color=myblue,rounded corners] (8.75,.8) rectangle (10.25,1.2);

    \draw[fill,color=myyellow,rounded corners] (1.75,-.2) rectangle (3.25,.2);
    \draw[fill,color=myyellow,rounded corners] (4.75,-.2) rectangle (9.25,.2);
    \draw[fill,color=myyellow,rounded corners] (10.75,-.2) rectangle (12.25,.2);
    \draw[fill,color=myyellow,rounded corners] (13.75,-.2) rectangle (15.25,.2);

    \draw (1,1) -- (14,1);
    \draw (2,0) -- (15,0);

    \foreach \x in {1,...,14}
        \draw (\x,1) -- (\x+1,0);

    \foreach \x in {1,...,15} {
        \draw[fill=white] (\x,1) circle (.3ex);
        \draw[fill=white] (\x,0) circle (.3ex);
    }

    \foreach \x in {1,2,3,4,5,7,9,10}
        \draw[fill=black] (\x,1) circle (.3ex);

    \foreach \x in {2,3,5,6,7,8,9,11,12,14,15}
        \draw[fill=black] (\x,0) circle (.3ex);

    \draw (15,1) node[cross=.5ex]{};
    \draw (1,0) node[cross=.5ex]{};

    \draw[line width=1mm,color=black,->] (2.25,0.25) -- (1.75,.75);
    \draw[line width=1mm,color=black,->] (7.1,0.75) -- (7.6,.25);
    \draw[line width=1mm,color=black,->] (14.25,0.25) -- (13.75,.75);
\end{tikzpicture}
\caption{ For Example~\ref{ex:plank.moves},
  the connected components of $(I,J)$ along with three possible plank
  moves are shown on the $w$-ocean.}
\label{fig:slide.1}
\end{figure}
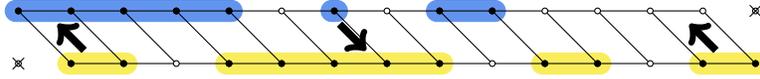

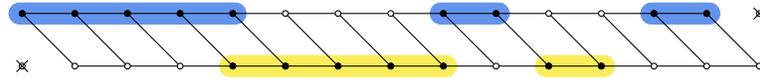
\begin{figure}
\begin{tikzpicture}[scale = .7]
    \draw[fill,color=myblue,rounded corners] (.75,.8) rectangle (5.25,1.2);
    \draw[fill,color=myblue,rounded corners] (8.75,.8) rectangle (10.25,1.2);
    \draw[fill,color=myblue,rounded corners] (12.75,.8) rectangle (14.25,1.2);

    \draw[fill,color=myyellow,rounded corners] (4.75,-.2) rectangle (9.25,.2);
    \draw[fill,color=myyellow,rounded corners] (10.75,-.2) rectangle (12.25,.2);

    \draw (1,1) -- (14,1);
    \draw (2,0) -- (15,0);

    \foreach \x in {1,...,14}
        \draw (\x,1) -- (\x+1,0);

    \foreach \x in {1,...,15} {
        \draw[fill=white] (\x,1) circle (.3ex);
        \draw[fill=white] (\x,0) circle (.3ex);
    }

    \foreach \x in {1,2,3,4,5,9,10,13,14}
        \draw[fill=black] (\x,1) circle (.3ex);

    \foreach \x in {5,6,7,8,9,11,12}
        \draw[fill=black] (\x,0) circle (.3ex);

    \draw (15,1) node[cross=.5ex]{};
    \draw (1,0) node[cross=.5ex]{};
\end{tikzpicture}
\caption{ For Example~\ref{ex:plank.moves}, the connected components
  of $(I',J')$ are shown on the $w$-ocean.  This pair is lex-minimal.}
\label{fig:slide.2}
\end{figure}

\subsection{Lex-minimal presentations near wharfs in star Coxeter groups}\label{sub:wharfs}

As a warm-up, let $(W,S)$ be a Coxeter system whose Coxeter graph is a star,
specifically $s_1$ is the central vertex and $s_2,\ldots, s_n$ for
$n\geq 4$ are adjacent to $s_1$ and nothing else.  Consider the
identity element $e \in W$.  Every $1\leq i \leq n$ indexes a small
ascent for $e$, and thus $\{(s_{1},0), (s_{1},1) \}$ is a wharf in
the $e$-ocean.  In this subsection, we study the enumeration of
lex-minimal presentations $(I,e,J)$ for such Coxeter groups in order
to determine the valid neighborhoods for wharfs in lex-minimal
presentations for any $w$ and any Coxeter group.

Partition the set of lex-minimal presentations $(I,e,J)$ according
to which nodes are selected at the wharf. We have four choices: we either
include or exclude the node on the upper level, and we either include or
exclude the node on the lower level.  Denote these four options by \omo, \omx,
\xmo, \xmx\ where again a filled dot means that it is selected to be in $I$ or
$J$ depending on if it is on the bottom or on the top, respectively.

By Theorem~\ref{T:minimal}, the allowed pairs for lex-minimal
pairs $(I, J)\subset S \times S$ for $e$ in the neighborhood of a wharf
are characterized as follows.
\begin{enumerate}[(a)]
\item If $s_1\in I$, then $I$ and $J$ must be incomparable in subset
  order.
\item If $s_1\not \in I$, then $J$ can be any subset, and either $I$ is
  empty or $I$ is not comparable to $J$.
\end{enumerate}

\begin{example} Consider the case of a star Coxeter graph with four
vertices. If the edge labels are all 3, this Coxeter group is type
$D_4$, and $c_{e} =72$.  Up to symmetry of the three leaves around
the central vertex, there are twenty-four distinct types of allowable
lex-minimal presentations $(I,e,J)$ as shown in Figure~\ref{fig:d4}.
\end{example}

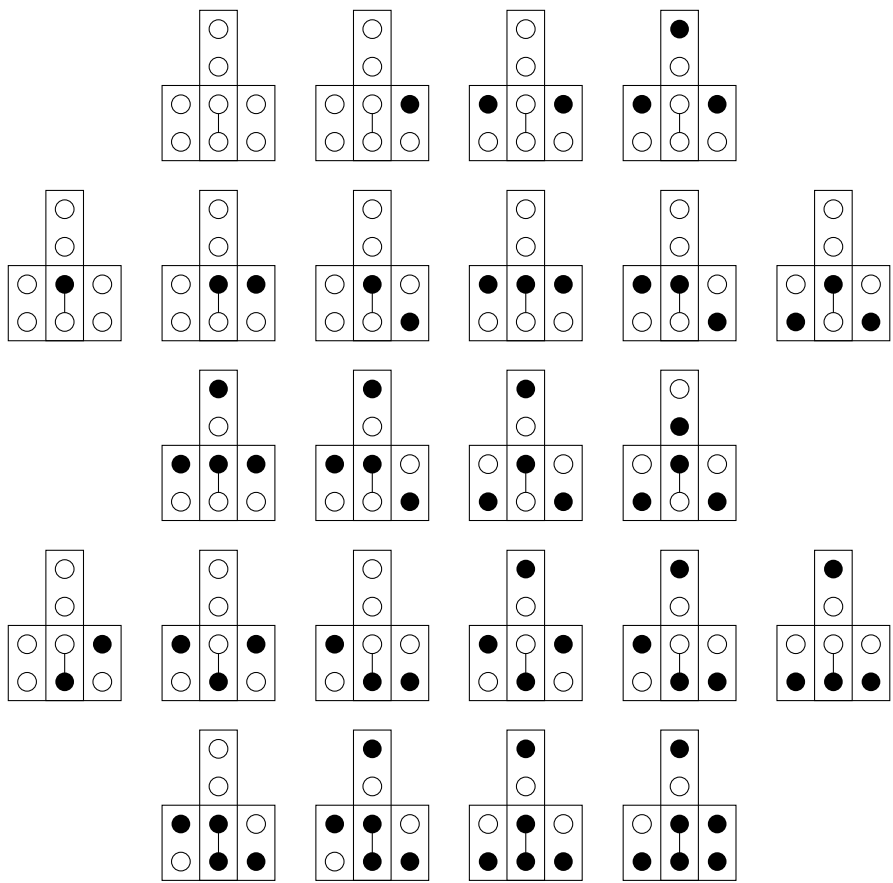
\begin{figure}[htbp]
 \begin{tikzpicture}[scale=.5,baseline=.1cm]
    \draw (-.5,-.5) rectangle (2.5,1.5);
    \draw (.5,-.5) rectangle (1.5,3.5);
     \draw (1,0.15) -- (1,0.85);
      \draw (1,0) node [circle, draw=black,fill=white, inner sep=2.5] {};
    \draw (1,1) node [circle, draw=black,fill=white, inner sep=2.5] {};
  \draw (1,2) node [circle, draw=black,fill=white, inner sep=2.5] {};
  \draw (1,3) node [circle, draw=black,fill=white, inner sep=2.5] {};
  \draw (0,0) node [circle, draw=black,fill=white, inner sep=2.5] {};
  \draw (0,1) node [circle, draw=black,fill=white, inner sep=2.5] {};
  \draw (2,0) node [circle, draw=black,fill=white, inner sep=2.5] {};
  \draw (2,1) node [circle, draw=black,fill=white, inner sep=2.5] {};
\end{tikzpicture}
\hspace{.1in}
 \begin{tikzpicture}[scale=.5,baseline=.1cm]
    \draw (-.5,-.5) rectangle (2.5,1.5);
    \draw (.5,-.5) rectangle (1.5,3.5);
     \draw (1,0.15) -- (1,0.85);
      \draw (1,0) node [circle, draw=black,fill=white, inner sep=2.5] {};
    \draw (1,1) node [circle, draw=black,fill=white, inner sep=2.5] {};
  \draw (1,2) node [circle, draw=black,fill=white, inner sep=2.5] {};
  \draw (1,3) node [circle, draw=black,fill=white, inner sep=2.5] {};
  \draw (0,0) node [circle, draw=black,fill=white, inner sep=2.5] {};
  \draw (0,1) node [circle, draw=black,fill=white, inner sep=2.5] {};
  \draw (2,0) node [circle, draw=black,fill=white, inner sep=2.5] {};
  \draw (2,1) node [circle,fill=black, inner sep=2.5] {};
\end{tikzpicture}
\hspace{.1in}
 \begin{tikzpicture}[scale=.5,baseline=.1cm]
    \draw (-.5,-.5) rectangle (2.5,1.5);
    \draw (.5,-.5) rectangle (1.5,3.5);
     \draw (1,0.15) -- (1,0.85);
      \draw (1,0) node [circle, draw=black,fill=white, inner sep=2.5] {};
    \draw (1,1) node [circle, draw=black,fill=white, inner sep=2.5] {};
  \draw (1,2) node [circle, draw=black,fill=white, inner sep=2.5] {};
  \draw (1,3) node [circle, draw=black,fill=white, inner sep=2.5] {};
  \draw (0,0) node [circle, draw=black,fill=white, inner sep=2.5] {};
  \draw (0,1) node [circle,fill=black, inner sep=2.5] {};
  \draw (2,0) node [circle, draw=black,fill=white, inner sep=2.5] {};
  \draw (2,1) node [circle,fill=black, inner sep=2.5] {};
\end{tikzpicture}
\hspace{.1in}
 \begin{tikzpicture}[scale=.5,baseline=.1cm]
    \draw (-.5,-.5) rectangle (2.5,1.5);
    \draw (.5,-.5) rectangle (1.5,3.5);
     \draw (1,0.15) -- (1,0.85);
      \draw (1,0) node [circle, draw=black,fill=white, inner sep=2.5] {};
    \draw (1,1) node [circle, draw=black,fill=white, inner sep=2.5] {};
  \draw (1,2) node [circle, draw=black,fill=white, inner sep=2.5] {};
  \draw (1,3) node [circle,fill=black, inner sep=2.5] {};
  \draw (0,0) node [circle, draw=black,fill=white, inner sep=2.5] {};
  \draw (0,1) node [circle,fill=black, inner sep=2.5] {};
  \draw (2,0) node [circle, draw=black,fill=white, inner sep=2.5] {};
  \draw (2,1) node [circle,fill=black, inner sep=2.5] {};
\end{tikzpicture}

\bigskip
 \begin{tikzpicture}[scale=.5,baseline=.1cm]
    \draw (-.5,-.5) rectangle (2.5,1.5);
    \draw (.5,-.5) rectangle (1.5,3.5);
     \draw (1,0.15) -- (1,0.85);
      \draw (1,0) node [circle, draw=black,fill=white, inner sep=2.5] {};
    \draw (1,1) node [circle,fill=black, inner sep=2.5] {};
  \draw (1,2) node [circle, draw=black,fill=white, inner sep=2.5] {};
  \draw (1,3) node [circle, draw=black,fill=white, inner sep=2.5] {};
  \draw (0,0) node [circle, draw=black,fill=white, inner sep=2.5] {};
  \draw (0,1) node [circle, draw=black,fill=white, inner sep=2.5] {};
  \draw (2,0) node [circle, draw=black,fill=white, inner sep=2.5] {};
  \draw (2,1) node [circle, draw=black,fill=white, inner sep=2.5] {};
\end{tikzpicture}
\hspace{.1in}
 \begin{tikzpicture}[scale=.5,baseline=.1cm]
    \draw (-.5,-.5) rectangle (2.5,1.5);
    \draw (.5,-.5) rectangle (1.5,3.5);
     \draw (1,0.15) -- (1,0.85);
      \draw (1,0) node [circle, draw=black,fill=white, inner sep=2.5] {};
    \draw (1,1) node [circle,fill=black, inner sep=2.5] {};
  \draw (1,2) node [circle, draw=black,fill=white, inner sep=2.5] {};
  \draw (1,3) node [circle, draw=black,fill=white, inner sep=2.5] {};
  \draw (0,0) node [circle, draw=black,fill=white, inner sep=2.5] {};
  \draw (0,1) node [circle, draw=black,fill=white, inner sep=2.5] {};
  \draw (2,0) node [circle, draw=black,fill=white, inner sep=2.5] {};
  \draw (2,1) node [circle,fill=black, inner sep=2.5] {};
\end{tikzpicture}
\hspace{.1in}
 \begin{tikzpicture}[scale=.5,baseline=.1cm]
    \draw (-.5,-.5) rectangle (2.5,1.5);
    \draw (.5,-.5) rectangle (1.5,3.5);
     \draw (1,0.15) -- (1,0.85);
      \draw (1,0) node [circle, draw=black,fill=white, inner sep=2.5] {};
    \draw (1,1) node [circle,fill=black, inner sep=2.5] {};
  \draw (1,2) node [circle, draw=black,fill=white, inner sep=2.5] {};
  \draw (1,3) node [circle, draw=black,fill=white, inner sep=2.5] {};
  \draw (0,0) node [circle, draw=black,fill=white, inner sep=2.5] {};
  \draw (0,1) node [circle, draw=black,fill=white, inner sep=2.5] {};
  \draw (2,0) node [circle,fill=black, inner sep=2.5] {};
  \draw (2,1) node [circle, draw=black,fill=white, inner sep=2.5] {};
\end{tikzpicture}
\hspace{.1in}
 \begin{tikzpicture}[scale=.5,baseline=.1cm]
    \draw (-.5,-.5) rectangle (2.5,1.5);
    \draw (.5,-.5) rectangle (1.5,3.5);
     \draw (1,0.15) -- (1,0.85);
      \draw (1,0) node [circle, draw=black,fill=white, inner sep=2.5] {};
    \draw (1,1) node [circle,fill=black, inner sep=2.5] {};
  \draw (1,2) node [circle, draw=black,fill=white, inner sep=2.5] {};
  \draw (1,3) node [circle, draw=black,fill=white, inner sep=2.5] {};
  \draw (0,0) node [circle, draw=black,fill=white, inner sep=2.5] {};
  \draw (0,1) node [circle,fill=black, inner sep=2.5] {};
  \draw (2,0) node [circle, draw=black,fill=white, inner sep=2.5] {};
  \draw (2,1) node [circle,fill=black, inner sep=2.5] {};
 \end{tikzpicture}
 \hspace{.1in}
 \begin{tikzpicture}[scale=.5,baseline=.1cm]
    \draw (-.5,-.5) rectangle (2.5,1.5);
    \draw (.5,-.5) rectangle (1.5,3.5);
     \draw (1,0.15) -- (1,0.85);
      \draw (1,0) node [circle, draw=black,fill=white, inner sep=2.5] {};
    \draw (1,1) node [circle,fill=black, inner sep=2.5] {};
  \draw (1,2) node [circle, draw=black,fill=white, inner sep=2.5] {};
  \draw (1,3) node [circle, draw=black,fill=white, inner sep=2.5] {};
  \draw (0,0) node [circle, draw=black,fill=white, inner sep=2.5] {};
  \draw (0,1) node [circle,fill=black, inner sep=2.5] {};
  \draw (2,0) node [circle,fill=black, inner sep=2.5] {};
  \draw (2,1) node [circle, draw=black,fill=white, inner sep=2.5] {};
\end{tikzpicture}
\hspace{.1in}
 \begin{tikzpicture}[scale=.5,baseline=.1cm]
    \draw (-.5,-.5) rectangle (2.5,1.5);
    \draw (.5,-.5) rectangle (1.5,3.5);
     \draw (1,0.15) -- (1,0.85);
      \draw (1,0) node [circle, draw=black,fill=white, inner sep=2.5] {};
    \draw (1,1) node [circle,fill=black, inner sep=2.5] {};
  \draw (1,2) node [circle, draw=black,fill=white, inner sep=2.5] {};
  \draw (1,3) node [circle, draw=black,fill=white, inner sep=2.5] {};
  \draw (0,0) node [circle,fill=black, inner sep=2.5] {};
  \draw (0,1) node [circle, draw=black,fill=white, inner sep=2.5] {};
  \draw (2,0) node [circle,fill=black, inner sep=2.5] {};
  \draw (2,1) node [circle, draw=black,fill=white, inner sep=2.5] {};
\end{tikzpicture}

\bigskip
 \begin{tikzpicture}[scale=.5,baseline=.1cm]
    \draw (-.5,-.5) rectangle (2.5,1.5);
    \draw (.5,-.5) rectangle (1.5,3.5);
     \draw (1,0.15) -- (1,0.85);
      \draw (1,0) node [circle, draw=black,fill=white, inner sep=2.5] {};
    \draw (1,1) node [circle,fill=black, inner sep=2.5] {};
  \draw (1,2) node [circle, draw=black,fill=white, inner sep=2.5] {};
  \draw (1,3) node [circle,fill=black, inner sep=2.5] {};
  \draw (0,0) node [circle, draw=black,fill=white, inner sep=2.5] {};
  \draw (0,1) node [circle,fill=black, inner sep=2.5] {};
  \draw (2,0) node [circle, draw=black,fill=white, inner sep=2.5] {};
  \draw (2,1) node [circle,fill=black, inner sep=2.5] {};
\end{tikzpicture}
\hspace{.1in}
 \begin{tikzpicture}[scale=.5,baseline=.1cm]
    \draw (-.5,-.5) rectangle (2.5,1.5);
    \draw (.5,-.5) rectangle (1.5,3.5);
     \draw (1,0.15) -- (1,0.85);
      \draw (1,0) node [circle, draw=black,fill=white, inner sep=2.5] {};
    \draw (1,1) node [circle,fill=black, inner sep=2.5] {};
  \draw (1,2) node [circle, draw=black,fill=white, inner sep=2.5] {};
  \draw (1,3) node [circle,fill=black, inner sep=2.5] {};
  \draw (0,0) node [circle, draw=black,fill=white, inner sep=2.5] {};
  \draw (0,1) node [circle,fill=black, inner sep=2.5] {};
  \draw (2,0) node [circle,fill=black, inner sep=2.5] {};
  \draw (2,1) node [circle, draw=black,fill=white, inner sep=2.5] {};
\end{tikzpicture}
\hspace{.1in}
 \begin{tikzpicture}[scale=.5,baseline=.1cm]
    \draw (-.5,-.5) rectangle (2.5,1.5);
    \draw (.5,-.5) rectangle (1.5,3.5);
     \draw (1,0.15) -- (1,0.85);
      \draw (1,0) node [circle, draw=black,fill=white, inner sep=2.5] {};
    \draw (1,1) node [circle,fill=black, inner sep=2.5] {};
  \draw (1,2) node [circle, draw=black,fill=white, inner sep=2.5] {};
  \draw (1,3) node [circle,fill=black, inner sep=2.5] {};
  \draw (0,0) node [circle,fill=black, inner sep=2.5] {};
  \draw (0,1) node [circle, draw=black,fill=white, inner sep=2.5] {};
  \draw (2,0) node [circle,fill=black, inner sep=2.5] {};
  \draw (2,1) node [circle, draw=black,fill=white, inner sep=2.5] {};
\end{tikzpicture}
\hspace{.1in}
 \begin{tikzpicture}[scale=.5,baseline=.1cm]
    \draw (-.5,-.5) rectangle (2.5,1.5);
    \draw (.5,-.5) rectangle (1.5,3.5);
     \draw (1,0.15) -- (1,0.85);
      \draw (1,0) node [circle, draw=black,fill=white, inner sep=2.5] {};
    \draw (1,1) node [circle,fill=black, inner sep=2.5] {};
  \draw (1,2) node [circle,fill=black, inner sep=2.5] {};
  \draw (1,3) node [circle, draw=black,fill=white, inner sep=2.5] {};
  \draw (0,0) node [circle,fill=black, inner sep=2.5] {};
  \draw (0,1) node [circle, draw=black,fill=white, inner sep=2.5] {};
  \draw (2,0) node [circle,fill=black, inner sep=2.5] {};
  \draw (2,1) node [circle, draw=black,fill=white, inner sep=2.5] {};
\end{tikzpicture}

\bigskip
 \begin{tikzpicture}[scale=.5,baseline=.1cm]
    \draw (-.5,-.5) rectangle (2.5,1.5);
    \draw (.5,-.5) rectangle (1.5,3.5);
     \draw (1,0.15) -- (1,0.85);
      \draw (1,0) node [circle,fill=black, inner sep=2.5] {};
    \draw (1,1) node [circle, draw=black,fill=white, inner sep=2.5] {};
  \draw (1,2) node [circle, draw=black,fill=white, inner sep=2.5] {};
  \draw (1,3) node [circle, draw=black,fill=white, inner sep=2.5] {};
  \draw (0,0) node [circle, draw=black,fill=white, inner sep=2.5] {};
  \draw (0,1) node [circle, draw=black,fill=white, inner sep=2.5] {};
  \draw (2,0) node [circle, draw=black,fill=white, inner sep=2.5] {};
  \draw (2,1) node [circle,fill=black, inner sep=2.5] {};
\end{tikzpicture}
\hspace{.1in}
 \begin{tikzpicture}[scale=.5,baseline=.1cm]
    \draw (-.5,-.5) rectangle (2.5,1.5);
    \draw (.5,-.5) rectangle (1.5,3.5);
     \draw (1,0.15) -- (1,0.85);
      \draw (1,0) node [circle,fill=black, inner sep=2.5] {};
    \draw (1,1) node [circle, draw=black,fill=white, inner sep=2.5] {};
  \draw (1,2) node [circle, draw=black,fill=white, inner sep=2.5] {};
  \draw (1,3) node [circle, draw=black,fill=white, inner sep=2.5] {};
  \draw (0,0) node [circle, draw=black,fill=white, inner sep=2.5] {};
  \draw (0,1) node [circle,fill=black, inner sep=2.5] {};
  \draw (2,0) node [circle, draw=black,fill=white, inner sep=2.5] {};
  \draw (2,1) node [circle,fill=black, inner sep=2.5] {};
\end{tikzpicture}
\hspace{.1in}
 \begin{tikzpicture}[scale=.5,baseline=.1cm]
    \draw (-.5,-.5) rectangle (2.5,1.5);
    \draw (.5,-.5) rectangle (1.5,3.5);
     \draw (1,0.15) -- (1,0.85);
      \draw (1,0) node [circle,fill=black, inner sep=2.5] {};
    \draw (1,1) node [circle, draw=black,fill=white, inner sep=2.5] {};
  \draw (1,2) node [circle, draw=black,fill=white, inner sep=2.5] {};
  \draw (1,3) node [circle, draw=black,fill=white, inner sep=2.5] {};
  \draw (0,0) node [circle, draw=black,fill=white, inner sep=2.5] {};
  \draw (0,1) node [circle,fill=black, inner sep=2.5] {};
  \draw (2,0) node [circle,fill=black, inner sep=2.5] {};
  \draw (2,1) node [circle, draw=black,fill=white, inner sep=2.5] {};
\end{tikzpicture}
\hspace{.1in}
 \begin{tikzpicture}[scale=.5,baseline=.1cm]
    \draw (-.5,-.5) rectangle (2.5,1.5);
    \draw (.5,-.5) rectangle (1.5,3.5);
     \draw (1,0.15) -- (1,0.85);
      \draw (1,0) node [circle,fill=black, inner sep=2.5] {};
    \draw (1,1) node [circle, draw=black,fill=white, inner sep=2.5] {};
  \draw (1,2) node [circle, draw=black,fill=white, inner sep=2.5] {};
  \draw (1,3) node [circle,fill=black, inner sep=2.5] {};
  \draw (0,0) node [circle, draw=black,fill=white, inner sep=2.5] {};
  \draw (0,1) node [circle,fill=black, inner sep=2.5] {};
  \draw (2,0) node [circle, draw=black,fill=white, inner sep=2.5] {};
  \draw (2,1) node [circle,fill=black, inner sep=2.5] {};
\end{tikzpicture}
\hspace{.1in}
 \begin{tikzpicture}[scale=.5,baseline=.1cm]
    \draw (-.5,-.5) rectangle (2.5,1.5);
    \draw (.5,-.5) rectangle (1.5,3.5);
     \draw (1,0.15) -- (1,0.85);
      \draw (1,0) node [circle,fill=black, inner sep=2.5] {};
    \draw (1,1) node [circle, draw=black,fill=white, inner sep=2.5] {};
  \draw (1,2) node [circle, draw=black,fill=white, inner sep=2.5] {};
  \draw (1,3) node [circle,fill=black, inner sep=2.5] {};
  \draw (0,0) node [circle, draw=black,fill=white, inner sep=2.5] {};
  \draw (0,1) node [circle,fill=black, inner sep=2.5] {};
  \draw (2,0) node [circle,fill=black, inner sep=2.5] {};
  \draw (2,1) node [circle, draw=black,fill=white, inner sep=2.5] {};
\end{tikzpicture}
\hspace{.1in}
 \begin{tikzpicture}[scale=.5,baseline=.1cm]
    \draw (-.5,-.5) rectangle (2.5,1.5);
    \draw (.5,-.5) rectangle (1.5,3.5);
     \draw (1,0.15) -- (1,0.85);
      \draw (1,0) node [circle,fill=black, inner sep=2.5] {};
    \draw (1,1) node [circle, draw=black,fill=white, inner sep=2.5] {};
  \draw (1,2) node [circle, draw=black,fill=white, inner sep=2.5] {};
  \draw (1,3) node [circle,fill=black, inner sep=2.5] {};
  \draw (0,0) node [circle,fill=black, inner sep=2.5] {};
  \draw (0,1) node [circle, draw=black,fill=white, inner sep=2.5] {};
  \draw (2,0) node [circle,fill=black, inner sep=2.5] {};
  \draw (2,1) node [circle, draw=black,fill=white, inner sep=2.5] {};
\end{tikzpicture}

\bigskip

 \begin{tikzpicture}[scale=.5,baseline=.1cm]
    \draw (-.5,-.5) rectangle (2.5,1.5);
    \draw (.5,-.5) rectangle (1.5,3.5);
     \draw (1,0.15) -- (1,0.85);
      \draw (1,0) node [circle,fill=black, inner sep=2.5] {};
    \draw (1,1) node [circle,fill=black, inner sep=2.5] {};
  \draw (1,2) node [circle, draw=black,fill=white, inner sep=2.5] {};
  \draw (1,3) node [circle, draw=black,fill=white, inner sep=2.5] {};
  \draw (0,0) node [circle, draw=black,fill=white, inner sep=2.5] {};
  \draw (0,1) node [circle,fill=black, inner sep=2.5] {};
  \draw (2,0) node [circle,fill=black, inner sep=2.5] {};
  \draw (2,1) node [circle, draw=black,fill=white, inner sep=2.5] {};
\end{tikzpicture}
\hspace{.1in}
 \begin{tikzpicture}[scale=.5,baseline=.1cm]
    \draw (-.5,-.5) rectangle (2.5,1.5);
    \draw (.5,-.5) rectangle (1.5,3.5);
     \draw (1,0.15) -- (1,0.85);
      \draw (1,0) node [circle,fill=black, inner sep=2.5] {};
    \draw (1,1) node [circle,fill=black, inner sep=2.5] {};
  \draw (1,2) node [circle, draw=black,fill=white, inner sep=2.5] {};
  \draw (1,3) node [circle,fill=black, inner sep=2.5] {};
  \draw (0,0) node [circle, draw=black,fill=white, inner sep=2.5] {};
  \draw (0,1) node [circle,fill=black, inner sep=2.5] {};
  \draw (2,0) node [circle,fill=black, inner sep=2.5] {};
  \draw (2,1) node [circle, draw=black,fill=white, inner sep=2.5] {};
\end{tikzpicture}
\hspace{.1in}
 \begin{tikzpicture}[scale=.5,baseline=.1cm]
    \draw (-.5,-.5) rectangle (2.5,1.5);
    \draw (.5,-.5) rectangle (1.5,3.5);
     \draw (1,0.15) -- (1,0.85);
      \draw (1,0) node [circle,fill=black, inner sep=2.5] {};
    \draw (1,1) node [circle,fill=black, inner sep=2.5] {};
  \draw (1,2) node [circle, draw=black,fill=white, inner sep=2.5] {};
  \draw (1,3) node [circle,fill=black, inner sep=2.5] {};
  \draw (0,0) node [circle,fill=black, inner sep=2.5] {};
  \draw (0,1) node [circle, draw=black,fill=white, inner sep=2.5] {};
  \draw (2,0) node [circle,fill=black, inner sep=2.5] {};
  \draw (2,1) node [circle, draw=black,fill=white, inner sep=2.5] {};
\end{tikzpicture}
\hspace{.1in}
 \begin{tikzpicture}[scale=.5,baseline=.1cm]
    \draw (-.5,-.5) rectangle (2.5,1.5);
    \draw (.5,-.5) rectangle (1.5,3.5);
     \draw (1,0.15) -- (1,0.85);
      \draw (1,0) node [circle,fill=black, inner sep=2.5] {};
    \draw (1,1) node [circle,fill=black, inner sep=2.5] {};
  \draw (1,2) node [circle, draw=black,fill=white, inner sep=2.5] {};
  \draw (1,3) node [circle,fill=black, inner sep=2.5] {};
  \draw (0,0) node [circle,fill=black, inner sep=2.5] {};
  \draw (0,1) node [circle, draw=black,fill=white, inner sep=2.5] {};
  \draw (2,0) node [circle,fill=black, inner sep=2.5] {};
  \draw (2,1) node [circle,fill=black, inner sep=2.5] {};
\end{tikzpicture}
\caption{Lex-minimal pairs for $D_4$, up to symmetry around the central
tile.}\label{fig:d4}
\end{figure}

\subsection{Enumeration for Coxeter groups}

Given an element $w$ in a general Coxeter group $W$, the $w$-ocean
formed on its ascents consists of rafts, tethers, ropes, floats, and
wharfs.  To enumerate the lex-minimal pairs
$(I,J) \in \asc(w^{-1}) \times \asc(w)$, we partition the set of
lex-minimal pairs according to which tethers, ropes, and wharfs are
included in $I$ and $J$. Every possible selection of tethers, ropes,
and wharfs leads to a nonempty set of lex-minimal pairs. After fixing
such a selection, we further partition the set of lex-minimal pairs by
adding local conditions around the wharfs which allow us to once again
reduce the enumeration to the fillings of rafts avoiding the forbidden
patterns in Figure~\ref{fig:avoid} and the automaton in
Figure~\ref{fig:aut}.

The purpose of this section is to describe this process in detail, starting
with the selection of tethers, ropes and wharfs. For each tether and rope, we
have two choices: we either include it or not, as in the case of $S_n$.
For each wharf, we have four choices: we either include or exclude the node in
the upper row, and we either include or exclude the node in the lower row. As
with $S_n$, we indicate our choices by filling in the vertices on the
$w$-ocean.  For the sake of discussion, we will denote our selection by $C$.

After choosing $C$, we need to make additional choices on the way that the
rafts are filled near wharfs. These conditions are recorded on a new graph
constructed from the $w$-ocean, which we call a harbor. Before we define a
harbor, consider the rafts in the $w$-ocean. It is convenient to think of rafts
as a path of planks, each of which has a node at the top and at the bottom. If
a raft has size greater than one, then it has two distinct endplanks, and each
of these endplanks can be adjacent to a wharf, or to some collection of ropes
and tethers on the top and/or bottom. Note that, in this case, if one of these
endplanks is adjacent to a wharf then it cannot be adjacent to a rope or a
tether, since that would make it a wharf.  For rafts consisting of a single
plank, then that plank can either be adjacent to two wharfs, or to at most one
wharf and some (possibly empty) collection of ropes and/or tethers. To treat
rafts in a uniform fashion, we will think of rafts of size one as having two
logical endplanks. In this way, we can divide adjacent wharfs, ropes, and
tethers among the two logical endplanks so that no endplank is adjacent to more
than one wharf, and no endplank is adjacent to both a wharf and a rope or
tether.

Roughly speaking, a harbor graph is defined by thinking of the rafts
as edges in a new graph. We now make this precise:
\begin{defn}\label{defn:harbor}
    For a choice $C$ of fillings of the vertices in the $w$-ocean corresponding
    to tethers, ropes, and wharfs, we define a simple graph $H_{C}$, called
    a \emph{harbor}, as follows: First, the harbor has a vertex for each
    wharf of the $w$-ocean, and one vertex for each endplank of a raft which is
    not adjacent to a wharf. As mentioned above, we think of rafts of size one as
    having two logical endplanks; if such a raft is adjacent to only one wharf, then
    we add one endplank vertex, and if the raft is not adjacent to any wharf,
    then we add two endplank vertices. The harbor also has an edge for every
    raft in the $w$-ocean.  This edge is incident to a vertex of the harbor if
    and only if the vertex corresponds to a wharf which is adjacent to the
    endplank of the raft in question, or the vertex corresponds to the endplank
    of the raft.  In addition, the harbor has a vertex and edge for each pair
    $(w,r)$, where $w$ is a wharf and $r$ is a selected rope or tether adjacent
    to $w$.  The edge connects this additional vertex to the vertex
    corresponding to $w$. Finally, we connect two wharfs by an edge if they
    are adjacent in the $w$-ocean.
\end{defn}

Next we describe the edge, vertex and half edge decorations on $H_C$.
The edge and vertex decorations are completely determined by the
$w$-ocean and the choice $C$.  In contrast, there are different
possible ways to decorate the half edges, and we will need to consider
all of them for the main enumeration formula.  This step will
necessarily be more complicated.  The reader may wish to look ahead to
Theorem~\ref{main.coxeter} and the example in Figure~\ref{branch}.

Observe that every vertex in the graph $H_{C}$ is connected to at
least one edge by construction.  Each edge represents a raft of some
size, possibly of size 0.  Rafts of size 0 come from edges connecting
two wharfs or a $(w,r)$ pair.  Every edge of the harbor $H_C$ is
decorated with the integer corresponding to the size of the
corresponding raft.

The vertices of the harbor $H_C$ are decorated with tiles from
    \begin{equation*}
        \Big\{ \omo\,, \omx\,, \xmo\,, \xmx\,, \oo\,, \ox\,, \xo\,, \xx \Big\}.
    \end{equation*}
    Wharfs are decorated with tiles \omo\,, \omx\,, \xmo\,, and \xmx\,
    according to which nodes of the wharf are selected.  Vertices
    corresponding to $(w,r)$ are decorated with \ox\, or \xo\,
    depending on whether the selected rope or tether $r$ is on the top
    or the bottom of the $w$-ocean.  Endplanks of rafts are decorated
    with tiles \oo\,, \ox\,, \xo\,, and \xx\,, where the top
    (resp. bottom) node of the tile is filled if the endplank is
    adjacent to any selected rope or tether on the top (resp. bottom).
    In this way, a selected tether is split into many ropes, and then
    selected ropes adjacent to the same endplank are amalgamated.

    For rafts of size one, we must again take special care for
    arbitrary Coxeter groups---if the raft is not adjacent to any
    wharf, then we can split adjacent ropes and tethers arbitrarily
    among the two endplanks before applying the above recipe without
    changing the lex-minimal conditions by Theorem~\ref{T:minimal}.
    For instance, we can assign all ropes and tethers to one endplank,
    meaning that we label that endplank as above, and then label the
    other endplank by {\oo}.

    Each half-edge of the harbor $H_C$ is decorated by one of the
    labels from the set $\mathcal{L} \cup \mathcal{L}'$ where
$$ \mathcal{L} = \raisebox{.5ex}{\Big\{}{\oo}\, , \ {\xo}\, , \ {\ox}\, , \ {\xx}\, , \ {\oX}\, , \ {\Ox}\, ,
\ {\Xx}\, , \ {\xX}\, , \ {\XX}\raisebox{.5ex}{\Big\}} \text{    and
}
\mathcal{L}' = \raisebox{.5ex}{\Big\{}{\aox}\, , \ {\axx}\raisebox{.5ex}{\Big\}}
$$
Not all possible labellings of the half-edges are allowed for $H_C$.
We refer to the labellings which are allowed as \emph{legal
  labellings}.  Legal labellings encode the local conditions on
lex-minimal fillings in the neighborhoods of wharfs which are similar
to those found in Section~\ref{sub:wharfs}. 
The idea is that the half-edge labels are used to specify the boundary
apparatus at each end of the corresponding raft in the $w$-ocean, and
then the enumeration of lex-minimal fillings reduces to finding the
number of lex-minimal fillings for each raft with the specified
boundary apparatus. This in turn translates to a walk on the automata
of Section~\ref{sec:Snenumeration}.  Furthermore, the half-edge labels
indicate specific selections of nodes on the initial or final segment
of rafts in certain lex-minimal pairs $(I,J)$ consistent with the
selection $C$. Thus, they also allow us to refer back to the ``top''
and ``bottom'' rows of the $w$-ocean, even though the harbor doesn't
have ``top'' or ``bottom'' vertices.

We now explain what makes a labeling legal; this also gives us an
opportunity to explain what each label means. Afterwards we give a
more concise formal definition.  Throughout this discussion we will
assume $(I,J)$ is a lex-minimal pair for $w$ consistent with $C$ and
the given half edge labeling.  Recall no contraction or upward slide
moves apply to $(I,J)$ by Corollary~\ref{cor:lex-min-fillings}.  The
pair $(I,J)$ determines an induced subgraph $G_C(I,J)$ of the
$w$-ocean with vertices $I$ on the bottom and $J$ on the top, and we
will frequently refer to the connected components of this subgraph.
Also, it will be useful to recall the meaning of the tiles on planks
in rafts from the type $A$ case.  For example, if we have a tile of
the form $\ox$ on the first plank of a raft, that means that the lower
endpoint of the plank is in $I$ and its upper endpoint is not in $J$.

\begin{enumerate}
\item Vertices of type \omo\ indicate a wharf with neither top nor
bottom node chosen in $C$. Thus they behave just like the boundary
apparatus \oo\ when filling the adjacent rafts. 
We label all of the half-edges emanating from a \omo\ vertex by $\oo$.  Recall
that $\oo$ is vertex 2 in the automaton from Section~\ref{sec:Snenumeration}.
Any walk in the automaton starting or ending at vertex 2 corresponds to an
initial or final segment of a lex-minimal filling of an adjacent raft.

\item Vertices of type \xmo\ indicate a wharf with only the top
node selected.
The selected node may join connected components of $G_C(I,J)$ on
the top in adjacent rafts. Because the bottom node of
the wharf is not selected, the connected component containing the
wharf will not be contained in any connected component on the bottom
(so the conditions of Theorem \ref{T:minimal} are satisfied).  Thus
wharfs of this type behave just like the boundary apparatus {\xo} when
filling the adjacent rafts. Thus we label all of the half-edges
emanating from a {\xmo} vertex by {\xo}.  Recall that {\xo} is vertex
3 in the automaton.  Any walk in the automaton starting or ending at
vertex 3 corresponds to an initial or final segment of a lex-minimal
filling of an adjacent raft.

\item Vertices of type {\omx} indicate a wharf with only the bottom
  node selected. As in the previous case, this selected node can join
  connected components of $G_C(I,J)$ on the bottom in adjacent
  rafts. Since the top node of the wharf is not selected, the
  connected component of $G_C(I,J)$ containing this wharf is not
  contained in a connected component on the top. To
prevent there  being an available upward slide move, 
we must also show either that the planks with bottom node in the
connected component contain or are adjacent to a
selected node on top, or that the connected component contains a large
ascent.  In either case, the selected top node or large ascent might
be connected to the wharf in question only after passing through
another wharf or sequence of wharfs. To keep track of the different
possibilities, we label the adjacent half-edges by {\oX}, {\Ox}, or
{\aox}. The meaning of these labels is as follows:

  \begin{enumerate} \item A half-edge label {\oX} means that on that
side of the corresponding raft, we have any nonnegative number of
$\ox$ tiles emanating from the wharf, after which we have a tile {\oo} or {\omo}. For
example, we may choose three intermediate tiles {\omx \ox \ox \ox \oo}, or
even zero intermediate tiles {\omx \oo}, along the raft corresponding to a
half-edge labeled {\oX}\,.  After the {\oo} tile, the allowed tiles follow the
automaton again, starting with vertex 2.  If the raft is oriented so that the
half-edge is at the terminal end of a raft, then the label implies that the
corresponding walks in the automaton terminate at vertex 5.

\item A half-edge label {\Ox} means that on that end of the
  corresponding raft, we have any nonnegative number of $\ox$ tiles
  moving away from the wharf, and then either a {\xx} tile, a {\xmx}
  tile, or a {\xo} tile. For example, we may choose three intermediate
  tiles $\omx \ox \ox \ox \xx$ or $\omx \ox \ox \ox \xo$\,, or even
  zero intermediate tiles {\omx \xx}, along the raft corresponding to
  a half-edge labeled {\oX}.  Note that in the automaton, after
  visiting vertex 5 \raisebox{.25ex}{\big(}\ox\raisebox{.25ex}{\big)},
  a walk must either visit vertex 3
  \raisebox{.25ex}{\big(}\xo\raisebox{.25ex}{\big)} or vertex 6
  \raisebox{.25ex}{\big(}\xx\raisebox{.25ex}{\big)}, and then proceed
  to visit vertex 3 before going on to other vertices.  Thus, this
  case is encoded in the automaton by starting with vertex 5 or ending
  at vertex 7.

  \item The half-edge label {\aox} is special, and means that every
    plank of the raft corresponding to that edge is tiled by {\ox}, and the other vertex of the edge is labeled by {\ox} or
    {\omx}. This can happen in two cases. Either:
    \begin{enumerate}
        \item the edge connects two wharfs, both labeled by {\omx}, and
            both half-edges of the edge are labeled by {\aox}; or
        \item the edge connects a wharf labeled by {\omx} to a vertex
            labeled by {\ox}. The half-edge incident to {\omx} is labeled
            by {\aox}, and the half-edge incident to {\ox} is labeled by
            {\ox}.
    \end{enumerate}
    We call a path in the harbor $H_{C}$ a {\aox}-\emph{path} if at least one
    half-edge of every edge of the path is labeled by {\aox}.
\end{enumerate}
To be a legal labeling, every vertex labeled by {\omx} must have an incident
half-edge labeled by {\Ox}, be connected by a {\aox}-path to a vertex with an
incident half-edge labeled by {\Ox}, or be connected by a {\aox}-path to a
vertex labeled by {\ox}. In the two former cases, the planks in the connected
component contain or are adjacent to a selected top node, while in
the latter case, the connected component will contain a large ascent.

\item Vertices of type \xmx\ indicate a wharf with both nodes selected. In this
    case, this wharf will join connected components in adjacent rafts on both
    the top and the bottom. For the conditions of Theorem \ref{T:minimal} to
    be satisfied, the connected component of vertices on the top must contain
    either a plank with only the top component selected, or a large ascent.
    The connected component of vertices on the bottom must satisfy the same
    condition. The wharf itself provides a plank whose bottom node is in the
    connected component on the bottom, and whose top node is selected, so this
    part of Theorem \ref{T:minimal} is automatically satisfied.  Once
    again, we have to consider several cases.  The different cases are
    distinguished by the tiles {\Xx}, {\xX}, {\XX} and {\axx}.
\begin{enumerate}
\item A label {\Xx} means that on that side, we choose some number of
doubly filled tiles, followed by {\xo} or {\xmo}. Thus, this case is encoded in the
automaton by starting with vertex 6 or ending at vertex 8.

\item A label {\xX} means that on that side, we choose some number of
doubly filled tiles, followed by {\ox} or {\omx}.  Thus, this case is encoded in the
automaton by starting with vertex 8 or ending at vertex 6.

\item A label {\XX} means that on that side, we choose any nonnegative
number of doubly filled tiles, followed by \oo\,.  Thus, if this label
appears at the initial end of the edge, the corresponding lex-minimal
conditions are encoded in the automaton by starting with some
nonnegative number of {\xx} tiles, followed by any walk starting at vertex 2.
If this label appears at the terminal end of the edge, then the walks must
all terminate at vertex 4.

\item The half-edge label {\axx} is special.  It means that every plank
  along the corresponding edge is tiled by {\xx}\, and the other
  vertex of the edge is labeled by {\xmx} or {\xx}. Therefore, it
  occurs in exactly two cases. Either:

\begin{enumerate}
\item it must connect a wharf of type {\xmx} to another wharf of the
same type, and both halves of the edge must have the {\axx} label; or
\item it must connect a wharf of type {\xmx} to a vertex of type
{\xx}, with the half-edge incident to {\xmx} labeled by {\axx}, and
the half-edge incident to {\xx} labeled by {\xx}.
\end{enumerate}
We call a path in the harbor $H_{C}$ a {\axx}-\emph{path} if every edge in the
path contains at least one half-edge label {\axx}\,.
\end{enumerate}

To be a legal labeling, every vertex labeled by {\xmx} must have an incident
half-edge labeled by {\Xx}, or be connected by a {\aox}-path to a vertex labeled
by {\xmx} with an incident half-edge labeled by {\Xx}, or be connected by an
{\aox}-path to a vertex labeled by {\xx}. In the two former cases, the top component
either contains a wharf with only the top node selected, or a large ascent (which one
depends on where the {\xo} or {\xmo} tile appears). In the latter case, the top component
contains a large ascent.
Similarly, every vertex labeled by {\xmx} must be incident to a half-edge
labeled by {\xX}, or be connected by a {\aox}-path to a vertex labeled by
{\xmx} with an incident half-edge labeled by {\xX}, or be connected by an
{\aox}-path to a vertex labeled by {\xx}.

\item The vertices of types \oo\,, \ox\,, \xo\,, and \xx \ all correspond to type
$A$ boundary apparatuses on rafts, and are adjacent to exactly one edge.
These impose the same initial/final conditions for lex-minimal fillings as
they did in Section~\ref{sec:Snenumeration}, so their half-edges are labeled by
the same symbol.
\end{enumerate}

We summarize this in a formal definition:

\begin{defn}\label{defn:legallabelling}
    Let $C$ be a selection of nodes for the wharfs, tethers, and ropes of a
    $w$-ocean. A labeling $L$ of the half-edges of the harbor $H_C$ is a
    \emph{legal labeling} if the following conditions are satisfied.
    \begin{enumerate}
        \item The half-edge labels are compatible with the vertex labels:
            \begin{enumerate}
             \item a half-edge coming from a vertex of type {\oo} or {\omo} is labeled by {\oo},
             \item a half-edge coming from a vertex of type {\xo} or {\xmo} is labeled by {\xo}, 
             \item a half-edge coming from a vertex of type {\ox} is labeled by {\ox}, 
             \item a half-edge coming from a vertex of type {\xx} is labeled by {\xx},
             \item a half-edge coming from a vertex of type {\omx} is labeled by
             {\oX}, {\Ox}, or {\aox}, 
             \item a half-edge coming from a vertex of type {\xmx} is labeled by
             {\Xx}, {\xX}, {\XX}, or {\axx}. 
             \end{enumerate}
        \item The special labels {\aox} and {\axx} are used correctly:
            \begin{enumerate}
                \item if a half-edge is labeled by {\aox} then the other
                    vertex is labeled by {\omx} or {\ox}, and the other half-edge
                    is labeled by {\aox} or {\ox} respectively; and
                \item if a half-edge is labeled by {\axx} then the other
                    vertex is labeled by {\xmx} or {\xx}, and the other half-edge
                    is labeled by {\axx} or {\xx} respectively.
            \end{enumerate}
        \item Several conditions to ensure lex-minimality are satisfied:
            \begin{enumerate}
                \item for every vertex labeled by {\omx}, there must be a {\aox}-path
                    (possibly of length zero) either to a vertex labeled by
                    {\omx} with an outgoing half-edge labeled by {\Ox}, or to
                    a vertex labeled by {\ox},
                \item for every vertex labeled by {\xmx}, there must be a {\axx}-path
                    (possibly of length zero) either to a vertex labeled by
                    {\xmx} with an outgoing half-edge labeled by {\Xx}, or to a vertex
                    labeled by {\xx}, and
                \item for every vertex labeled by {\xmx}, there must be a {\axx}-path
                    (possibly of length zero) either to a vertex labeled by
                    {\xmx} with an outgoing half-edge labeled by {\xX}, or to a vertex
                    labeled by {\xx}.
            \end{enumerate}
            As above, a {\aox}-path (resp. {\axx}-path) is a path in the
            harbor in which every edge has at least one half-edge labeled by
            {\aox} (resp. {\axx}).
    \end{enumerate}
\end{defn}

We also define what it means for a lex-minimal filling of the $w$-ocean to be
\emph{consistent} with a legal labeling. In the informal description of half-edge
labels above, we used tiles to refer to the selections of planks, and we now
formalize this with the notion of a \emph{tile sequence} associated to an edge.
\begin{defn}\label{defn:consistentwlabelling}
    Let $L$ be a legal labeling as outlined above, and let $F$ be a
    (not necessarily lex-minimal) filling of the $w$-ocean consistent with $C$. To every
    oriented edge $e$ (meaning an edge $e$ with choice of orientation) in the harbor,
    we associate a sequence of tiles which begins with the label of the initial
    half-edge, and
    ends with the label of the final half-edge. The tiles in between are drawn
    from  \{{\oo},{\xo},{\ox},{\xx}\}, and indicate the filling in $F$ of the
    corresponding raft (so the overall sequence has $n+2$ tiles, where $n$ is
    the label of the edge $\varepsilon$ in the harbor). We refer to this sequence as the
    \emph{tile sequence associated to $\varepsilon$}. Using this terminology, we say
    that $F$ is \emph{consistent with $L$} if, for every oriented edge, the
    associated tile sequence satisfies the following conditions:
    \begin{enumerate}
        \item If the sequence begins with a {\oX}, then the remaining sequence
            starts with a non-negative number of {\ox}'s, followed by a {\oo}.
        \item If the sequence begins with a {\Ox}, then the remaining sequence
            starts with a non-negative number of {\ox}'s, followed by a {\xo}, {\xx}, or {\xX}.
        \item If the sequence begins with a {\aox}, then every tile in the
            sequence, aside from the first and last, is a {\ox}.
        \item If the sequence begins with a {\Xx}, then the remaining sequence
            starts with a non-negative number of {\xx}'s, followed by a {\xo}.
        \item If the sequence begins with a {\xX}, then the remaining sequence
            starts with a non-negative number of {\xx}'s, followed by a {\ox}
            or {\Ox}.
        \item If the sequence begins with a {\xmx}, then the remaining sequence
            starts with a non-negative number of {\xx}'s, followed by a {\oo}.
        \item If the sequence begins with a {\axx}, then every tile in the
            sequence, aside from the first and last, is a {\xx}.
    \end{enumerate}
\end{defn}
Note that $F$ does not have to be lex-minimal in the above definition, and
consistency with $L$ alone is not enough to guarantee lex-minimality, since the
conditions of Theorem \ref{T:minimal} might not be satisfied in the middle of a
raft.
\begin{lem}\label{L:legallabelling}
    Let $L$ be a legal labeling, and suppose that $F$ is a filling of the $w$-ocean
    consistent with $L$. Then $F$ is lex-minimal if and only if for every oriented edge
    $\varepsilon$, the associated tile sequence $t_0 \cdots t_{n+1}$ satisfies two conditions:
    \begin{enumerate}[(a)]
        \item If $t_i t_{i+1} t_{i+2} \cdots t_{j}$ is a subsequence with all
            top nodes selected, such that $1 \leq i,j \leq n$ and $t_{i-1}$ and $t_{j+1}$
            do not have top nodes selected, then there is some $i \leq k \leq j$ such
            that the bottom node of $t_k$ is not selected.
        \item If $t_i t_{i+1} t_{i+2} \cdots t_{j}$ is a subsequence with all
            bottom nodes selected, such that $1 \leq i,j \leq n$ and $t_{i-1}$ and $t_{j+1}$
            do not have bottom nodes selected, then there is some $i \leq k \leq j$ such
            that the top node of $t_k$ is not selected, and some $i-1 \leq k' \leq j+1$
            such that the top node of $t_{k'}$ is selected.
    \end{enumerate}
\end{lem}
Note that this criterion only looks at subsequences not containing the endplanks.
In other words, if $F$ is consistent with a legal labeling, then
lex-minimality reduces to checking the conditions of Theorem \ref{T:minimalsym}
on the interior of tile sequences.
\begin{proof}
    Follows from Theorem \ref{T:minimal}.  The proof is given in the informal
    description of Definitions \ref{defn:legallabelling} and
    \ref{defn:consistentwlabelling}.
\end{proof}
Finally, we define a generalization of the $a$-sequences from
Section~\ref{sub:a.sequences}.
\begin{defn}\label{defn:a.sequence}
    Let $u,v \in \mc{L} \cup \mc{L}'$, the set of half-edge labellings. We say that
    $(u,v)$ is a \emph{legal pair} if neither $u$ or $v$ belongs to $\mc{L}' =
    \{{\aox},{\axx}\}$, or one of $u$ or $v$ is {\aox} (resp. {\axx}) and the
    other is {\ox} or {\aox} (resp. {\xx} or {\axx}).

    If $u,v$ is a legal pair, and $m \geq 0$, then we define $a(u,v;m)$ to be
    the number of tile sequences $t_0 \cdots t_{m+1}$ such that:
    \begin{enumerate}[(i)]
        \item $t_0 = u$ and $t_{m+1} = v$,
        \item both $t_0 \cdots t_{m+1}$ and $t_{m+1} t_m \cdots t_0$ satisfy conditions (1)-(7)
            of Definition \ref{defn:consistentwlabelling}, and
        \item $t_0 \cdots t_{m+1}$ satisfies conditions (a) and (b) of Lemma
            \ref{L:legallabelling}.
    \end{enumerate}
\end{defn}

\begin{lem}\label{L:a.sequence}
    The sequences $a(u,v;m)$ satisfy the following properties:
    \begin{itemize}
        \item $a(u,v;m) = a(v,u;m)$.
        \item If one of $u$ or $v$ is {\aox} or {\axx} then $a(u,v;m) = 1$ for
            all $m \geq 0$.
        \item If neither $u$ or $v$ is {\aox} or {\axx}, then the family of
            sequences $a(u,v;m)$ for $u,v \in \mathcal{L}$ are determined by
            the entries of Table~\ref{table:generating functions in terms of finite automaton},
            given in terms of the generating functions for walks on the
            automaton.
    \end{itemize}
\end{lem}
\begin{proof}
    The first two properties follow immediately by construction. For the third
    property, while we have added additional boundary conditions, the enumeration
    is otherwise the same as in the symmetric group case. In particular, each
    sequence is determined by a function of the $R_{i,j}$ generating functions
    given in Equation~\eqref{eq:Rij}, as shown in Table~\ref{table:generating functions in terms of finite automaton}.
  Note that along each row of the table, the initial index $i$ in
$R_{i,j}$ is the same.  These values come from our description of the
half-edge labels on harbors and/or the original assumption from type
$A$ that walks must begin in vertices 1, 2, 3, or 7, according to the
label. The set of indices $j$ that appear in each $R_{i,j}$ in any one
column can be determined similarly, either by referring back to the type
$A$ case, or by using the descriptions of the half-edge labels.

The factor $1/(1-t)$ occurs when we can start to fill a raft
moving away from a wharf with any number of tiles of a given type, and
then we follow an allowed walk starting at a given vertex.  We
subtract 1 in cases where there is a constant term arising from  $i=j$, because walks with 0 edges cannot occur in the enumeration
of lex-minimal fillings of rafts since we always include two tiles
representing the apparatus on either end of the raft.  Finally, in the
case of $a\raisebox{.25ex}{\big(}{\Xx}, {\xX}; m \raisebox{.25ex}{\big)}$, the generating function
$R_{6,6}-1/(1-t)$ counts walks starting at vertex 6 and
ending at vertex 6, except for the walk that never leaves vertex 6.  That
exception is an $\axx$\,-path, and it is counted as a different labeling
of the harbor.

The verification of the entries in Table~\ref{table:generating
  functions in terms of finite automaton} is now straightforward, although a bit tedious.
\end{proof}

\begin{table}[htbp]
$$\begin{array}{c||c|c|c|c|c|c|c|c|c}
 \raisebox{0cm}[.5cm][.4cm]{} & \oo & \xo & \ox & \oX & \Ox & \xx & \Xx & \xX & \XX  \\ \hline
  \hline
\oo \raisebox{0cm}[.8cm][.4cm]{}& \frac{R_{2,2}-1}t & \frac{R_{2,3}}t & \frac{R_{2,57}}t & \frac{R_{2,5}}t & \frac{R_{2,7}}t & \frac{R_{2,468}}t & \frac{R_{2,8}}t & \frac{R_{2,6}}t & \frac{R_{2,4}}t \Tstrut\\
  \hline
\xo \raisebox{0cm}[.8cm][.4cm]{}& & \frac{R_{3,3} - 1}{t} & \frac{R_{3,57}}t & \frac{R_{3,5}}{t} & \frac{R_{3,7}}{t} & \frac{R_{3,468}}{t} & \frac{R_{3,8}}t & \frac{R_{3,6}}t & \frac{R_{3,4}}t \\
  \hline
\ox \raisebox{0cm}[.8cm][.4cm]{}& & & \frac{R_{7,57}-1}t & \frac{R_{7,5}}t & \frac{R_{7,7}-1}t & \frac{R_{7,468}}t & \frac{R_{7,8}}t & \frac{R_{7,6}}t & \frac{R_{7,4}}t \\
  \hline
\oX \raisebox{0cm}[.8cm][.4cm]{}& & & & \frac{R_{2,5}}{1-t} & \frac{R_{2,7}}{1-t} & \frac{R_{2,468}}{1-t} & \frac{R_{2,8}}{1-t} & \frac{R_{2,6}}{1-t} & \frac{R_{2,4}}{1-t}\\
  \hline
\Ox \raisebox{0cm}[1.1cm][.6cm]{}& & & & & \frac{R_{5,7}}t & \frac{R_{5,468}}t & \frac{R_{5,8}}t & \frac{R_{5,6}}t &  \frac{R_{5,4}}t \\
  \hline
\xx \raisebox{0cm}[.8cm][.4cm]{}& & & & & & \frac{R_{1,1468}-1}{t} & \frac{R_{1,8}}{t} & \frac{R_{1,6}}{t} & \frac{R_{1,4}}{t}\\
  \hline
\Xx \raisebox{0cm}[.8cm][.4cm]{}& & & & & & & \frac{R_{6,8}}t & \frac{R_{6,6} - \frac 1{1-t}}{t} & \frac{R_{6,4}}{t}\\
  \hline
\xX \raisebox{0cm}[.8cm][.4cm]{}& & & & & & & & \frac{R_{8,6}}t & \frac{R_{8,4}}{t}\\
  \hline
\XX \raisebox{0cm}[.8cm][.4cm]{}& & & & & & & & & \frac{R_{2,4}}{1 - t}
\end{array}$$
\caption{Generating functions of the sequences $a(u,v;n)$ in terms of the finite automaton.}
\label{table:generating functions in terms of finite automaton}
\end{table}

\begin{cor}\label{cor:recurrences}
Each of the generalized $a$-sequences from Table~\ref{table:generating
functions in terms of finite automaton} satisfies one of the following
recurrences:
\begin{description}
 \item[R1] $a_n = 5a_{n-1} - 7a_{n-2} + 4a_{n-3}$ for $n \geq 3$,
 \item[R2] $a_n = 6a_{n-1} - 12a_{n-2} + 11a_{n-3} - 4a_{n-4}$ for $n \geq 4$,
 \item[R3] $a_n = 6a_{n-1} - 13a_{n-2} + 16a_{n-3} - 11a_{n-4} + 4a_{n-5}$ for $n \geq 5$, or
 \item[R4] $a_n = 7a_{n-1} - 19a_{n-2} + 29a_{n-3} - 27a_{n-4} + 15a_{n-5} -4a_{n-6}$ for $n \geq 6$.
\end{description}
The recurrences and initial conditions are shown in Table~\ref{table:recurrences and initial conditions}.
\end{cor}

The initial conditions in Table~\ref{table:recurrences and initial
  conditions} can all be verified from the generating functions or by
considering the lex-minimal presentations in 3 families of cases.

\begin{enumerate}

\item No wharfs: The $w$-ocean for $w=s_1s_n=[2,1,3,4,\ldots, n-1,n]$ in
type $A_n$ for $n>4$ has 4 ropes at $(2,0), (2,1), (n-1,0), (n-1,1)$
attached to the 4 corners of one raft of size $n-4$. Each rope can be
selected independently.

\item One wharf: In type $D_n$, say the unique leaf not connected to the
branch node of the Coxeter graph is labeled $n$, and the branch node
is labeled $s_1$.  Then if $w=s_n$, the $w$-ocean has a wharf on
vertices $\{(1,0),(1,1)\}$, ropes at $(n-1,0)$ and $(n-1,1)$ and each
$\{(i,0),(i,1)\}$ for $1< i <n-1$ is a plank in one of the 3 rafts.

\item Two wharfs: In type $\widetilde{D}_n$, the identity element has two
wharfs corresponding with the two branch nodes and every generator is
a small ascent on the left and the right.
\end{enumerate}

\begin{table}[htbp]
$$\begin{array}{c||c|c|c|c|c|c|c|c|c}
\raisebox{0cm}[.8cm][.4cm]{}  & \oo & \xo & \ox & \oX & \Ox & \xx & \Xx & \xX & \XX  \\ \hline
  \hline
\oo \raisebox{0cm}[.8cm][.4cm]{}& \ri 1{1,2,6}  & \ri 1{1,3,9} & \ri 1{1,3,9} & \ri 1{1,2,4} & \ri 1{0,1,5} & \ri 1{1,4,12} & \ri 1{0,1,4} & \ri 1{0,1,4} & \ri 1{1,2,4} \Tstrut\\
\hline
\xo \raisebox{0cm}[.8cm][.4cm]{}& & \ri 3{1,3,11,37,119} & \ri 3{1,4,12,37,118} & \ri 1{0,1,4} & \ri 3{1,3,8,24,77} & \ri 1{1,4,14} & \ri 3{1,2,5,16,53} & \ri 3{0,1,5,17,54} & \ri 1{0,1,4} \\
\hline
\ox \raisebox{0cm}[.8cm][.4cm]{}& & & \ri 3{1,3,11,37,119} & \ri 1{0,1,4} & \ri 3{1,2,7,24,78} & \ri 1{1,4,14} & \ri 3{0,1,5,17,54} & \ri 3{1,2,5,16,53} & \ri 1{0,1,4} \\
\hline
\oX \raisebox{0cm}[.8cm][.4cm]{}& & & & \ri 2{0,1,3,7} & \ri 2{0,0,1,6} & \ri 2{0,1,5,17} & \ri 1{0,0,1} & \ri 1{0,0,1} & \ri 2{0,1,3,7}\\
\hline
\Ox \raisebox{0cm}[.8cm][.4cm]{}& & & & & \ri 3{0,1,5,17,53} & \ri 2{1,3,9,29} & \ri 3{0,1,4,12,36} & \ri 3{1,2,4,11,35} & \ri 2{0,0,1,6}\\
\hline
\xx \raisebox{0cm}[.8cm][.4cm]{}& & & & & & \ri 1{1,4,16} & \ri 2{0,1,5,19} & \ri 2{0,1,5,19} & \ri 2{0,1,5,17}\\ 
\hline
\Xx \raisebox{0cm}[.8cm][.4cm]{}& & & & & & & \ri 3{0,1,3,8,24} & \ri 4{0,0,1,6,23,77} & \ri 1{0,0,1}\\
\hline
\xX \raisebox{0cm}[.8cm][.4cm]{}& & & & & & & & \ri 3{0,1,3,8,24} & \ri 1{0,0,1}\\
\hline
\XX \raisebox{0cm}[.8cm][.4cm]{}& & & & & & & & & \ri 2{0,1,3,7}
\end{array}$$
\caption{Recurrence relations and initial conditions arising from the finite automaton.}
\label{table:recurrences and initial conditions}
\end{table}

We can now state the main theorem in full generality.

\begin{thm}\label{main.coxeter}
Fix an arbitrary Coxeter group $W$ and an element $w \in W$. The number of parabolic double cosets with minimal element $w$ is
\[
 c_w = 2^{|\floats(w)|} \sum_{\substack{\text{choices $C$ of}\\ \text{tethers, ropes, and} \\ \text{wharfs for $w$}}} \ \sum_{\substack{\text{legal labelings $L$}\\ \text{of the harbor $H_C$}}} \ \prod_{\p R \in \rafts(w)} a(R,C,L),
\]
where $a(R,C,L)$ is determined by the
rational generating functions (or, equivalently, the linear recurrence
relations) given in Table~\ref{table:generating functions in terms of
  finite automaton} (Table~\ref{table:recurrences and initial
  conditions}).
\end{thm}

\begin{proof}
    The proof follows from Definitions \ref{defn:legallabelling}, \ref{defn:consistentwlabelling}
    and \ref{defn:a.sequence}, and Lemmas \ref{L:legallabelling} and \ref{L:a.sequence}.
\end{proof}

We demonstrate this result for examples for Weyl groups and affine Weyl groups in the
next subsections.

\begin{remark}
As in Section~\ref{sub:b-sequences}, we can sum together collections
of $a$-sequences containing the option to select a given rope(s) or
not.  Thus, there exist $b$-sequences for all Coxeter groups as well.
The proof of Theorem~\ref{main.general.case} now follows as a
corollary to Theorem~\ref{main.coxeter} provided we define
$\wharfs(w)$ to be the set of all possible choices of dots on all of
the wharfs of $w$ along with legal labels of all the half-edges
emanating from the wharfs.
\end{remark}

\subsection{Example: a wharf with three branches}

We now study the case of a wharf with three branches. These will have sizes $i$, $j$, and $k$, and each will end in a doubly unfilled tile. This structure arises as the $e$-ocean in Coxeter groups $D_n$, $E_n$, $\widetilde B_n$, and $\widetilde E_n$, where $e$ is the identity. We denote the total number of parabolic double cosets in this case by
$$\branch(i,j,k).$$
There are four options for the wharf. Throughout the following computation, we will refer to Figure~\ref{branch}.

The blank wharf $\omo$ gives only one possible labeling (see Figure~\ref{branch}, drawing 1), and contributes
\begin{equation} \label{branch1}
 a\raisebox{.25ex}{\big(}\oo\,,\oo\,;i\raisebox{.25ex}{\big)}\,a\raisebox{.25ex}{\big(}\oo\,,\oo\,;j\raisebox{.25ex}{\big)}\,a\raisebox{.25ex}{\big(}\oo\,,\oo\,;k\raisebox{.25ex}{\big)}.
\end{equation}
The wharf $\xmo$ also gives only one possible labeling (see Figure~\ref{branch}, drawing 2), and contributes
\begin{equation} \label{branch2}
 a\raisebox{.25ex}{\big(}\oo\,,\xo\,;i\raisebox{.25ex}{\big)}\,a\raisebox{.25ex}{\big(}\oo\,,\xo\,;j\raisebox{.25ex}{\big)}\,a\raisebox{.25ex}{\big(}\oo\,,\xo\,;k\raisebox{.25ex}{\big)}.
\end{equation}
The wharf $\omx$ gives seven possible labelings (see Figure~\ref{branch}, drawings 3--9): every half-edge can be either $\Ox$ or $\oX$\,, but at least one of them has to be labeled $\Ox$\,. The total contribution is therefore
\begin{align} \label{branch3}
 a\raisebox{.25ex}{\big(}\oo\,,\Ox\,;i\raisebox{.25ex}{\big)}\,&a\raisebox{.25ex}{\big(}\oo\,,\Ox\,;j\raisebox{.25ex}{\big)}\,a\raisebox{.25ex}{\big(}\oo\,,\Ox\,;k\raisebox{.25ex}{\big)}
 + a\raisebox{.25ex}{\big(}\oo\,,\oX\,;i\raisebox{.25ex}{\big)}\,a\raisebox{.25ex}{\big(}\oo\,,\Ox\,;j\raisebox{.25ex}{\big)}\,a\raisebox{.25ex}{\big(}\oo\,,\Ox\,;k\raisebox{.25ex}{\big)}  \\
\nonumber
& + a\raisebox{.25ex}{\big(}\oo\,,\Ox\,;i\raisebox{.25ex}{\big)}\,a\raisebox{.25ex}{\big(}\oo\,,\oX\,;j\raisebox{.25ex}{\big)}\,a\raisebox{.25ex}{\big(}\oo\,,\Ox\,;k\raisebox{.25ex}{\big)}
+ a\raisebox{.25ex}{\big(}\oo\,,\Ox\,;i\raisebox{.25ex}{\big)}\,a\raisebox{.25ex}{\big(}\oo\,,\Ox\,;j\raisebox{.25ex}{\big)}\,a\raisebox{.25ex}{\big(}\oo\,,\oX\,;k\raisebox{.25ex}{\big)} \\
\nonumber
& + a\raisebox{.25ex}{\big(}\oo\,,\oX\,;i\raisebox{.25ex}{\big)}\,a\raisebox{.25ex}{\big(}\oo\,,\oX\,;j\raisebox{.25ex}{\big)}\,a\raisebox{.25ex}{\big(}\oo\,,\Ox\,;k\raisebox{.25ex}{\big)}
+ a\raisebox{.25ex}{\big(}\oo\,,\oX\,;i\raisebox{.25ex}{\big)}\,a\raisebox{.25ex}{\big(}\oo\,,\Ox\,\,;j\raisebox{.25ex}{\big)}\,a\raisebox{.25ex}{\big(}\oo\,,\oX\,;k\raisebox{.25ex}{\big)} \\
\nonumber
&+ a\raisebox{.25ex}{\big(}\oo\,,\Ox\,;i\raisebox{.25ex}{\big)}\,a\raisebox{.25ex}{\big(}\oo\,,\oX\,;j\raisebox{.25ex}{\big)}\,a\raisebox{.25ex}{\big(}\oo\,,\oX\,;k\raisebox{.25ex}{\big)}.
\end{align}
Finally, the wharf $\xmx$ gives twelve possible labelings (see Figure~\ref{branch}, drawings 10--21): every half-edge can be either $\xX$\,, $\Xx$\,, or $\XX$\,, at least one of them must be $\xX$\,, and at least one of them must be $\Xx$\,. The total contribution is
\begin{align} \label{branch4}
 a\raisebox{.25ex}{\big(}\oo\,,\xX\,;i\raisebox{.25ex}{\big)}\,&a\raisebox{.25ex}{\big(}\oo\,,\Xx\,;j\raisebox{.25ex}{\big)}\,a\raisebox{.25ex}{\big(}\oo\,,\XX\,;k\raisebox{.25ex}{\big)}
 + a\raisebox{.25ex}{\big(}\oo\,,\xX\,;i\raisebox{.25ex}{\big)}\,a\raisebox{.25ex}{\big(}\oo\,,\XX\,;j\raisebox{.25ex}{\big)}\,a\raisebox{.25ex}{\big(}\oo\,,\Xx\,;k\raisebox{.25ex}{\big)}\\
 \nonumber
& + a\raisebox{.25ex}{\big(}\oo\,,\Xx\,;i\raisebox{.25ex}{\big)}\,a\raisebox{.25ex}{\big(}\oo\,,\xX\,;j\raisebox{.25ex}{\big)}\,a\raisebox{.25ex}{\big(}\oo\,,\XX\,;k\raisebox{.25ex}{\big)}
+ a\raisebox{.25ex}{\big(}\oo\,,\Xx\,;i\raisebox{.25ex}{\big)}\,a\raisebox{.25ex}{\big(}\oo\,,\XX\,;j\raisebox{.25ex}{\big)}\,a\raisebox{.25ex}{\big(}\oo\,,\xX\,;k\raisebox{.25ex}{\big)} \\
\nonumber
& + a\raisebox{.25ex}{\big(}\oo\,,\XX\,;i\raisebox{.25ex}{\big)}\,a\raisebox{.25ex}{\big(}\oo\,,\xX\,;j\raisebox{.25ex}{\big)}\,a\raisebox{.25ex}{\big(}\oo\,,\Xx\,;k\raisebox{.25ex}{\big)}
+ a\raisebox{.25ex}{\big(}\oo\,,\XX\,;i\raisebox{.25ex}{\big)}\,a\raisebox{.25ex}{\big(}\oo\,,\Xx\,;j\raisebox{.25ex}{\big)}\,a\raisebox{.25ex}{\big(}\oo\,,\xX\,;k\raisebox{.25ex}{\big)} \\
\nonumber
&+ a\raisebox{.25ex}{\big(}\oo\,,\xX\,;i\raisebox{.25ex}{\big)}\,a\raisebox{.25ex}{\big(}\oo\,,\xX\,;j\raisebox{.25ex}{\big)}\,a\raisebox{.25ex}{\big(}\oo\,,\Xx\,;k\raisebox{.25ex}{\big)}
+ a\raisebox{.25ex}{\big(}\oo\,,\xX\,;i\raisebox{.25ex}{\big)}\,a\raisebox{.25ex}{\big(}\oo\,,\Xx\,;j\raisebox{.25ex}{\big)}\,a\raisebox{.25ex}{\big(}\oo\,,\xX\,;k\raisebox{.25ex}{\big)} \\
\nonumber
& + a\raisebox{.25ex}{\big(}\oo\,,\Xx\,;i\raisebox{.25ex}{\big)}\,a\raisebox{.25ex}{\big(}\oo\,,\xX\,;j\raisebox{.25ex}{\big)}\,a\raisebox{.25ex}{\big(}\oo\,,\xX\,;k\raisebox{.25ex}{\big)}
+ a\raisebox{.25ex}{\big(}\oo\,,\xX\,;i\raisebox{.25ex}{\big)}\,a\raisebox{.25ex}{\big(}\oo\,,\Xx\,;j\raisebox{.25ex}{\big)}\,a\raisebox{.25ex}{\big(}\oo\,,\Xx\,;k\raisebox{.25ex}{\big)} \\
\nonumber
& + a\raisebox{.25ex}{\big(}\oo\,,\Xx,;i\raisebox{.25ex}{\big)}\,a\raisebox{.25ex}{\big(}\oo\,,\xX\,;j\raisebox{.25ex}{\big)}\,a\raisebox{.25ex}{\big(}\oo\,,\Xx\,;k\raisebox{.25ex}{\big)}
+ a\raisebox{.25ex}{\big(}\oo\,,\Xx\,;i\raisebox{.25ex}{\big)}\,a\raisebox{.25ex}{\big(}\oo\,,\Xx\,;j\raisebox{.25ex}{\big)}\,a\raisebox{.25ex}{\big(}\oo\,,\xX\,;k\raisebox{.25ex}{\big)}.
\end{align}
Therefore, $\branch(i,j,k)$ is the sum of expressions \eqref{branch1}--\eqref{branch4}.

\begin{figure}[htbp]
\begin{center}
 \input{branch.tex} \caption{All 21 legally labeled, decorated,
harbor graphs for the ocean with one wharf and three rafts of sizes
$i$, $j$ and $k$.}\label{branch}
\end{center}
\end{figure}
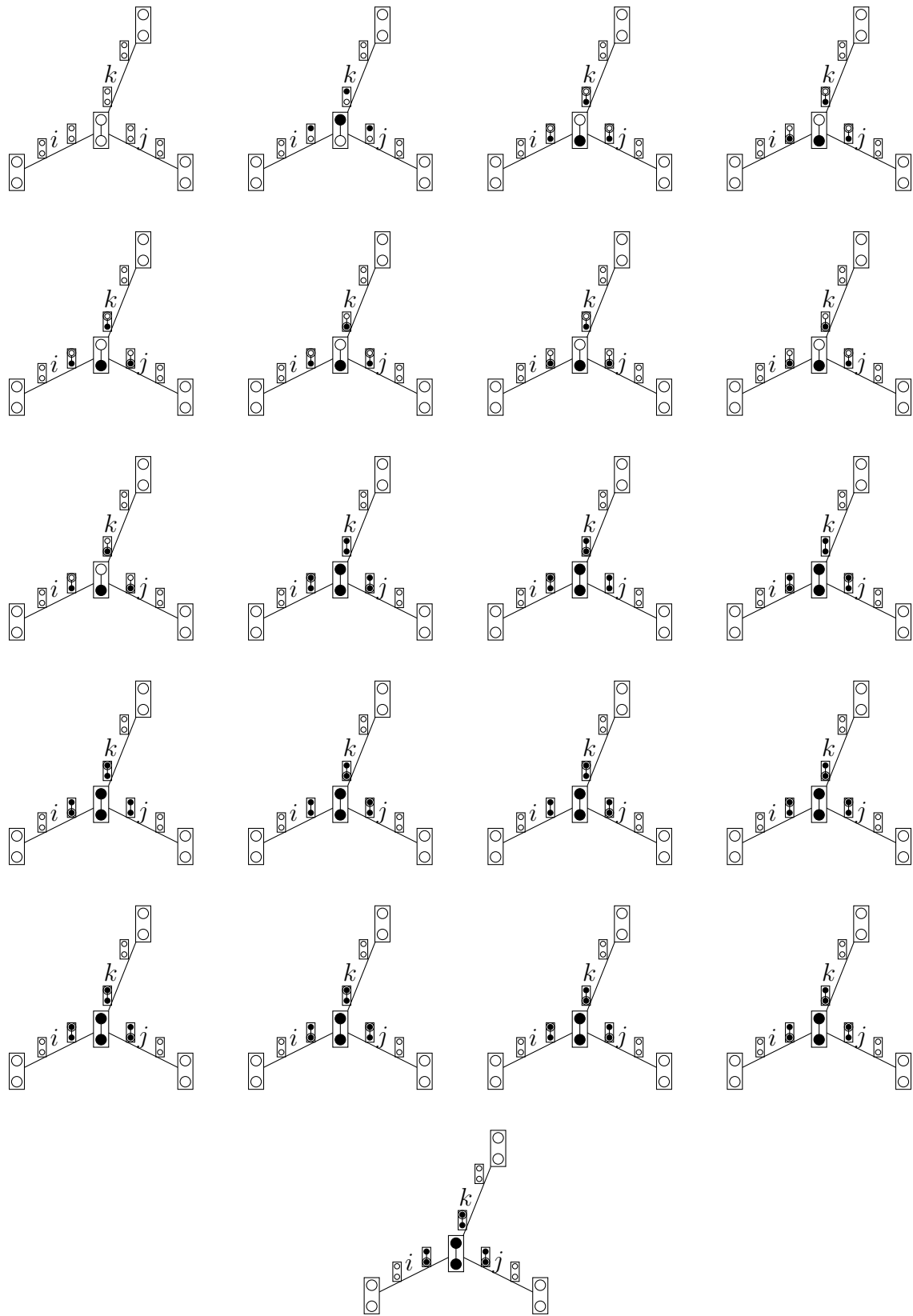

\begin{example}
For $n \ge 3$, consider the identity elements in $D_n$ and $\widetilde B_{n-1}$.  The number of parabolic double cosets whose minimal representative is one of these elements is $\branch(1,1,n-3)$. For $n = 3,\ldots,10$, this gives the sequence
$$20, 72, 234, 746, 2380, 7614, 24394, 78192,$$
and has generating function
$$\frac{t^3 \left(20-28t+14 t^2\right)}{1-5t+7t^2-4t^3}.$$
\end{example}

\begin{example}
Consider the identity element in $E_n$, for $n \geq 4$. For this, we compute $\branch(2,1,n-4)$. For $n = 4,\ldots,10$, this gives the sequence
$$66, 234, 750, 2376, 7566, 24198, 77532,$$
and has generating function
$$\frac{t^4 \left(66-96t+42 t^2\right)}{1-5t+7t^2-4t^3}.$$
\end{example}

\begin{example}
The number of parabolic double cosets whose minimal representative is the identity element in the affine Coxeter group $\widetilde E_n$, for $n = 6,7,8$, is
\begin{align*}
\branch(2,2,2) &= 2378,\\
\branch(3,1,3) &= 7514, \text{ and}\\
\branch(2,1,5) &= 24198.
\end{align*}
\end{example}

\subsection{Example: a circular raft}

We can enumerate parabolic double cosets in the affine group
$\widetilde A_n$, for $n \geq 0$. The Coxeter graph in this case is a
cycle with $n+1$ elements.  To study the $e$-ocean, we introduce a wharf on any one of the planks. Again,
we have four options for filling the wharf. The unfilled wharf $\omo$ gives
only one possible labeling (see Figure~\ref{circular}, drawing 1), the
wharf $\xmo$ also gives only one possible labeling (see
Figure~\ref{circular}, drawing 2), the wharf $\omx$ has three possible
labelings (see Figure~\ref{circular}, drawings 3--5), and the wharf
$\xmx$ has two possible labelings (see Figure~\ref{circular}, drawings
6--7). Thus the number of parabolic double cosets is
\begin{equation}
 a\raisebox{.25ex}{\big(}\oo\,,\oo\,;n\raisebox{.25ex}{\big)} +a\raisebox{.25ex}{\big(}\xo\,,\xo\,;n\raisebox{.25ex}{\big)} +a\raisebox{.25ex}{\big(}\Ox\,,\Ox\,;n\raisebox{.25ex}{\big)} +a\raisebox{.25ex}{\big(}\Ox\,,\oX\,;n\raisebox{.25ex}{\big)}+a\raisebox{.25ex}{\big(}\oX\,,\Ox\,;n\raisebox{.25ex}{\big)} +a\raisebox{.25ex}{\big(}\xX\,,\Xx\,;n\raisebox{.25ex}{\big)} +a\raisebox{.25ex}{\big(}\Xx\,,\xX\,;n\raisebox{.25ex}{\big)}.
\end{equation}
For $n = 0,\ldots,10$, this gives $2, 6, 26, 98, 332, 1080, 3474, 11146, 35738, 114566, 367248$. The generating function is
$$\frac{2-8t+22t^2-28t^3+20t^4-4 t^5}{(1-t) \left(1-t+t^2\right) \left(1-5t+7t^2-4t^3\right)}.$$

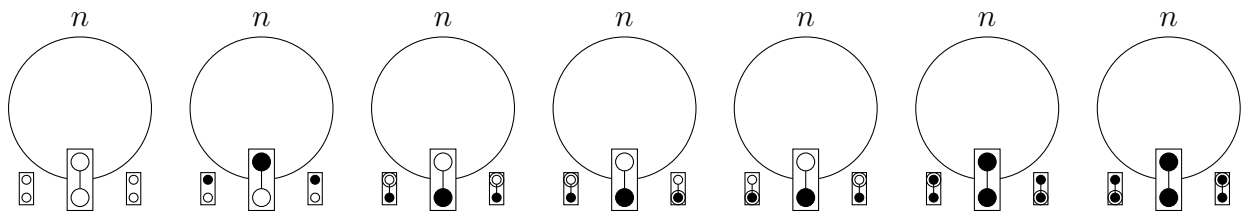
\begin{figure}[htbp]
\begin{center}
 \input{circular.tex}
 \caption{A circular raft, and all seven legal labelings.}\label{circular}
\end{center}
\end{figure}

\subsection{Example: two wharfs with three branches}

Our final example is the case of the identity for a Coxeter graph with two branch points connected by a path (of size $k$), and with two more branches (of sizes $i_1$, $j_1$, $i_2$, and $j_2$) coming out of each branch point. An example is the affine group $\widetilde D_n$, for $n \geq 4$ (with $k = n-4$ and $i_1=j_1=i_2=j_2=1$).

There are now too many labelings to state in a concise manner. For each wharf, there are four choices, so we have 16 choices total. If both wharfs are either $\omo$ or $\xmo$\,, then there is only one labeling. For example, for left wharf $\omo$ and right wharf $\xmo$\,, the contribution is
$$a\raisebox{.25ex}{\big(}\oo\,,,\oo\,;i_1\raisebox{.25ex}{\big)}\,a\raisebox{.25ex}{\big(}\oo\,,\oo\,;j_1\raisebox{.25ex}{\big)}\,a\raisebox{.25ex}{\big(}\oo\,,\xo\,;k\raisebox{.25ex}{\big)}\,a\raisebox{.25ex}{\big(}\xo\,,\oo\,;i_2\raisebox{.25ex}{\big)}\,a\raisebox{.25ex}{\big(}\xo\,,\oo\,;j_2\raisebox{.25ex}{\big)}$$
(see Figure~\ref{twobranch}, drawing 1). But, for example, for left wharf $\omx$ and right wharf $\omx$\,, there are $49 + 15 = 64$ possible labelings (we can either label all three half-edges coming out of each wharf in one of $7$ possible ways -- the choices are $\oX$ or $\Ox$ for each, and we cannot select $\oX$ for all three -- or we can label the edge between the wharfs by {\ox}\,, and then we can label the remaining four half-edges either by $\oX$ or $\Ox$\,, but we cannot label them all $\oX$\,). In Figure~\ref{twobranch}, drawings 2,
we see a labeling that contribute $a\raisebox{.25ex}{\big(}\oo\,,\Ox\,;i_1\raisebox{.25ex}{\big)}\,a\raisebox{.25ex}{\big(}\oo\,,\oX\,;j_1\raisebox{.25ex}{\big)}\,a\raisebox{.25ex}{\big(}\Ox\,,\Ox\,;k\raisebox{.25ex}{\big)}\,a\raisebox{.25ex}{\big(}\oX\,,\oo\,;i_2\raisebox{.25ex}{\big)}\,a\raisebox{.25ex}{\big(}\oX\,,\oo\,;j_2\raisebox{.25ex}{\big)}$ while drawing 3 in Figure~\ref{twobranch} contributes $a\raisebox{.25ex}{\big(}\oo\,,\Ox\,;i_1\raisebox{.25ex}{\big)}\,a\raisebox{.25ex}{\big(}\oo\,,\oX\,;j_1\raisebox{.25ex}{\big)}\,a\raisebox{.25ex}{\big(}\oX\,,\oo\,;i_2\raisebox{.25ex}{\big)}\,a\raisebox{.25ex}{\big(}\oX\,,\oo\,;j_2\raisebox{.25ex}{\big)}$. In total, there are
$$(1+1+7+12)+(1+1+7+12)+(7+7+64+84)+(12+12+84+194)=506$$
possible labelings, each contributing a product of four or five terms to the sum.

\begin{figure}[htbp]
\begin{center}
 \input{twobranch.tex}
 \caption{A harbor graph with two wharfs and some of its legal labelings.}\label{twobranch}
\end{center}
\end{figure}
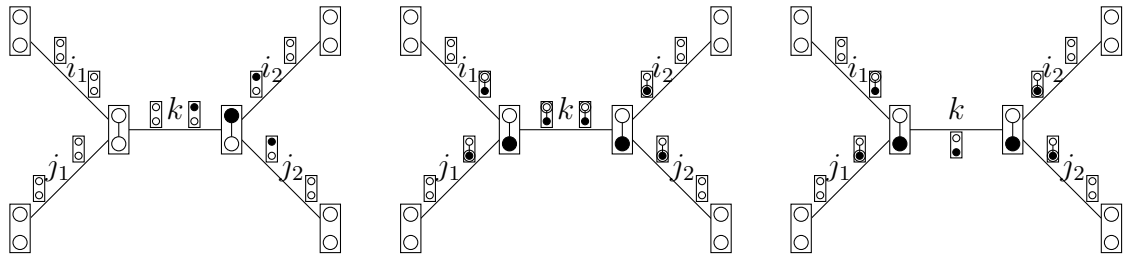

\begin{example}
The number of parabolic double cosets whose minimal representative is the identity in $\widetilde D_n$, with $n = 5,\ldots,14$,
$$814, 2558, 8176, 26230, 84150, 269844, 865090, 2773142, 8889456, 28495646, 91344606,$$
and the generating function is (the surprisingly simple)
$$\frac{t^4 (814 - 1512 t + 1084 t^2)}{1 - 5 t + 7 t^2 - 4 t^3}.$$
\end{example}

\section{Parabolic double cosets with restricted simple reflections}\label{sec:restricted}

We finish with some remarks about our enumerative formulas. The formula for
$S_n$ and the formula for general Coxeter groups both center around the number
of parabolic double cosets over rafts, subject to different boundary conditions.
The boundary conditions do not change the underlying recurrence, only its
initial conditions. The number of parabolic double cosets for a raft of size
$n$, ignoring boundary conditions, is the same as the number of parabolic
double cosets in $S_{n+1}$ whose minimal element is the identity.  Boundary
conditions amount to enumerating parabolic double cosets with presentations
$W_I e W_J$ where simple reflections $s_1$ and $s_n$ may be forbidden from
belonging to $I$ or $J$.

This suggests looking at the problem of enumerating parabolic double cosets
$W_I w W_J$ where certain simple reflections are not allowed to belong to $I$
or $J$. It turns out that the characterization of lex-minimal elements in
Theorem~\ref{T:minimal} also applies to this more general problem. By working
in this framework, we get an intriguing structural explanation for our
enumerative formulas: the set of parabolic double cosets with fixed minimal
element $w$ is in bijection with a restricted set of parabolic double cosets of
the identity in a larger Coxeter group. This suggests that our enumerative
formulas are somewhat natural, despite their apparent complexity.

\begin{defn}
    Fix subsets $X_L, X_R \subseteq S$. A presentation $C = W_I w
    W_J$ of a parabolic double coset \emph{avoids} $X_L$ and $X_R$ if $I \cap
    X_L = \emptyset$ and $J \cap X_R = \emptyset$. A parabolic
    double coset $C$ \emph{avoids} $X_L$ and $X_R$ if it has a presentation
    that avoids $X_L$ and $X_R$.
\end{defn}
In other words, a parabolic double coset $C$ avoids $X_L$ and $X_R$ if $C$ has
a presentation which does not use any elements of $X_L$ on the left, nor any
elements of $X_R$ on the right.

The natural question to ask is then: given an element $w \in W$, and sets $X_L,
X_R \subseteq S$, how many parabolic double cosets with minimal element $w$
avoid $X_L$ and $X_R$? (To make the question interesting, we can assume that
$X_L$ and $X_R$ are subsets of the left and right ascent set of $w$,
respectively.)

Even if a parabolic double coset $C$ avoids $X_L$ and $X_R$, this does not mean
that every presentation of $C$ avoids $X_L$ and $X_R$, nor even that the
lex-minimal presentation avoids these sets. For instance, the parabolic double
coset $C = W$ has two minimal presentations, $C = W_S  e
W_{\emptyset} = W_{\emptyset} e W_S$, with the latter being
lex-minimal. If we set $X_L = \emptyset$ and $X_R = \{s\}$ for any
$s \in S$, then $W_{\emptyset} e W_S$ does not avoid $X_L$ and
$X_R$. On the other hand, the former presentation, $W_S e W_{\emptyset}$, is now
lex-minimal among all presentations that avoid $X_L$ and $X_R$.

Fortunately, if $C$ avoids $X_L$ and $X_R$, then it clearly has some minimal
presentation (in the sense of Definition~\ref{D:minimal}) that avoids $X_L$
and $X_R$, and we can characterize the unique lex-minimal presentation of this
form.

\begin{prop}\label{P:res_minimal}
    Let $X_L$ and $X_R$ be subsets of the left and right ascent sets of $w \in
    W$, respectively. Let $C$ be a parabolic double coset with minimal element $w$ and a
    presentation $C = W_I w W_J$ that avoids $X_L$ and $X_R$. This
    presentation is lex-minimal among all $X_L$- and $X_R$-avoiding
    presentations for $C$ if and only if
    \begin{enumerate}[(a)]
        \item no connected component of $J$ is contained within $(w^{-1} I w) \cap S$, and
        \item if a connected component $I_0$ of $I$ is contained in $w S w^{-1}$, then either
            some element of $w^{-1} I_0 w$ is contained in $X_R$, or
            some element of $w^{-1} I_0 w$ is adjacent to but not contained in $J$.
    \end{enumerate}
    Furthermore, every parabolic double coset avoiding $X_L$ and $X_R$ has a unique
    presentation that is lex-minimal among the coset's $X_L$- and $X_R$-avoiding presentations.
\end{prop}

\begin{proof}
    By Proposition~\ref{P:minimal}, the minimal presentations of $C$ differ only by
    switching the sides of certain connected components. If $I_0$ is a connected
    component of $I$ with $I \cap (w X_R w^{-1}) \neq \emptyset$, then $I_0$ must appear
    on the left in an $X_L$- and $X_R$-avoiding presentation.
\end{proof}

It is much easier to work with the criteria in Theorem~\ref{T:minimal} and
Proposition~\ref{P:res_minimal} if $w$ is the identity. By allowing restricted
simple reflections, we can expand the Coxeter graph to reduce to this case.
To explain how this works, suppose we are given some element $w \in W$, and let
$G$ be the Coxeter graph of $W$. Now proceed as follows.

\begin{enumerate}
    \item Make a new Coxeter graph $G_L \sqcup G_R$, where $G_L$ and $G_R$ are
        each isomorphic to $G$, and $\sqcup$ refers to the disjoint
        union of graphs. The new Coxeter graph has vertex set $S_L \sqcup
        S_R$, where $S_L$ (respectively, $S_R$) is the vertex set of $G_L$ (respectively, $G_R$).
        Each set $S_L$ and $S_R$ is canonically identified with $S$ via bijections
        $\phi^L : S \longrightarrow S_L$ and $\phi^R : S \longrightarrow S_R$.
    \item Delete the vertices in $S_L$ (respectively, $S_R$) corresponding to
        left (respectively, right) descents of $w$. The functions $\phi^L$ and $\phi^R$ are
        now defined only on the left and right ascent sets $\asc_L(w)$ and
        $\asc_R(w)$ of $w$, respectively.
    \item If $s$ is a left ascent of $w$ such that $w^{-1} s w$ is a right ascent
        of $w$ (which happens if and only if $w^{-1} s w \in S$), then identify
        the vertices $\phi^L(s)$ and $\phi^R(w^{-1} s w)$. Call the resulting graph
        $\overline{G}$. The induced functions $\overline{\phi^L}$ and $\overline{\phi^R}$
        are still injective, but their images are no longer necessarily disjoint.
    \item Given sets $X_L \subseteq \asc_L(w)$ and $X_R \subseteq \asc_R(w)$, set
        \begin{align*}
            \overline{X}_L & := \overline{\phi^L}(X_L) \cup \left( \overline{\phi^R}(\asc_R(w))
                \setminus \overline{\phi^L}(\asc_L(w)) \right) \text{ and } \\
            \overline{X}_R & := \overline{\phi^R}(X_R) \cup \left( \overline{\phi^L}(\asc_L(w))
                \setminus \overline{\phi^R}(\asc_R(w)) \right).
        \end{align*}
        In other words, any right ascent of $w$ that is not conjugate to a left ascent
        is not allowed to act on the left, and vice versa.
\end{enumerate}

\begin{prop}\label{P:dynkin}
    Given an element $w \in W$ and sets $X_L \subseteq \asc_L(w)$ and $X_R \subseteq \asc_R(w)$, define
    $\overline{G}$, $\overline{\phi^L}$, $\overline{\phi^R}$, $\overline{X}_L$, and
    $\overline{X}_R$ as above. Let $\overline{W}$ be the Coxeter group with Coxeter graph
    $\overline{G}$. If $I \subseteq \asc_L(w)$ and $J \subseteq \asc_R(w)$, then $C = W_I w
    W_J$ is lex-minimal among $X_L$- and $X_R$-avoiding presentations if and only if
    $(\overline{W}_{\overline{\phi^L}(I)}) e (\overline{W}_{\overline{\phi^R}(J)})$
    is lex-minimal among $\overline{X}_L$- and $\overline{X}_R$-avoiding presentations.
\end{prop}

\begin{proof}
    The image of $\overline{\phi^L}(\asc_L(w))$ avoids $\overline{\phi^R}(\asc_R(w))
    \setminus \overline{\phi^L}(\asc_L(w))$, so $I$ avoids $X_L$ if and only if
    $\overline{\phi^L}(I)$ avoids $\overline{X}_L$. Similarly $J$ avoids $X_R$
    if and only if $\overline{\phi^R}(I)$ avoids $\overline{X}_R$. The remainder
    of the proposition follows immediately from Proposition~\ref{P:res_minimal}.
\end{proof}

\begin{cor}\label{C:dynkin}
    There is a bijection between $X_L$- and $X_R$-avoiding parabolic double cosets in
    $W$ with minimal element $w$, and $\overline{X}_L$-
    and $\overline{X}_R$-avoiding parabolic double cosets in $\overline{W}$ with minimal
    element $e$.
\end{cor}

\begin{proof}
    $\overline{\phi^L}$ induces a bijection between subsets of $\asc_L(w)$ and subsets of the simple
    reflections of $\overline{W}$ that avoid
    $$\left(\overline{\phi^R}(\asc_R(w)) \setminus \overline{\phi^L}(\asc_L(w))\right).$$
    A similar statement can be made for $\overline{\phi^R}$. Thus
    the corollary follows from Propositions~\ref{P:res_minimal} and
   ~\ref{P:dynkin}.
\end{proof}
\begin{example}
    Consider the permutation $w = 13542 \in
    S_5$, which has reduced
    expression $s_2 s_3 s_4 s_3$. The left and right descent sets of $w$ are
    $\{s_2, s_4\}$ and $\{s_3,s_4\}$ respectively, and there are no simple
    reflections $s$ such that $w s w^{-1}$ is also simple. Thus the Coxeter graph
    $\overline{G}$ consists of two copies of the Coxeter graph of $S_5$, with the descent sets
    deleted from each respective copy.

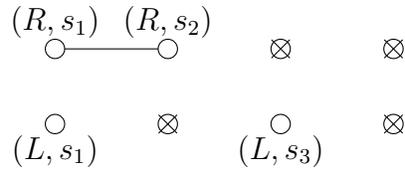
\[
\begin{tikzpicture}[xscale=1.5]
 \draw (1,1) -- (2,1);
 \draw (1,1) node[above] {$(R,s_1)$};
 \draw (2,1) node[above] {$(R,s_2)$};
 \draw (1,0) node[below] {$(L,s_1)$};
 \draw (3,0) node[below] {$(L,s_3)$};
 \foreach \x in {1,...,4}
     {\draw  (\x,1) node[fill=white,draw=black,circle,inner sep=.5ex] {};
     \draw  (\x,0) node[fill=white,draw=black,circle,inner sep=.5ex] {};}
 \draw (2,0) node[cross=.75ex]{};
 \draw (4,0) node[cross=.75ex]{};
 \draw (3,1) node[cross=.75ex]{};
 \draw (4,1) node[cross=.75ex]{};
\end{tikzpicture}
\]

    In this diagram, the vertices from $\overline{\phi^L}(\asc_L(w))$ are labeled
    by $(L,s_i)$ and the vertices from $\overline{\phi^R}(\asc_R(w))$ are labeled
    by $(R,s_i)$. After deleting descents, the nodes that remain are isomorphic to the Coxeter graph $\overline{G}$ of $\overline{W}\cong S_3 \times S_2 \times S_2$.

    If we start with $X_L = X_R = \emptyset$, then $\overline{X}_L
    = \{(R,s_1), (R,s_2)\}$, $\overline{X}_R = \{(L,s_1), (L,s_3)\}$.
    Thus the total number of parabolic double cosets with minimal element $w$ is equal to the
    total number of parabolic double cosets $\overline{W}_I  e  \overline{W}_J$ where
    $I \subseteq \{(L,s_1), (L,s_3)\}$ and $J \subseteq \{(R,s_1), (R,s_2)\}$.
    Every such choice of $I$ and $J$ gives a distinct parabolic double coset, so there are
    $16$ such cosets.
\end{example}

\begin{example}
    Consider another permutation $w = 13425 \in
    S_5$, this one with
    reduced expression $w = s_2 s_3$. The left and right descent sets of $w$ are
    $\{s_2\}$ and $\{s_3\}$ respectively. Among the left ascents $w^{-1} s_3 w = s_2$,
    so the Coxeter graph for $\overline{G}$ consists of two copies of the Coxeter
    graph of $S_5$, with $s_2$ deleted from the left copy, $s_3$ deleted from
    the right copy, and $s_3$ from the left copy identified with $s_2$ in the
    right copy:
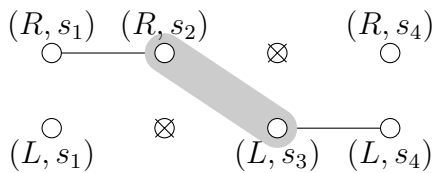
\[
\begin{tikzpicture}[xscale=1.5]
 \draw (1,1) -- (2,1)--(3,0)--(4,0);
 \draw[color=white!80!black,cap=round,line width=15] (2,1)--(3,0);
 \draw (1,1) node[above] {$(R,s_1)$};
 \draw (2,1) node[above] {$(R,s_2)$};
 \draw (4,1) node[above] {$(R,s_4)$};
 \draw (1,0) node[below] {$(L,s_1)$};
 \draw (3,0) node[below] {$(L,s_3)$};
 \draw (4,0) node[below] {$(L,s_4)$};
 \foreach \x in {1,...,4}
     {\draw  (\x,1) node[fill=white,draw=black,circle,inner sep=.5ex] {};
     \draw  (\x,0) node[fill=white,draw=black,circle,inner sep=.5ex] {};}
 \draw (2,0) node[cross=.75ex]{};
 \draw (3,1) node[cross=.75ex]{};
 \end{tikzpicture}
\]

    In this case, then, $\overline{W} \cong S_4 \times S_2 \times S_2$. If we start with $X_L = X_R = \emptyset$, then we must avoid $\overline{X}_L = \{(R,s_1), (R,s_4)\}$
    and $\overline{X}_R = \{(L,s_1), (L,s_4)\}$ in the new group.
\end{example}

This correspondence also preserves the Bruhat order on each individual coset.

\begin{prop}\label{P:bruhat}
    Let $C$ be a parabolic double coset in $W$ with minimal element $w$, and let
    $\overline{C}$ be the corresponding parabolic double coset in $\overline{W}$. Then
    $C$ and $\overline{C}$ are isomorphic as posets in Bruhat order.
\end{prop}
\begin{proof}
    Continuing with the notation above, the parabolic subgroup $W_I$ of $W$ is
    isomorphic to the parabolic subgroup $\overline{W}_{\overline{\phi^L}(I)}$
    of $\overline{W}$, and similarly with $W_J$. If $C = W_I w W_J$ for
    $w \in {}^I W^J$, then $\overline{C} = (W_{\overline{\phi^L}(I)}) e
    (W_{\overline{\phi^R(I)}})$. (Note that this does not depend on whether $C = W_I w W_J$
    is a lex-minimal presentation, but we can choose a lex-minimal presentation
    if we wish to do so.) By the construction of $\overline{G}$ (specifically, the vertex
    identification), there is a well-defined bijection $C \longrightarrow \overline{C}$ sending
    $u w v$ to $\overline{u} w \overline{v}$, where $\overline{u}$ is the
    element of $\overline{W}_{\overline{\phi^L}(I)}$ corresponding to $u$, and
    similarly for $v$ and $\overline{v}$. By Proposition~\ref{P:standard1}(b), this bijection is an order isomorphism with respect to Bruhat order.
\end{proof}

\begin{question}
    If $C$ is finite, then $\overline{C}$ is the Bruhat interval between the
    identity and the maximal element of $\overline{C}$. Thus $\overline{C}$
    can be considered as the $1$-skeleton of a Schubert variety. Is there a
    geometric version of Proposition~\ref{P:bruhat}?
\end{question}

\section*{Acknowledgments}
Many thanks to Edward Richmond, Milen Yakimov, Tewodros Amdeberhan,
Victor Reiner, Richard Stanley, and Joshua Swanson for helpful discussions.

\newpage
\section{Appendix}
The $b$-sequences for the symmetric groups were defined in
\eqref{eq:b.seq} on Page~\pageref{eq:b.seq}.  Each such sequence is
denoted by a superscript 4-tuple $(I_1,I_2,I_3,I_4) \in
\{\{0\},\{1\},\{0,1\}\}$.  The initial values for all 81 $b$-sequences
are given below starting at $m=0$.  Note, there are only 27 distinct
sequences due to the symmetries among the $a$-sequences.

For example, the 4-tuple $(\{0,1\}, \{1\}, \{0\}, \{1\})$, abbreviated
$(01,1,0,1)$, corresponds with the sequence $b_m^{(I_1, I_2,I_3,I_4)}$
with $I_1 =\{0,1\}$, $I_2 =\{1\}$, $I_3 =\{0\}$, $I_4 =\{1\}$.  This
sequence expands in terms of the $a$-sequences as $b_m^{(01,1,0,1)}
=a_m(\xx,\xo)+ a_m(\xo,\xo)$.  This sequence has initial values
$2,7,25,83,267,854,2734$ starting at $m=0$ so for example
$b_1^{(01,1,0,1)}=a_1(\xx,\xo)+ a_1(\xo,\xo)=4+3=7$ using
Table~\ref{table:4x4.table solved} and the symmetry property of the
$a$-sequences.

\begin{small}
$$
  \begin{array}{ll}
    \text{Initial Values} & \text{4-tuples}\\ \hline
1,2,6,20,66,214,688:& (0,0,0,0)\\
1,3,9,28,89,285,914:& (1,0,0,0)(0,1,0,0)(0,0,1,0)(0,0,0,1)\\
1,3,11,37,119,380,1216:& (1,0,1,0)(0,1,0,1)\\
1,4,12,36,112,356,1140:& (1,1,0,0)(0,0,1,1)\\
1,4,12,37,118,379,1216:& (0,1,1,0)(1,0,0,1)\\
1,4,14,46,148,474,1518:& (1,1,1,0)(1,1,0,1)(1,0,1,1)(0,1,1,1)\\
1,4,16,56,184,592,1896:& (1,1,1,1)\\
2,5,15,48,155,499,1602:& (01,0,0,0)(0,01,0,0)(0,0,01,0)(0,0,0,01)\\
2,6,20,65,208,665,2130:& (01,0,1,0)(1,0,01,0)(0,01,0,1)(0,1,0,01)\\
2,7,21,64,201,641,2054:& (01,1,0,0)(1,01,0,0)(0,0,01,1)(0,0,1,01)\\
2,7,21,65,207,664,2130:& (0,01,1,0)(0,1,01,0)(01,0,0,1)(1,0,0,01)\\
2,7,25,83,267,854,2734:& (1,01,1,0)(01,1,0,1)(0,1,01,1)(1,0,1,01)\\
2,8,26,82,260,830,2658:& (1,1,01,0)(01,0,1,1)(0,01,1,1)(1,1,0,01)\\
2,8,26,83,266,853,2734:& (01,1,1,0)(1,01,0,1)(1,0,01,1)(0,1,1,01)\\
2,8,30,102,332,1066,3414:& (01,1,1,1)(1,01,1,1)(1,1,01,1)(1,1,1,01)\\
4,11,35,113,363,1164,3732:& (01,0,01,0)(0,01,0,01)\\
4,12,36,112,356,1140,3656:& (01,01,0,0)(0,0,01,01)\\
4,12,36,113,362,1163,3732:& (0,01,01,0)(01,0,0,01)\\
4,14,46,147,468,1495,4788:& (1,01,01,0)(0,01,01,1)(01,1,0,01)(01,0,1,01)\\
4,14,46,148,474,1518,4864:& (01,01,1,0)(01,01,0,1)(1,0,01,01)(0,1,01,01)\\
4,15,47,147,467,1494,4788:& (01,1,01,0)(01,0,01,1)(1,01,0,01)(0,01,1,01)\\
4,15,55,185,599,1920,6148:& (01,1,01,1)(1,01,1,01)\\
4,16,56,184,592,1896,6072:& (01,01,1,1)(1,1,01,01)\\
4,16,56,185,598,1919,6148:& (1,01,01,1)(01,1,1,01)\\
8,26,82,260,830,2658,8520:& (01,01,01,0)(01,01,0,01)(01,0,01,01)(0,01,01,01)\\
8,30,102,332,1066,3414,10936:& (01,01,01,1)(01,01,1,01)(01,1,01,01)(1,01,01,01)\\
16,56,184,592,1896,6072,19456:& (01,01,01,01)
\end{array}
$$
\end{small}

\end{document}

%% file: branch.tex
\begin{tikzpicture}[scale=.7]
\draw (0,0) coordinate (a);
\draw (-2,-1) coordinate (b);
\draw (2,-1) coordinate (c);
\draw (1,2.5) coordinate (d);
\draw (b) -- (a) -- (c);
\draw (a) -- (d);
\foreach \x in {(a),(b),(c),(d)} {
\fill[white] \x++(-.18,.43) rectangle ++(.36,-.86);
\draw \x ++(0,.25) circle (.125);\draw \x ++(0,-.25) circle (.125);
\draw \x++(-.18,.43) rectangle ++(.36,-.86);}
\draw (a)++(0,-.125) -- ++(0,.25);
\draw ($(a)!0.5!(d)$)++(.1,.1) node[left] {$k$};
\draw ($(a)!0.5!(b)$)++(.2,.3) node[left] {$i$};
\draw ($(a)!0.5!(c)$)++(.35,.3) node[left] {$j$};
\draw ($(a)!0.7!(b)$) coordinate (aab);
\draw ($(a)!0.35!(b)$) coordinate (abb);
\draw ($(a)!0.7!(c)$) coordinate (aac);
\draw ($(a)!0.35!(c)$) coordinate (acc);
\draw ($(a)!0.65!(d)$)++(-.1,0) coordinate (aad);
\draw ($(a)!0.25!(d)$)++(-.1,-.1) coordinate (add);
\foreach \x in {(aab),(abb),(aac),(acc),(aad),(add)} {\draw \x ++(0,.15) circle (.0625); \draw \x ++(0,.4) circle (.0625);
\draw \x++(-.1,.04) rectangle ++(.2,.46);}
\end{tikzpicture}
\hspace{.25in}
\begin{tikzpicture}[scale=.7]
\draw (0,0) coordinate (a);
\draw (-2,-1) coordinate (b);
\draw (2,-1) coordinate (c);
\draw (1,2.5) coordinate (d);
\draw (b) -- (a) -- (c);
\draw (a) -- (d);
\foreach \x in {(a),(b),(c),(d)} {
\fill[white] \x++(-.18,.43) rectangle ++(.36,-.86);
\draw \x ++(0,.25) circle (.125);\draw \x ++(0,-.25) circle (.125);
\draw \x++(-.18,.43) rectangle ++(.36,-.86);}
\draw (a)++(0,-.125) -- ++(0,.25);
\draw ($(a)!0.5!(d)$)++(.1,.1) node[left] {$k$};
\draw ($(a)!0.5!(b)$)++(.2,.3) node[left] {$i$};
\draw ($(a)!0.5!(c)$)++(.35,.3) node[left] {$j$};
\draw ($(a)!0.7!(b)$) coordinate (abb);
\draw ($(a)!0.35!(b)$) coordinate (aab);
\draw ($(a)!0.7!(c)$) coordinate (acc);
\draw ($(a)!0.35!(c)$) coordinate (aac);
\draw ($(a)!0.65!(d)$)++(-.1,0) coordinate (add);
\draw ($(a)!0.25!(d)$)++(-.1,-.1) coordinate (aad);
\foreach \x in {(aab),(abb),(aac),(acc),(aad),(add)} {\draw \x ++(0,.15) circle (.0625); \draw \x ++(0,.4) circle (.0625);
\draw \x++(-.1,.04) rectangle ++(.2,.46);}
\foreach \x in {(a)} {\fill[black] \x ++(0,.25) circle (.125);} 
\foreach \x in {(aab),(aac),(aad)}{\fill[black] \x++(0,.4) circle (.0625);} 
\end{tikzpicture}
\hspace{.25in}
\begin{tikzpicture}[scale=.7]
\draw (0,0) coordinate (a);
\draw (-2,-1) coordinate (b);
\draw (2,-1) coordinate (c);
\draw (1,2.5) coordinate (d);
\draw (b) -- (a) -- (c);
\draw (a) -- (d);
\foreach \x in {(a),(b),(c),(d)} {
\fill[white] \x++(-.18,.43) rectangle ++(.36,-.86);
\draw \x ++(0,.25) circle (.125);\draw \x ++(0,-.25) circle (.125);
\draw \x++(-.18,.43) rectangle ++(.36,-.86);}
\draw (a)++(0,-.125) -- ++(0,.25);
\draw ($(a)!0.5!(d)$)++(.1,.1) node[left] {$k$};
\draw ($(a)!0.5!(b)$)++(.2,.3) node[left] {$i$};
\draw ($(a)!0.5!(c)$)++(.35,.3) node[left] {$j$};
\draw ($(a)!0.7!(b)$) coordinate (abb);
\draw ($(a)!0.35!(b)$) coordinate (aab);
\draw ($(a)!0.7!(c)$) coordinate (acc);
\draw ($(a)!0.35!(c)$) coordinate (aac);
\draw ($(a)!0.65!(d)$)++(-.1,0) coordinate (add);
\draw ($(a)!0.25!(d)$)++(-.1,-.1) coordinate (aad);
\foreach \x in {(aab),(abb),(aac),(acc),(aad),(add)} {\draw \x ++(0,.15) circle (.0625); \draw \x ++(0,.4) circle (.0625);
\draw \x++(-.1,.04) rectangle ++(.2,.46);}
\foreach \x in {(a)} {\fill[black] \x ++(0,-.25) circle (.125);} 
\foreach \x in {(aab),(aac),(aad)}{\fill[black] \x++(0,.15) circle (.0625);} 
\foreach \x in {(aab),(aac),(aad)}{\draw \x++(0,.2125) -- ++(0,.125);} 
\foreach \x in {(aab),(aac),(aad)}{\draw \x++(0,.4) circle (.1);} 
\end{tikzpicture}
\hspace{.25in}
\begin{tikzpicture}[scale=.7]
\draw (0,0) coordinate (a);
\draw (-2,-1) coordinate (b);
\draw (2,-1) coordinate (c);
\draw (1,2.5) coordinate (d);
\draw (b) -- (a) -- (c);
\draw (a) -- (d);
\foreach \x in {(a),(b),(c),(d)} {
\fill[white] \x++(-.18,.43) rectangle ++(.36,-.86);
\draw \x ++(0,.25) circle (.125);\draw \x ++(0,-.25) circle (.125);
\draw \x++(-.18,.43) rectangle ++(.36,-.86);}
\draw (a)++(0,-.125) -- ++(0,.25);
\draw ($(a)!0.5!(d)$)++(.1,.1) node[left] {$k$};
\draw ($(a)!0.5!(b)$)++(.2,.3) node[left] {$i$};
\draw ($(a)!0.5!(c)$)++(.35,.3) node[left] {$j$};
\draw ($(a)!0.7!(b)$) coordinate (abb);
\draw ($(a)!0.35!(b)$) coordinate (aab);
\draw ($(a)!0.7!(c)$) coordinate (acc);
\draw ($(a)!0.35!(c)$) coordinate (aac);
\draw ($(a)!0.65!(d)$)++(-.1,0) coordinate (add);
\draw ($(a)!0.25!(d)$)++(-.1,-.1) coordinate (aad);
\foreach \x in {(aab),(abb),(aac),(acc),(aad),(add)} {\draw \x ++(0,.15) circle (.0625); \draw \x ++(0,.4) circle (.0625);
\draw \x++(-.1,.04) rectangle ++(.2,.46);}
\foreach \x in {(a)} {\fill[black] \x ++(0,-.25) circle (.125);} 
\foreach \x in {(aab),(aac),(aad)}{\fill[black] \x++(0,.15) circle (.0625);} 
\foreach \x in {(aab),(aac),(aad)}{\draw \x++(0,.2125) -- ++(0,.125);} 
\foreach \x in {(aac),(aad)}{\draw \x++(0,.4) circle (.1);} 
\foreach \x in {(aab)}{\draw \x++(0,.15) circle (.1);} 
\end{tikzpicture}

\vspace{.25in}

\begin{tikzpicture}[scale=.7]
\draw (0,0) coordinate (a);
\draw (-2,-1) coordinate (b);
\draw (2,-1) coordinate (c);
\draw (1,2.5) coordinate (d);
\draw (b) -- (a) -- (c);
\draw (a) -- (d);
\foreach \x in {(a),(b),(c),(d)} {
\fill[white] \x++(-.18,.43) rectangle ++(.36,-.86);
\draw \x ++(0,.25) circle (.125);\draw \x ++(0,-.25) circle (.125);
\draw \x++(-.18,.43) rectangle ++(.36,-.86);}
\draw (a)++(0,-.125) -- ++(0,.25);
\draw ($(a)!0.5!(d)$)++(.1,.1) node[left] {$k$};
\draw ($(a)!0.5!(b)$)++(.2,.3) node[left] {$i$};
\draw ($(a)!0.5!(c)$)++(.35,.3) node[left] {$j$};
\draw ($(a)!0.7!(b)$) coordinate (abb);
\draw ($(a)!0.35!(b)$) coordinate (aab);
\draw ($(a)!0.7!(c)$) coordinate (acc);
\draw ($(a)!0.35!(c)$) coordinate (aac);
\draw ($(a)!0.65!(d)$)++(-.1,0) coordinate (add);
\draw ($(a)!0.25!(d)$)++(-.1,-.1) coordinate (aad);
\foreach \x in {(aab),(abb),(aac),(acc),(aad),(add)} {\draw \x ++(0,.15) circle (.0625); \draw \x ++(0,.4) circle (.0625);
\draw \x++(-.1,.04) rectangle ++(.2,.46);}
\foreach \x in {(a)} {\fill[black] \x ++(0,-.25) circle (.125);} 
\foreach \x in {(aab),(aac),(aad)}{\fill[black] \x++(0,.15) circle (.0625);} 
\foreach \x in {(aab),(aac),(aad)}{\draw \x++(0,.2125) -- ++(0,.125);} 
\foreach \x in {(aab),(aad)}{\draw \x++(0,.4) circle (.1);} 
\foreach \x in {(aac)}{\draw \x++(0,.15) circle (.1);} 
\end{tikzpicture}
\hspace{.25in}
\begin{tikzpicture}[scale=.7]
\draw (0,0) coordinate (a);
\draw (-2,-1) coordinate (b);
\draw (2,-1) coordinate (c);
\draw (1,2.5) coordinate (d);
\draw (b) -- (a) -- (c);
\draw (a) -- (d);
\foreach \x in {(a),(b),(c),(d)} {
\fill[white] \x++(-.18,.43) rectangle ++(.36,-.86);
\draw \x ++(0,.25) circle (.125);\draw \x ++(0,-.25) circle (.125);
\draw \x++(-.18,.43) rectangle ++(.36,-.86);}
\draw (a)++(0,-.125) -- ++(0,.25);
\draw ($(a)!0.5!(d)$)++(.1,.1) node[left] {$k$};
\draw ($(a)!0.5!(b)$)++(.2,.3) node[left] {$i$};
\draw ($(a)!0.5!(c)$)++(.35,.3) node[left] {$j$};
\draw ($(a)!0.7!(b)$) coordinate (abb);
\draw ($(a)!0.35!(b)$) coordinate (aab);
\draw ($(a)!0.7!(c)$) coordinate (acc);
\draw ($(a)!0.35!(c)$) coordinate (aac);
\draw ($(a)!0.65!(d)$)++(-.1,0) coordinate (add);
\draw ($(a)!0.25!(d)$)++(-.1,-.1) coordinate (aad);
\foreach \x in {(aab),(abb),(aac),(acc),(aad),(add)} {\draw \x ++(0,.15) circle (.0625); \draw \x ++(0,.4) circle (.0625);
\draw \x++(-.1,.04) rectangle ++(.2,.46);}
\foreach \x in {(a)} {\fill[black] \x ++(0,-.25) circle (.125);} 
\foreach \x in {(aab),(aac),(aad)}{\fill[black] \x++(0,.15) circle (.0625);} 
\foreach \x in {(aab),(aac),(aad)}{\draw \x++(0,.2125) -- ++(0,.125);} 
\foreach \x in {(aab),(aac)}{\draw \x++(0,.4) circle (.1);} 
\foreach \x in {(aad)}{\draw \x++(0,.15) circle (.1);} 
\end{tikzpicture}
\hspace{.25in}
\begin{tikzpicture}[scale=.7]
\draw (0,0) coordinate (a);
\draw (-2,-1) coordinate (b);
\draw (2,-1) coordinate (c);
\draw (1,2.5) coordinate (d);
\draw (b) -- (a) -- (c);
\draw (a) -- (d);
\foreach \x in {(a),(b),(c),(d)} {
\fill[white] \x++(-.18,.43) rectangle ++(.36,-.86);
\draw \x ++(0,.25) circle (.125);\draw \x ++(0,-.25) circle (.125);
\draw \x++(-.18,.43) rectangle ++(.36,-.86);}
\draw (a)++(0,-.125) -- ++(0,.25);
\draw ($(a)!0.5!(d)$)++(.1,.1) node[left] {$k$};
\draw ($(a)!0.5!(b)$)++(.2,.3) node[left] {$i$};
\draw ($(a)!0.5!(c)$)++(.35,.3) node[left] {$j$};
\draw ($(a)!0.7!(b)$) coordinate (abb);
\draw ($(a)!0.35!(b)$) coordinate (aab);
\draw ($(a)!0.7!(c)$) coordinate (acc);
\draw ($(a)!0.35!(c)$) coordinate (aac);
\draw ($(a)!0.65!(d)$)++(-.1,0) coordinate (add);
\draw ($(a)!0.25!(d)$)++(-.1,-.1) coordinate (aad);
\foreach \x in {(aab),(abb),(aac),(acc),(aad),(add)} {\draw \x ++(0,.15) circle (.0625); \draw \x ++(0,.4) circle (.0625);
\draw \x++(-.1,.04) rectangle ++(.2,.46);}
\foreach \x in {(a)} {\fill[black] \x ++(0,-.25) circle (.125);} 
\foreach \x in {(aab),(aac),(aad)}{\fill[black] \x++(0,.15) circle (.0625);} 
\foreach \x in {(aab),(aac),(aad)}{\draw \x++(0,.2125) -- ++(0,.125);} 
\foreach \x in {(aad)}{\draw \x++(0,.4) circle (.1);} 
\foreach \x in {(aab),(aac)}{\draw \x++(0,.15) circle (.1);} 
\end{tikzpicture}
\hspace{.25in}
\begin{tikzpicture}[scale=.7]
\draw (0,0) coordinate (a);
\draw (-2,-1) coordinate (b);
\draw (2,-1) coordinate (c);
\draw (1,2.5) coordinate (d);
\draw (b) -- (a) -- (c);
\draw (a) -- (d);
\foreach \x in {(a),(b),(c),(d)} {
\fill[white] \x++(-.18,.43) rectangle ++(.36,-.86);
\draw \x ++(0,.25) circle (.125);\draw \x ++(0,-.25) circle (.125);
\draw \x++(-.18,.43) rectangle ++(.36,-.86);}
\draw (a)++(0,-.125) -- ++(0,.25);
\draw ($(a)!0.5!(d)$)++(.1,.1) node[left] {$k$};
\draw ($(a)!0.5!(b)$)++(.2,.3) node[left] {$i$};
\draw ($(a)!0.5!(c)$)++(.35,.3) node[left] {$j$};
\draw ($(a)!0.7!(b)$) coordinate (abb);
\draw ($(a)!0.35!(b)$) coordinate (aab);
\draw ($(a)!0.7!(c)$) coordinate (acc);
\draw ($(a)!0.35!(c)$) coordinate (aac);
\draw ($(a)!0.65!(d)$)++(-.1,0) coordinate (add);
\draw ($(a)!0.25!(d)$)++(-.1,-.1) coordinate (aad);
\foreach \x in {(aab),(abb),(aac),(acc),(aad),(add)} {\draw \x ++(0,.15) circle (.0625); \draw \x ++(0,.4) circle (.0625);
\draw \x++(-.1,.04) rectangle ++(.2,.46);}
\foreach \x in {(a)} {\fill[black] \x ++(0,-.25) circle (.125);} 
\foreach \x in {(aab),(aac),(aad)}{\fill[black] \x++(0,.15) circle (.0625);} 
\foreach \x in {(aab),(aac),(aad)}{\draw \x++(0,.2125) -- ++(0,.125);} 
\foreach \x in {(aac)}{\draw \x++(0,.4) circle (.1);} 
\foreach \x in {(aab),(aad)}{\draw \x++(0,.15) circle (.1);} 
\end{tikzpicture}

\vspace{.25in}

\begin{tikzpicture}[scale=.7]
\draw (0,0) coordinate (a);
\draw (-2,-1) coordinate (b);
\draw (2,-1) coordinate (c);
\draw (1,2.5) coordinate (d);
\draw (b) -- (a) -- (c);
\draw (a) -- (d);
\foreach \x in {(a),(b),(c),(d)} {
\fill[white] \x++(-.18,.43) rectangle ++(.36,-.86);
\draw \x ++(0,.25) circle (.125);\draw \x ++(0,-.25) circle (.125);
\draw \x++(-.18,.43) rectangle ++(.36,-.86);}
\draw (a)++(0,-.125) -- ++(0,.25);
\draw ($(a)!0.5!(d)$)++(.1,.1) node[left] {$k$};
\draw ($(a)!0.5!(b)$)++(.2,.3) node[left] {$i$};
\draw ($(a)!0.5!(c)$)++(.35,.3) node[left] {$j$};
\draw ($(a)!0.7!(b)$) coordinate (abb);
\draw ($(a)!0.35!(b)$) coordinate (aab);
\draw ($(a)!0.7!(c)$) coordinate (acc);
\draw ($(a)!0.35!(c)$) coordinate (aac);
\draw ($(a)!0.65!(d)$)++(-.1,0) coordinate (add);
\draw ($(a)!0.25!(d)$)++(-.1,-.1) coordinate (aad);
\foreach \x in {(aab),(abb),(aac),(acc),(aad),(add)} {\draw \x ++(0,.15) circle (.0625); \draw \x ++(0,.4) circle (.0625);
\draw \x++(-.1,.04) rectangle ++(.2,.46);}
\foreach \x in {(a)} {\fill[black] \x ++(0,-.25) circle (.125);} 
\foreach \x in {(aab),(aac),(aad)}{\fill[black] \x++(0,.15) circle (.0625);} 
\foreach \x in {(aab),(aac),(aad)}{\draw \x++(0,.2125) -- ++(0,.125);} 
\foreach \x in {(aab)}{\draw \x++(0,.4) circle (.1);} 
\foreach \x in {(aac),(aad)}{\draw \x++(0,.15) circle (.1);} 
\end{tikzpicture}
\hspace{.25in}
\begin{tikzpicture}[scale=.7]
\draw (0,0) coordinate (a);
\draw (-2,-1) coordinate (b);
\draw (2,-1) coordinate (c);
\draw (1,2.5) coordinate (d);
\draw (b) -- (a) -- (c);
\draw (a) -- (d);
\foreach \x in {(a),(b),(c),(d)} {
\fill[white] \x++(-.18,.43) rectangle ++(.36,-.86);
\draw \x ++(0,.25) circle (.125);\draw \x ++(0,-.25) circle (.125);
\draw \x++(-.18,.43) rectangle ++(.36,-.86);}
\draw (a)++(0,-.125) -- ++(0,.25);
\draw ($(a)!0.5!(d)$)++(.1,.1) node[left] {$k$};
\draw ($(a)!0.5!(b)$)++(.2,.3) node[left] {$i$};
\draw ($(a)!0.5!(c)$)++(.35,.3) node[left] {$j$};
\draw ($(a)!0.7!(b)$) coordinate (abb);
\draw ($(a)!0.35!(b)$) coordinate (aab);
\draw ($(a)!0.7!(c)$) coordinate (acc);
\draw ($(a)!0.35!(c)$) coordinate (aac);
\draw ($(a)!0.65!(d)$)++(-.1,0) coordinate (add);
\draw ($(a)!0.25!(d)$)++(-.1,-.1) coordinate (aad);
\foreach \x in {(aab),(abb),(aac),(acc),(aad),(add)} {\draw \x ++(0,.15) circle (.0625); \draw \x ++(0,.4) circle (.0625);
\draw \x++(-.1,.04) rectangle ++(.2,.46);}
\foreach \x in {(a)} {\fill[black] \x ++(0,-.25) circle (.125);} 
\foreach \x in {(a)} {\fill[black] \x ++(0,.25) circle (.125);} 
\foreach \x in {(aab),(aac),(aad)}{\fill[black] \x++(0,.15) circle (.0625);} 
\foreach \x in {(aab),(aac),(aad)}{\fill[black] \x++(0,.4) circle (.0625);} 
\foreach \x in {(aab),(aac),(aad)}{\draw \x++(0,.2125) -- ++(0,.125);} 
\foreach \x in {(aab)}{\draw \x++(0,.4) circle (.1);} 
\foreach \x in {(aac)}{\draw \x++(0,.15) circle (.1);} 
\end{tikzpicture}
\hspace{.25in}
\begin{tikzpicture}[scale=.7]
\draw (0,0) coordinate (a);
\draw (-2,-1) coordinate (b);
\draw (2,-1) coordinate (c);
\draw (1,2.5) coordinate (d);
\draw (b) -- (a) -- (c);
\draw (a) -- (d);
\foreach \x in {(a),(b),(c),(d)} {
\fill[white] \x++(-.18,.43) rectangle ++(.36,-.86);
\draw \x ++(0,.25) circle (.125);\draw \x ++(0,-.25) circle (.125);
\draw \x++(-.18,.43) rectangle ++(.36,-.86);}
\draw (a)++(0,-.125) -- ++(0,.25);
\draw ($(a)!0.5!(d)$)++(.1,.1) node[left] {$k$};
\draw ($(a)!0.5!(b)$)++(.2,.3) node[left] {$i$};
\draw ($(a)!0.5!(c)$)++(.35,.3) node[left] {$j$};
\draw ($(a)!0.7!(b)$) coordinate (abb);
\draw ($(a)!0.35!(b)$) coordinate (aab);
\draw ($(a)!0.7!(c)$) coordinate (acc);
\draw ($(a)!0.35!(c)$) coordinate (aac);
\draw ($(a)!0.65!(d)$)++(-.1,0) coordinate (add);
\draw ($(a)!0.25!(d)$)++(-.1,-.1) coordinate (aad);
\foreach \x in {(aab),(abb),(aac),(acc),(aad),(add)} {\draw \x ++(0,.15) circle (.0625); \draw \x ++(0,.4) circle (.0625);
\draw \x++(-.1,.04) rectangle ++(.2,.46);}
\foreach \x in {(a)} {\fill[black] \x ++(0,-.25) circle (.125);} 
\foreach \x in {(a)} {\fill[black] \x ++(0,.25) circle (.125);} 
\foreach \x in {(aab),(aac),(aad)}{\fill[black] \x++(0,.15) circle (.0625);} 
\foreach \x in {(aab),(aac),(aad)}{\fill[black] \x++(0,.4) circle (.0625);} 
\foreach \x in {(aab),(aac),(aad)}{\draw \x++(0,.2125) -- ++(0,.125);} 
\foreach \x in {(aab)}{\draw \x++(0,.4) circle (.1);} 
\foreach \x in {(aad)}{\draw \x++(0,.15) circle (.1);} 
\end{tikzpicture}
\hspace{.25in}
\begin{tikzpicture}[scale=.7]
\draw (0,0) coordinate (a);
\draw (-2,-1) coordinate (b);
\draw (2,-1) coordinate (c);
\draw (1,2.5) coordinate (d);
\draw (b) -- (a) -- (c);
\draw (a) -- (d);
\foreach \x in {(a),(b),(c),(d)} {
\fill[white] \x++(-.18,.43) rectangle ++(.36,-.86);
\draw \x ++(0,.25) circle (.125);\draw \x ++(0,-.25) circle (.125);
\draw \x++(-.18,.43) rectangle ++(.36,-.86);}
\draw (a)++(0,-.125) -- ++(0,.25);
\draw ($(a)!0.5!(d)$)++(.1,.1) node[left] {$k$};
\draw ($(a)!0.5!(b)$)++(.2,.3) node[left] {$i$};
\draw ($(a)!0.5!(c)$)++(.35,.3) node[left] {$j$};
\draw ($(a)!0.7!(b)$) coordinate (abb);
\draw ($(a)!0.35!(b)$) coordinate (aab);
\draw ($(a)!0.7!(c)$) coordinate (acc);
\draw ($(a)!0.35!(c)$) coordinate (aac);
\draw ($(a)!0.65!(d)$)++(-.1,0) coordinate (add);
\draw ($(a)!0.25!(d)$)++(-.1,-.1) coordinate (aad);
\foreach \x in {(aab),(abb),(aac),(acc),(aad),(add)} {\draw \x ++(0,.15) circle (.0625); \draw \x ++(0,.4) circle (.0625);
\draw \x++(-.1,.04) rectangle ++(.2,.46);}
\foreach \x in {(a)} {\fill[black] \x ++(0,-.25) circle (.125);} 
\foreach \x in {(a)} {\fill[black] \x ++(0,.25) circle (.125);} 
\foreach \x in {(aab),(aac),(aad)}{\fill[black] \x++(0,.15) circle (.0625);} 
\foreach \x in {(aab),(aac),(aad)}{\fill[black] \x++(0,.4) circle (.0625);} 
\foreach \x in {(aab),(aac),(aad)}{\draw \x++(0,.2125) -- ++(0,.125);} 
\foreach \x in {(aac)}{\draw \x++(0,.4) circle (.1);} 
\foreach \x in {(aab)}{\draw \x++(0,.15) circle (.1);} 
\end{tikzpicture}

\vspace{.25in}

\begin{tikzpicture}[scale=.7]
\draw (0,0) coordinate (a);
\draw (-2,-1) coordinate (b);
\draw (2,-1) coordinate (c);
\draw (1,2.5) coordinate (d);
\draw (b) -- (a) -- (c);
\draw (a) -- (d);
\foreach \x in {(a),(b),(c),(d)} {
\fill[white] \x++(-.18,.43) rectangle ++(.36,-.86);
\draw \x ++(0,.25) circle (.125);\draw \x ++(0,-.25) circle (.125);
\draw \x++(-.18,.43) rectangle ++(.36,-.86);}
\draw (a)++(0,-.125) -- ++(0,.25);
\draw ($(a)!0.5!(d)$)++(.1,.1) node[left] {$k$};
\draw ($(a)!0.5!(b)$)++(.2,.3) node[left] {$i$};
\draw ($(a)!0.5!(c)$)++(.35,.3) node[left] {$j$};
\draw ($(a)!0.7!(b)$) coordinate (abb);
\draw ($(a)!0.35!(b)$) coordinate (aab);
\draw ($(a)!0.7!(c)$) coordinate (acc);
\draw ($(a)!0.35!(c)$) coordinate (aac);
\draw ($(a)!0.65!(d)$)++(-.1,0) coordinate (add);
\draw ($(a)!0.25!(d)$)++(-.1,-.1) coordinate (aad);
\foreach \x in {(aab),(abb),(aac),(acc),(aad),(add)} {\draw \x ++(0,.15) circle (.0625); \draw \x ++(0,.4) circle (.0625);
\draw \x++(-.1,.04) rectangle ++(.2,.46);}
\foreach \x in {(a)} {\fill[black] \x ++(0,-.25) circle (.125);} 
\foreach \x in {(a)} {\fill[black] \x ++(0,.25) circle (.125);} 
\foreach \x in {(aab),(aac),(aad)}{\fill[black] \x++(0,.15) circle (.0625);} 
\foreach \x in {(aab),(aac),(aad)}{\fill[black] \x++(0,.4) circle (.0625);} 
\foreach \x in {(aab),(aac),(aad)}{\draw \x++(0,.2125) -- ++(0,.125);} 
\foreach \x in {(aad)}{\draw \x++(0,.4) circle (.1);} 
\foreach \x in {(aab)}{\draw \x++(0,.15) circle (.1);} 
\end{tikzpicture}
\hspace{.25in}
\begin{tikzpicture}[scale=.7]
\draw (0,0) coordinate (a);
\draw (-2,-1) coordinate (b);
\draw (2,-1) coordinate (c);
\draw (1,2.5) coordinate (d);
\draw (b) -- (a) -- (c);
\draw (a) -- (d);
\foreach \x in {(a),(b),(c),(d)} {
\fill[white] \x++(-.18,.43) rectangle ++(.36,-.86);
\draw \x ++(0,.25) circle (.125);\draw \x ++(0,-.25) circle (.125);
\draw \x++(-.18,.43) rectangle ++(.36,-.86);}
\draw (a)++(0,-.125) -- ++(0,.25);
\draw ($(a)!0.5!(d)$)++(.1,.1) node[left] {$k$};
\draw ($(a)!0.5!(b)$)++(.2,.3) node[left] {$i$};
\draw ($(a)!0.5!(c)$)++(.35,.3) node[left] {$j$};
\draw ($(a)!0.7!(b)$) coordinate (abb);
\draw ($(a)!0.35!(b)$) coordinate (aab);
\draw ($(a)!0.7!(c)$) coordinate (acc);
\draw ($(a)!0.35!(c)$) coordinate (aac);
\draw ($(a)!0.65!(d)$)++(-.1,0) coordinate (add);
\draw ($(a)!0.25!(d)$)++(-.1,-.1) coordinate (aad);
\foreach \x in {(aab),(abb),(aac),(acc),(aad),(add)} {\draw \x ++(0,.15) circle (.0625); \draw \x ++(0,.4) circle (.0625);
\draw \x++(-.1,.04) rectangle ++(.2,.46);}
\foreach \x in {(a)} {\fill[black] \x ++(0,-.25) circle (.125);} 
\foreach \x in {(a)} {\fill[black] \x ++(0,.25) circle (.125);} 
\foreach \x in {(aab),(aac),(aad)}{\fill[black] \x++(0,.15) circle (.0625);} 
\foreach \x in {(aab),(aac),(aad)}{\fill[black] \x++(0,.4) circle (.0625);} 
\foreach \x in {(aab),(aac),(aad)}{\draw \x++(0,.2125) -- ++(0,.125);} 
\foreach \x in {(aac)}{\draw \x++(0,.4) circle (.1);} 
\foreach \x in {(aad)}{\draw \x++(0,.15) circle (.1);} 
\end{tikzpicture}
\hspace{.25in}
\begin{tikzpicture}[scale=.7]
\draw (0,0) coordinate (a);
\draw (-2,-1) coordinate (b);
\draw (2,-1) coordinate (c);
\draw (1,2.5) coordinate (d);
\draw (b) -- (a) -- (c);
\draw (a) -- (d);
\foreach \x in {(a),(b),(c),(d)} {
\fill[white] \x++(-.18,.43) rectangle ++(.36,-.86);
\draw \x ++(0,.25) circle (.125);\draw \x ++(0,-.25) circle (.125);
\draw \x++(-.18,.43) rectangle ++(.36,-.86);}
\draw (a)++(0,-.125) -- ++(0,.25);
\draw ($(a)!0.5!(d)$)++(.1,.1) node[left] {$k$};
\draw ($(a)!0.5!(b)$)++(.2,.3) node[left] {$i$};
\draw ($(a)!0.5!(c)$)++(.35,.3) node[left] {$j$};
\draw ($(a)!0.7!(b)$) coordinate (abb);
\draw ($(a)!0.35!(b)$) coordinate (aab);
\draw ($(a)!0.7!(c)$) coordinate (acc);
\draw ($(a)!0.35!(c)$) coordinate (aac);
\draw ($(a)!0.65!(d)$)++(-.1,0) coordinate (add);
\draw ($(a)!0.25!(d)$)++(-.1,-.1) coordinate (aad);
\foreach \x in {(aab),(abb),(aac),(acc),(aad),(add)} {\draw \x ++(0,.15) circle (.0625); \draw \x ++(0,.4) circle (.0625);
\draw \x++(-.1,.04) rectangle ++(.2,.46);}
\foreach \x in {(a)} {\fill[black] \x ++(0,-.25) circle (.125);} 
\foreach \x in {(a)} {\fill[black] \x ++(0,.25) circle (.125);} 
\foreach \x in {(aab),(aac),(aad)}{\fill[black] \x++(0,.15) circle (.0625);} 
\foreach \x in {(aab),(aac),(aad)}{\fill[black] \x++(0,.4) circle (.0625);} 
\foreach \x in {(aab),(aac),(aad)}{\draw \x++(0,.2125) -- ++(0,.125);} 
\foreach \x in {(aad)}{\draw \x++(0,.4) circle (.1);} 
\foreach \x in {(aac)}{\draw \x++(0,.15) circle (.1);} 
\end{tikzpicture}
\hspace{.25in}
\begin{tikzpicture}[scale=.7]
\draw (0,0) coordinate (a);
\draw (-2,-1) coordinate (b);
\draw (2,-1) coordinate (c);
\draw (1,2.5) coordinate (d);
\draw (b) -- (a) -- (c);
\draw (a) -- (d);
\foreach \x in {(a),(b),(c),(d)} {
\fill[white] \x++(-.18,.43) rectangle ++(.36,-.86);
\draw \x ++(0,.25) circle (.125);\draw \x ++(0,-.25) circle (.125);
\draw \x++(-.18,.43) rectangle ++(.36,-.86);}
\draw (a)++(0,-.125) -- ++(0,.25);
\draw ($(a)!0.5!(d)$)++(.1,.1) node[left] {$k$};
\draw ($(a)!0.5!(b)$)++(.2,.3) node[left] {$i$};
\draw ($(a)!0.5!(c)$)++(.35,.3) node[left] {$j$};
\draw ($(a)!0.7!(b)$) coordinate (abb);
\draw ($(a)!0.35!(b)$) coordinate (aab);
\draw ($(a)!0.7!(c)$) coordinate (acc);
\draw ($(a)!0.35!(c)$) coordinate (aac);
\draw ($(a)!0.65!(d)$)++(-.1,0) coordinate (add);
\draw ($(a)!0.25!(d)$)++(-.1,-.1) coordinate (aad);
\foreach \x in {(aab),(abb),(aac),(acc),(aad),(add)} {\draw \x ++(0,.15) circle (.0625); \draw \x ++(0,.4) circle (.0625);
\draw \x++(-.1,.04) rectangle ++(.2,.46);}
\foreach \x in {(a)} {\fill[black] \x ++(0,-.25) circle (.125);} 
\foreach \x in {(a)} {\fill[black] \x ++(0,.25) circle (.125);} 
\foreach \x in {(aab),(aac),(aad)}{\fill[black] \x++(0,.15) circle (.0625);} 
\foreach \x in {(aab),(aac),(aad)}{\fill[black] \x++(0,.4) circle (.0625);} 
\foreach \x in {(aab),(aac),(aad)}{\draw \x++(0,.2125) -- ++(0,.125);} 
\foreach \x in {(aab),(aac)}{\draw \x++(0,.4) circle (.1);} 
\foreach \x in {(aad)}{\draw \x++(0,.15) circle (.1);} 
\end{tikzpicture}

\vspace{.25in}

\begin{tikzpicture}[scale=.7]
\draw (0,0) coordinate (a);
\draw (-2,-1) coordinate (b);
\draw (2,-1) coordinate (c);
\draw (1,2.5) coordinate (d);
\draw (b) -- (a) -- (c);
\draw (a) -- (d);
\foreach \x in {(a),(b),(c),(d)} {
\fill[white] \x++(-.18,.43) rectangle ++(.36,-.86);
\draw \x ++(0,.25) circle (.125);\draw \x ++(0,-.25) circle (.125);
\draw \x++(-.18,.43) rectangle ++(.36,-.86);}
\draw (a)++(0,-.125) -- ++(0,.25);
\draw ($(a)!0.5!(d)$)++(.1,.1) node[left] {$k$};
\draw ($(a)!0.5!(b)$)++(.2,.3) node[left] {$i$};
\draw ($(a)!0.5!(c)$)++(.35,.3) node[left] {$j$};
\draw ($(a)!0.7!(b)$) coordinate (abb);
\draw ($(a)!0.35!(b)$) coordinate (aab);
\draw ($(a)!0.7!(c)$) coordinate (acc);
\draw ($(a)!0.35!(c)$) coordinate (aac);
\draw ($(a)!0.65!(d)$)++(-.1,0) coordinate (add);
\draw ($(a)!0.25!(d)$)++(-.1,-.1) coordinate (aad);
\foreach \x in {(aab),(abb),(aac),(acc),(aad),(add)} {\draw \x ++(0,.15) circle (.0625); \draw \x ++(0,.4) circle (.0625);
\draw \x++(-.1,.04) rectangle ++(.2,.46);}
\foreach \x in {(a)} {\fill[black] \x ++(0,-.25) circle (.125);} 
\foreach \x in {(a)} {\fill[black] \x ++(0,.25) circle (.125);} 
\foreach \x in {(aab),(aac),(aad)}{\fill[black] \x++(0,.15) circle (.0625);} 
\foreach \x in {(aab),(aac),(aad)}{\fill[black] \x++(0,.4) circle (.0625);} 
\foreach \x in {(aab),(aac),(aad)}{\draw \x++(0,.2125) -- ++(0,.125);} 
\foreach \x in {(aab),(aad)}{\draw \x++(0,.4) circle (.1);} 
\foreach \x in {(aac)}{\draw \x++(0,.15) circle (.1);} 
\end{tikzpicture}
\hspace{.25in}
\begin{tikzpicture}[scale=.7]
\draw (0,0) coordinate (a);
\draw (-2,-1) coordinate (b);
\draw (2,-1) coordinate (c);
\draw (1,2.5) coordinate (d);
\draw (b) -- (a) -- (c);
\draw (a) -- (d);
\foreach \x in {(a),(b),(c),(d)} {
\fill[white] \x++(-.18,.43) rectangle ++(.36,-.86);
\draw \x ++(0,.25) circle (.125);\draw \x ++(0,-.25) circle (.125);
\draw \x++(-.18,.43) rectangle ++(.36,-.86);}
\draw (a)++(0,-.125) -- ++(0,.25);
\draw ($(a)!0.5!(d)$)++(.1,.1) node[left] {$k$};
\draw ($(a)!0.5!(b)$)++(.2,.3) node[left] {$i$};
\draw ($(a)!0.5!(c)$)++(.35,.3) node[left] {$j$};
\draw ($(a)!0.7!(b)$) coordinate (abb);
\draw ($(a)!0.35!(b)$) coordinate (aab);
\draw ($(a)!0.7!(c)$) coordinate (acc);
\draw ($(a)!0.35!(c)$) coordinate (aac);
\draw ($(a)!0.65!(d)$)++(-.1,0) coordinate (add);
\draw ($(a)!0.25!(d)$)++(-.1,-.1) coordinate (aad);
\foreach \x in {(aab),(abb),(aac),(acc),(aad),(add)} {\draw \x ++(0,.15) circle (.0625); \draw \x ++(0,.4) circle (.0625);
\draw \x++(-.1,.04) rectangle ++(.2,.46);}
\foreach \x in {(a)} {\fill[black] \x ++(0,-.25) circle (.125);} 
\foreach \x in {(a)} {\fill[black] \x ++(0,.25) circle (.125);} 
\foreach \x in {(aab),(aac),(aad)}{\fill[black] \x++(0,.15) circle (.0625);} 
\foreach \x in {(aab),(aac),(aad)}{\fill[black] \x++(0,.4) circle (.0625);} 
\foreach \x in {(aab),(aac),(aad)}{\draw \x++(0,.2125) -- ++(0,.125);} 
\foreach \x in {(aac),(aad)}{\draw \x++(0,.4) circle (.1);} 
\foreach \x in {(aab)}{\draw \x++(0,.15) circle (.1);} 
\end{tikzpicture}
\hspace{.25in}
\begin{tikzpicture}[scale=.7]
\draw (0,0) coordinate (a);
\draw (-2,-1) coordinate (b);
\draw (2,-1) coordinate (c);
\draw (1,2.5) coordinate (d);
\draw (b) -- (a) -- (c);
\draw (a) -- (d);
\foreach \x in {(a),(b),(c),(d)} {
\fill[white] \x++(-.18,.43) rectangle ++(.36,-.86);
\draw \x ++(0,.25) circle (.125);\draw \x ++(0,-.25) circle (.125);
\draw \x++(-.18,.43) rectangle ++(.36,-.86);}
\draw (a)++(0,-.125) -- ++(0,.25);
\draw ($(a)!0.5!(d)$)++(.1,.1) node[left] {$k$};
\draw ($(a)!0.5!(b)$)++(.2,.3) node[left] {$i$};
\draw ($(a)!0.5!(c)$)++(.35,.3) node[left] {$j$};
\draw ($(a)!0.7!(b)$) coordinate (abb);
\draw ($(a)!0.35!(b)$) coordinate (aab);
\draw ($(a)!0.7!(c)$) coordinate (acc);
\draw ($(a)!0.35!(c)$) coordinate (aac);
\draw ($(a)!0.65!(d)$)++(-.1,0) coordinate (add);
\draw ($(a)!0.25!(d)$)++(-.1,-.1) coordinate (aad);
\foreach \x in {(aab),(abb),(aac),(acc),(aad),(add)} {\draw \x ++(0,.15) circle (.0625); \draw \x ++(0,.4) circle (.0625);
\draw \x++(-.1,.04) rectangle ++(.2,.46);}
\foreach \x in {(a)} {\fill[black] \x ++(0,-.25) circle (.125);} 
\foreach \x in {(a)} {\fill[black] \x ++(0,.25) circle (.125);} 
\foreach \x in {(aab),(aac),(aad)}{\fill[black] \x++(0,.15) circle (.0625);} 
\foreach \x in {(aab),(aac),(aad)}{\fill[black] \x++(0,.4) circle (.0625);} 
\foreach \x in {(aab),(aac),(aad)}{\draw \x++(0,.2125) -- ++(0,.125);} 
\foreach \x in {(aab)}{\draw \x++(0,.4) circle (.1);} 
\foreach \x in {(aac),(aad)}{\draw \x++(0,.15) circle (.1);} 
\end{tikzpicture}
\hspace{.25in}
\begin{tikzpicture}[scale=.7]
\draw (0,0) coordinate (a);
\draw (-2,-1) coordinate (b);
\draw (2,-1) coordinate (c);
\draw (1,2.5) coordinate (d);
\draw (b) -- (a) -- (c);
\draw (a) -- (d);
\foreach \x in {(a),(b),(c),(d)} {
\fill[white] \x++(-.18,.43) rectangle ++(.36,-.86);
\draw \x ++(0,.25) circle (.125);\draw \x ++(0,-.25) circle (.125);
\draw \x++(-.18,.43) rectangle ++(.36,-.86);}
\draw (a)++(0,-.125) -- ++(0,.25);
\draw ($(a)!0.5!(d)$)++(.1,.1) node[left] {$k$};
\draw ($(a)!0.5!(b)$)++(.2,.3) node[left] {$i$};
\draw ($(a)!0.5!(c)$)++(.35,.3) node[left] {$j$};
\draw ($(a)!0.7!(b)$) coordinate (abb);
\draw ($(a)!0.35!(b)$) coordinate (aab);
\draw ($(a)!0.7!(c)$) coordinate (acc);
\draw ($(a)!0.35!(c)$) coordinate (aac);
\draw ($(a)!0.65!(d)$)++(-.1,0) coordinate (add);
\draw ($(a)!0.25!(d)$)++(-.1,-.1) coordinate (aad);
\foreach \x in {(aab),(abb),(aac),(acc),(aad),(add)} {\draw \x ++(0,.15) circle (.0625); \draw \x ++(0,.4) circle (.0625);
\draw \x++(-.1,.04) rectangle ++(.2,.46);}
\foreach \x in {(a)} {\fill[black] \x ++(0,-.25) circle (.125);} 
\foreach \x in {(a)} {\fill[black] \x ++(0,.25) circle (.125);} 
\foreach \x in {(aab),(aac),(aad)}{\fill[black] \x++(0,.15) circle (.0625);} 
\foreach \x in {(aab),(aac),(aad)}{\fill[black] \x++(0,.4) circle (.0625);} 
\foreach \x in {(aab),(aac),(aad)}{\draw \x++(0,.2125) -- ++(0,.125);} 
\foreach \x in {(aac)}{\draw \x++(0,.4) circle (.1);} 
\foreach \x in {(aab),(aad)}{\draw \x++(0,.15) circle (.1);} 
\end{tikzpicture}

\vspace{.25in}

\begin{tikzpicture}[scale=.7]
\draw (0,0) coordinate (a);
\draw (-2,-1) coordinate (b);
\draw (2,-1) coordinate (c);
\draw (1,2.5) coordinate (d);
\draw (b) -- (a) -- (c);
\draw (a) -- (d);
\foreach \x in {(a),(b),(c),(d)} {
\fill[white] \x++(-.18,.43) rectangle ++(.36,-.86);
\draw \x ++(0,.25) circle (.125);\draw \x ++(0,-.25) circle (.125);
\draw \x++(-.18,.43) rectangle ++(.36,-.86);}
\draw (a)++(0,-.125) -- ++(0,.25);
\draw ($(a)!0.5!(d)$)++(.1,.1) node[left] {$k$};
\draw ($(a)!0.5!(b)$)++(.2,.3) node[left] {$i$};
\draw ($(a)!0.5!(c)$)++(.35,.3) node[left] {$j$};
\draw ($(a)!0.7!(b)$) coordinate (abb);
\draw ($(a)!0.35!(b)$) coordinate (aab);
\draw ($(a)!0.7!(c)$) coordinate (acc);
\draw ($(a)!0.35!(c)$) coordinate (aac);
\draw ($(a)!0.65!(d)$)++(-.1,0) coordinate (add);
\draw ($(a)!0.25!(d)$)++(-.1,-.1) coordinate (aad);
\foreach \x in {(aab),(abb),(aac),(acc),(aad),(add)} {\draw \x ++(0,.15) circle (.0625); \draw \x ++(0,.4) circle (.0625);
\draw \x++(-.1,.04) rectangle ++(.2,.46);}
\foreach \x in {(a)} {\fill[black] \x ++(0,-.25) circle (.125);} 
\foreach \x in {(a)} {\fill[black] \x ++(0,.25) circle (.125);} 
\foreach \x in {(aab),(aac),(aad)}{\fill[black] \x++(0,.15) circle (.0625);} 
\foreach \x in {(aab),(aac),(aad)}{\fill[black] \x++(0,.4) circle (.0625);} 
\foreach \x in {(aab),(aac),(aad)}{\draw \x++(0,.2125) -- ++(0,.125);} 
\foreach \x in {(aad)}{\draw \x++(0,.4) circle (.1);} 
\foreach \x in {(aab),(aac)}{\draw \x++(0,.15) circle (.1);} 
\end{tikzpicture}

%% file: circular.tex
\begin{tikzpicture}[scale=.95]
\draw (0,0) circle (1cm);
\fill[white] (-.18,-.57) rectangle (.18,-1.43);
\draw (-.18,-.57) rectangle (.18,-1.43);
\draw (0,-.75) circle (.125cm);
\draw (0,-1.25) circle (.125cm);
\draw (0,-.875) -- (0,-1.125);
\foreach \x in {(-.75,-1.25), (-.75,-1), (.75,-1.25), (.75,-1)} {\draw \x circle (.0625cm);}
\draw (-.85,-1.35) rectangle (-.65,-.9); \draw (.65,-1.35) rectangle (.85,-.9); 
\draw (0,1) node[above] {$n$};
\end{tikzpicture}
\hspace{.1in}
\begin{tikzpicture}[scale=.95]
\draw (0,0) circle (1cm);
\fill[white] (-.18,-.57) rectangle (.18,-1.43);
\draw (-.18,-.57) rectangle (.18,-1.43);
\fill[black] (0,-.75) circle (.125cm); \draw (0,-.75) circle (.125cm);
\draw (0,-1.25) circle (.125cm);
\draw (0,-.875) -- (0,-1.125);
\foreach \x in {(-.75,-1.25), (.75,-1.25)} {\draw \x circle (.0625cm);}
\foreach \x in {(-.75,-1), (.75,-1)} {\fill[black] \x circle (.0625cm); \draw \x circle (.0625cm);}
\draw (-.85,-1.35) rectangle (-.65,-.9); \draw (.65,-1.35) rectangle (.85,-.9); 
\draw (0,1) node[above] {$n$};
\end{tikzpicture}
\hspace{.1in}
\begin{tikzpicture}[scale=.95]
\draw (0,0) circle (1cm);
\fill[white] (-.18,-.57) rectangle (.18,-1.43);
\draw (-.18,-.57) rectangle (.18,-1.43);
\draw (0,-.75) circle (.125cm);
\fill[black] (0,-1.25) circle (.125cm); \draw (0,-1.25) circle (.125cm);
\draw (0,-.875) -- (0,-1.125);
\foreach \x in {(-.75,-1.25), (.75,-1.25)} {\fill[black] \x circle (.0625cm); \draw \x circle (.0625cm); \draw \x -- ++(0,.1875);}
\foreach \x in {(-.75,-1), (.75,-1)} {\draw \x circle (.0625cm); \draw \x circle (.1cm);}
\draw (-.85,-1.35) rectangle (-.65,-.9); \draw (.65,-1.35) rectangle (.85,-.9); 
\draw (0,1) node[above] {$n$};
\end{tikzpicture}
\hspace{.1in}
\begin{tikzpicture}[scale=.95]
\draw (0,0) circle (1cm);
\fill[white] (-.18,-.57) rectangle (.18,-1.43);
\draw (-.18,-.57) rectangle (.18,-1.43);
\draw (0,-.75) circle (.125cm);
\fill[black] (0,-1.25) circle (.125cm); \draw (0,-1.25) circle (.125cm);
\draw (0,-.875) -- (0,-1.125);
\foreach \x in {(-.75,-1.25), (.75,-1.25)} {\fill[black] \x circle (.0625cm); \draw \x circle (.0625cm); \draw \x -- ++(0,.1875);}
\foreach \x in {(-.75,-1), (.75,-1)} {\draw \x circle (.0625cm);}
\foreach \x in {(-.75,-1),(.75,-1.25)} {\draw \x circle (.1cm);}
\draw (-.85,-1.35) rectangle (-.65,-.9); \draw (.65,-1.35) rectangle (.85,-.9); 
\draw (0,1) node[above] {$n$};
\end{tikzpicture}
\hspace{.1in}
\begin{tikzpicture}[scale=.95]
\draw (0,0) circle (1cm);
\fill[white] (-.18,-.57) rectangle (.18,-1.43);
\draw (-.18,-.57) rectangle (.18,-1.43);
\draw (0,-.75) circle (.125cm);
\fill[black] (0,-1.25) circle (.125cm); \draw (0,-1.25) circle (.125cm);
\draw (0,-.875) -- (0,-1.125);
\foreach \x in {(-.75,-1.25), (.75,-1.25)} {\fill[black] \x circle (.0625cm); \draw \x circle (.0625cm); \draw \x -- ++(0,.1875);}
\foreach \x in {(-.75,-1), (.75,-1)} {\draw \x circle (.0625cm);}
\foreach \x in {(.75,-1),(-.75,-1.25)} {\draw \x circle (.1cm);}
\draw (-.85,-1.35) rectangle (-.65,-.9); \draw (.65,-1.35) rectangle (.85,-.9); 
\draw (0,1) node[above] {$n$};
\end{tikzpicture}
\hspace{.1in}
\begin{tikzpicture}[scale=.95]
\draw (0,0) circle (1cm);
\fill[white] (-.18,-.57) rectangle (.18,-1.43);
\draw (-.18,-.57) rectangle (.18,-1.43);
\fill[black] (0,-.75) circle (.125cm);\draw (0,-.75) circle (.125cm);
\fill[black] (0,-1.25) circle (.125cm); \draw (0,-1.25) circle (.125cm);
\draw (0,-.875) -- (0,-1.125);
\foreach \x in {(-.75,-1.25), (.75,-1.25)} {\fill[black] \x circle (.0625cm); \draw \x circle (.0625cm); \draw \x -- ++(0,.1875);}
\foreach \x in {(-.75,-1), (.75,-1)} {\fill[black] \x circle (.0625cm);\draw \x circle (.0625cm);}
\foreach \x in {(-.75,-1),(.75,-1.25)} {\draw \x circle (.1cm);}
\draw (-.85,-1.35) rectangle (-.65,-.9); \draw (.65,-1.35) rectangle (.85,-.9); 
\draw (0,1) node[above] {$n$};
\end{tikzpicture}
\hspace{.1in}
\begin{tikzpicture}[scale=.95]
\draw (0,0) circle (1cm);
\fill[white] (-.18,-.57) rectangle (.18,-1.43);
\draw (-.18,-.57) rectangle (.18,-1.43);
\fill[black] (0,-.75) circle (.125cm);\draw (0,-.75) circle (.125cm);
\fill[black] (0,-1.25) circle (.125cm); \draw (0,-1.25) circle (.125cm);
\draw (0,-.875) -- (0,-1.125);
\foreach \x in {(-.75,-1.25), (.75,-1.25)} {\fill[black] \x circle (.0625cm); \draw \x circle (.0625cm); \draw \x -- ++(0,.1875);}
\foreach \x in {(-.75,-1), (.75,-1)} {\fill[black] \x circle (.0625cm);\draw \x circle (.0625cm);}
\foreach \x in {(.75,-1),(-.75,-1.25)} {\draw \x circle (.1cm);}
\draw (-.85,-1.35) rectangle (-.65,-.9); \draw (.65,-1.35) rectangle (.85,-.9); 
\draw (0,1) node[above] {$n$};
\end{tikzpicture}

%% file: twobranch.tex
\begin{tikzpicture}[scale=.75]
\draw (-1,0) coordinate (a);
\draw (1,0) coordinate (b);
\draw (2.75,1.75) coordinate (c);
\draw (2.75,-1.75) coordinate (d);
\draw (-2.75,1.75) coordinate (e);
\draw (-2.75,-1.75) coordinate (f);
\draw (a) -- (b);
\draw (c) -- (b) -- (d);
\draw ($(b)!0.5!(c)$)++(-.15,.2) node {$i_2$}; \draw ($(b)!0.5!(d)$)++(.2,.2) node {$j_2$};
\draw ($(a)!0.5!(e)$)++(.15,.2) node {$i_1$}; \draw ($(a)!0.5!(f)$)++(-.2,.2) node {$j_1$};
\draw (e) -- (a) -- (f);
\draw ($(a)!0.5!(b)$) node[above] {$k$};
\foreach \x in {(a),(b),(c),(d),(e),(f)} {
\fill[white] \x++(-.18,-.43) rectangle ++(.36,.86);
\draw \x ++(0,.25) circle (.125);\draw \x ++(0,-.25) circle (.125);
\draw \x++(-.18,-.43) rectangle ++(.36,.86);}
\draw (a)++(0,-.125) -- ++(0,.25);
\draw (b)++(0,-.125) -- ++(0,.25);
\draw ($(a)!0.33!(b)$) coordinate (aab);
\draw ($(b)!0.33!(a)$) coordinate (abb);
\draw ($(b)!0.25!(c)$)++(0,.1) coordinate (bbc);
\draw ($(b)!0.65!(c)$)++(-.1,0) coordinate (bcc);
\draw ($(b)!0.35!(d)$)++(.1,0) coordinate (bbd);
\draw ($(b)!0.75!(d)$)++(.1,0) coordinate (bdd);
\draw ($(a)!0.25!(e)$)++(0,.1) coordinate (aae);
\draw ($(a)!0.65!(e)$)++(.1,0) coordinate (aee);
\draw ($(a)!0.35!(f)$)++(-.1,0) coordinate (aaf);
\draw ($(a)!0.75!(f)$)++(-.1,0) coordinate (aff);
\foreach \x in {(aab),(abb),(bbc),(bcc),(bbd),(bdd),(aae),(aee),(aaf),(aff)} {\draw \x ++(0,.15) circle (.0625); \draw \x ++(0,.4) circle (.0625);
\draw \x++(-.1,.04) rectangle ++(.2,.46);}
\foreach \x in {(abb),(bbc),(bbd)}{\fill[black] \x++(0,.4) circle (.0625);} 
\foreach \x in {(b)} {\fill[black] \x ++(0,.25) circle (.125);} 
\end{tikzpicture}
\hspace{.2in}
\begin{tikzpicture}[scale=.75]
\draw (-1,0) coordinate (a);
\draw (1,0) coordinate (b);
\draw (2.75,1.75) coordinate (c);
\draw (2.75,-1.75) coordinate (d);
\draw (-2.75,1.75) coordinate (e);
\draw (-2.75,-1.75) coordinate (f);
\draw (a) -- (b);
\draw (c) -- (b) -- (d);
\draw ($(b)!0.5!(c)$)++(-.15,.2) node {$i_2$}; \draw ($(b)!0.5!(d)$)++(.2,.2) node {$j_2$};
\draw ($(a)!0.5!(e)$)++(.15,.2) node {$i_1$}; \draw ($(a)!0.5!(f)$)++(-.2,.2) node {$j_1$};
\draw (e) -- (a) -- (f);
\draw ($(a)!0.5!(b)$) node[above] {$k$};
\foreach \x in {(a),(b),(c),(d),(e),(f)} {
\fill[white] \x++(-.18,-.43) rectangle ++(.36,.86);
\draw \x ++(0,.25) circle (.125);\draw \x ++(0,-.25) circle (.125);
\draw \x++(-.18,-.43) rectangle ++(.36,.86);}
\draw (a)++(0,-.125) -- ++(0,.25);
\draw (b)++(0,-.125) -- ++(0,.25);
\draw ($(a)!0.33!(b)$) coordinate (aab);
\draw ($(b)!0.33!(a)$) coordinate (abb);
\draw ($(b)!0.25!(c)$)++(0,.1) coordinate (bbc);
\draw ($(b)!0.65!(c)$)++(-.1,0) coordinate (bcc);
\draw ($(b)!0.35!(d)$)++(.1,0) coordinate (bbd);
\draw ($(b)!0.75!(d)$)++(.1,0) coordinate (bdd);
\draw ($(a)!0.25!(e)$)++(0,.1) coordinate (aae);
\draw ($(a)!0.65!(e)$)++(.1,0) coordinate (aee);
\draw ($(a)!0.35!(f)$)++(-.1,0) coordinate (aaf);
\draw ($(a)!0.75!(f)$)++(-.1,0) coordinate (aff);
\foreach \x in {(aab),(abb),(bbc),(bcc),(bbd),(bdd),(aae),(aee),(aaf),(aff)} {\draw \x ++(0,.15) circle (.0625); \draw \x ++(0,.4) circle (.0625);
\draw \x++(-.1,.04) rectangle ++(.2,.46);}
\foreach \x in {(aab),(abb),(bbc),(bbd),(aae),(aaf)}{\fill[black] \x++(0,.15) circle (.0625);} 
\foreach \x in {(a),(b)} {\fill[black] \x ++(0,-.25) circle (.125);} 
\foreach \x in {(aab),(abb),(bbc),(bbd),(aae),(aaf)}{\draw \x++(0,.2125) -- ++(0,.125);} 
\foreach \x in {(aab),(abb),(aae)}{\draw \x++(0,.4) circle (.1);} 
\foreach \x in {(aaf),(bbc),(bbd)}{\draw \x++(0,.15) circle (.1);} 
\end{tikzpicture}
\hspace{.2in}
\begin{tikzpicture}[scale=.75]
\draw (-1,0) coordinate (a);
\draw (1,0) coordinate (b);
\draw (2.75,1.75) coordinate (c);
\draw (2.75,-1.75) coordinate (d);
\draw (-2.75,1.75) coordinate (e);
\draw (-2.75,-1.75) coordinate (f);
\draw (a) -- (b);
\draw (c) -- (b) -- (d);
\draw ($(b)!0.5!(c)$)++(-.15,.2) node {$i_2$}; \draw ($(b)!0.5!(d)$)++(.2,.2) node {$j_2$};
\draw ($(a)!0.5!(e)$)++(.15,.2) node {$i_1$}; \draw ($(a)!0.5!(f)$)++(-.2,.2) node {$j_1$};
\draw (e) -- (a) -- (f);
\draw ($(a)!0.5!(b)$) node[above] {$k$};
\foreach \x in {(a),(b),(c),(d),(e),(f)} {
\fill[white] \x++(-.18,-.43) rectangle ++(.36,.86);
\draw \x ++(0,.25) circle (.125);\draw \x ++(0,-.25) circle (.125);
\draw \x++(-.18,-.43) rectangle ++(.36,.86);}
\draw (a)++(0,-.125) -- ++(0,.25);
\draw (b)++(0,-.125) -- ++(0,.25);
\draw ($(a)!0.33!(b)$) coordinate (aab);
\draw ($(b)!0.33!(a)$) coordinate (abb);
\draw ($(b)!0.25!(c)$)++(0,.1) coordinate (bbc);
\draw ($(b)!0.65!(c)$)++(-.1,0) coordinate (bcc);
\draw ($(b)!0.35!(d)$)++(.1,0) coordinate (bbd);
\draw ($(b)!0.75!(d)$)++(.1,0) coordinate (bdd);
\draw ($(a)!0.25!(e)$)++(0,.1) coordinate (aae);
\draw ($(a)!0.65!(e)$)++(.1,0) coordinate (aee);
\draw ($(a)!0.35!(f)$)++(-.1,0) coordinate (aaf);
\draw ($(a)!0.75!(f)$)++(-.1,0) coordinate (aff);
\foreach \x in {(bbc),(bcc),(bbd),(bdd),(aae),(aee),(aaf),(aff)} {\draw \x ++(0,.15) circle (.0625); \draw \x ++(0,.4) circle (.0625);
\draw \x++(-.1,.04) rectangle ++(.2,.46);}
\foreach \x in {(bbc),(bbd),(aae),(aaf)}{\fill[black] \x++(0,.15) circle (.0625);} 
\foreach \x in {(a),(b)} {\fill[black] \x ++(0,-.25) circle (.125);} 
\foreach \x in {(bbc),(bbd),(aae),(aaf)}{\draw \x++(0,.2125) -- ++(0,.125);} 
\foreach \x in {(aae)}{\draw \x++(0,.4) circle (.1);} 
\foreach \x in {(aaf),(bbc),(bbd)}{\draw \x++(0,.15) circle (.1);} 
\draw ($(a)!0.5!(b)$)++(0,-.15) circle (.0625);
\draw ($(a)!0.5!(b)$)++(0,-.4) circle (.0625);
\draw ($(a)!0.5!(b)$)++(-.1,-.04) rectangle ++(.2,-.46);
\fill[black] ($(a)!0.5!(b)$)++(0,-.4) circle (.0625);
\end{tikzpicture}